\documentclass[12pt,reqno]{amsart}
\usepackage[margin=1in]{geometry}
\usepackage{amsmath,amssymb,amsthm,graphicx,amsxtra, setspace}
\usepackage[utf8]{inputenc}
\usepackage{mathrsfs}
\usepackage{hyperref}
\usepackage{upgreek}
\usepackage{mathtools}
\usepackage[mathcal]{euscript}
\usepackage{xcolor}
\allowdisplaybreaks

\usepackage[pagewise]{lineno}

\usepackage{graphicx,eurosym}
\usepackage{hyperref}
\usepackage{mathtools}

\usepackage[cyr]{aeguill}

\colorlet{darkblue}{blue!50!black}

\hypersetup{
	colorlinks,%
	citecolor=blue,%
	filecolor=red,%
	linkcolor=red,%
	urlcolor=blue,%
	pdfnewwindow=true,%
	pdfstartview={FitH}
}

\newtheorem{theorem}{Theorem}[section]
\newtheorem{lemma}[theorem]{Lemma}
\newtheorem{proposition}[theorem]{Proposition}
\newtheorem{assumption}[theorem]{Assumption}
\newtheorem{corollary}[theorem]{Corollary}
\newtheorem{definition}[theorem]{Definition}

\newtheorem{remark}[theorem]{Remark}

\allowdisplaybreaks

\let\originalleft\left
\let\originalright\right
\renewcommand{\left}{\mathopen{}\mathclose\bgroup\originalleft}
\renewcommand{\right}{\aftergroup\egroup\originalright}

\renewcommand{\d}{\/\mathrm{d}\/}

\def\w{\textbf{W}^{\varepsilon}_{{\theta}^{\varepsilon}}}

\def\L{\mathbb{L}}
\def\A{\mathrm{A}}
\def\I{\mathrm{I}}

\def\C{\mathrm{C}}
\def\f{\boldsymbol{f}}
\def\J{\mathrm{J}}
\def\B{\mathrm{B}}
\def\D{\mathrm{D}}
\def\y{\boldsymbol{y}}

\def\E{\mathbb{E}}

\def\x{\boldsymbol{x}}

\def\p{\boldsymbol{p}}
\def\h{\boldsymbol{h}}
\def\z{\Upsilon}
\def\v{\boldsymbol{v}}
\def\V{\mathbb{v}}
\def\w{\boldsymbol{w}}
\def\W{\mathrm{W}}
\def\G{\mathrm{G}}

\def\N{\mathbb{N}}
\def\r{\mathrm{r}}

\def\V{\mathbb{V}}
\def\wi{\widetilde}

\def\u{\mathrm{U}}
\def\P{\mathrm{P}}
\def\u{\boldsymbol{u}}
\def\H{\mathbb{H}}

\newcommand{\R}{\mathbb{R}}

\renewcommand{\d}{\/\mathrm{d}\/}

\newcommand{\Addresses}{{
		\footnote{
			
			\noindent \textsuperscript{1,2}Department of Mathematics, Indian Institute of Technology Roorkee-IIT Roorkee,
			Haridwar Highway, Roorkee, Uttarakhand 247667, INDIA.\par\nopagebreak
			\noindent  \textit{e-mail:} \texttt{Manil T. Mohan: maniltmohan@ma.iitr.ac.in, maniltmohan@gmail.com.}
			
			\textit{e-mail:} \texttt{Kush Kinra: kkinra@ma.iitr.ac.in.}
			
			\noindent \textsuperscript{*}Corresponding author.
			
			\textit{Key words:} Stochastic convective Brinkman-Forchheimer equations, unbounded domains, cylindrical Wiener process,  absorbing sets, asymptotically compactness, random attractors.
			
			Mathematics Subject Classification (2020): Primary 35B41, 35Q35; Secondary 37L55, 37N10, 35R60.

}}}

\begin{document}
	
	\title[Random attractors for SCBF equations]{Random attractors for 2D and 3D stochastic convective Brinkman-Forchheimer equations in some unbounded domains
		\Addresses}
	\author[K. Kinra and M. T. Mohan]
	{Kush Kinra\textsuperscript{1} and Manil T. Mohan\textsuperscript{2*}}

	\maketitle
	
	\begin{abstract}
		In this work, we consider the two and three-dimensional stochastic convective Brinkman-Forchheimer (2D and 3D SCBF) equations driven by irregular additive white noise $$\d\u-[\mu \Delta\u-(\u\cdot\nabla)\u-\alpha\u-\beta|\u|^{r-1}\u-\nabla p]\d t=\boldsymbol{f}\d t+\d\W,\ \nabla\cdot\u=0,$$ for $r\in[1,\infty),$ $\mu,\alpha,\beta>0$ in unbounded domains (like Poincar\'e domains) $\mathcal{O}\subset\R^d$ ($d=2,3$) where $\W(\cdot)$ is a Hilbert space valued Wiener process on some given filtered probability space, and discuss the asymptotic behavior of its solution. For $d=2$ with $r\in[1,\infty)$ and $d=3$ with $r\in[3,\infty)$ (for $d=r=3$ with $2\beta\mu\geq 1$), we first prove the existence and uniqueness of a weak  solution (in the analytic sense) satisfying the energy equality for SCBF equations driven by an irregular additive white noise in Poincar\'e domains by using a Faedo-Galerkin approximation technique. Since the energy equality for SCBF equations is not immediate, we construct a sequence which converges in Lebesgue  and Sobolev spaces simultaneously and it helps us to demonstrate the energy equality. Then, we establish the existence of random attractors for the stochastic flow generated by the SCBF equations. One of the technical difficulties connected with the irregular white noise is overcome with the help of the corresponding Cameron-Martin space (or Reproducing Kernel Hilbert space). Furthermore, we observe that the regularity of the irregular white noise needed to obtain random attractors for the SCBF equations for $d=2$ with $r\in[1,3]$ and $d=r=3$ with $2\beta\mu\geq1$, is the same as that in the case of 2D Navier-Stokes equations, whereas for the cases $d=2,3$ and $r\in(3,\infty)$, we require more spatial regularity on the noise. Finally, we address the existence of a unique invariant measure for 2D and 3D SCBF equations defined on Poincar\'e domains (bounded or unbounded). Moreover, we provide a remark on the extension of the above mentioned results  to general unbounded domains also. 
	\end{abstract}

	\section{Introduction} \label{sec1}\setcounter{equation}{0}
			\subsection{Literature survey and motivations} 
			Analysis of the asymptotic behavior of dynamical systems is one of the most significant and far-reaching areas  of mechanics and mathematical physics. As far as the theory of deterministic infinite dimensional dynamical systems are concerned, the concept of attractors occupies a central position (cf. \cite{R.Temam}). In the study of dynamics of  stochastic partial differential equations (SPDEs), an elementary problem is to establish that it generates a random dynamical system (RDS) or stochastic flow. It is well-known in the literature that a large class of PDEs with stationary random coefficients and It\^o stochastic ordinary differential equations generate random dynamical systems (cf. \cite{Arnold,PEK}).  The analysis of infinite dimensional RDS is also an essential branch in the study of qualitative properties of SPDEs (cf. \cite{BCF,CF,Crauel}, etc. for more details). In this work, we consider the random dynamics of convective Brinkman-Forchheimer (CBF) equations, which  describe the motion of incompressible fluid flows in a saturated porous medium.  In mathematical point of view, CBF model can also be considered as damped Navier-Stokes equations (NSE). 	Our plan is to discuss the long time behavior of the solutions of two and three-dimensional stochastic convective Brinkman-Forchheimer (SCBF) equations driven by irregular additive white noise.
	
		We consider the following CBF equations in $\mathcal{O}$ (satisfying Assumption \ref{assumpO} given below) with homogeneous Dirichlet boundary conditions:
		\begin{equation}\label{1}
			\left\{
			\begin{aligned}
			\frac{\partial \u}{\partial t}-\mu \Delta\u+(\u\cdot\nabla)\u+\alpha\u+\beta|\u|^{r-1}\u+\nabla \p&=\boldsymbol{f}, \ \text{ in } \ \mathcal{O}\times(0,\infty), \\ \nabla\cdot\u&=0, \ \text{ in } \ \mathcal{O}\times(0,\infty), \\
			\u&=\mathbf{0},\ \text{ on } \ \partial\mathcal{O}\times[0,\infty), \\
			\u(0)&=\u_0 \ \text{ in } \ \mathcal{O},\\
			\int_{\mathcal{O}}p(x,t)\d x&=0, \ \text{ in } \ (0,\infty).
			\end{aligned}
			\right.
		\end{equation}
	Here $\u(x,t) \in \R^d$, $p(x,t)\in\R$ and $\f(x,t)\in\R^d$ represent the velocity field at position $x$ and time $t$, the pressure field and an external forcing, respectively. The final condition in \eqref{1} is imposed for the uniqueness of the pressure $p$. The constant $\mu>0$ represents the Brinkman coefficient (effective viscosity), the positive constants $\alpha$ and $\beta$ stand for the Darcy (permeability of porous medium) and Forchheimer (proportional to the porosity of the material) coefficients, respectively. The exponent $r\in[1,\infty)$ is called the absorption exponent. For $\alpha=\beta=0$, we obtain the classical $d$-dimensional NSE. For the unique solvability of the deterministic system \eqref{1} on bounded domains,  the interested readers are referred to see \cite{AO,FHR,HR,Mohan1}, etc., and for its stochastic counterpart, see \cite{Mohan}. The asymptotic analysis of the deterministic system \eqref{1} (for $d=2$) in general unbounded domains is addressed in the works \cite{Mohan2,MTM2}, etc. The random dynamics for 2D and 3D SCBF equations driven by Hilbert space valued additive white noise on bounded or periodic domains are discussed in \cite{KM1,KM3}. The random dynamics for 2D and 3D SCBF equations driven by finite dimensional additive/multiplicative white noise on whole space is discussed in \cite{KM6,KM7}. 
	
	To the best of our knowledge, there are no results available in the literature on the existence and uniqueness of solutions as well as the existence of random attractors for 2D and 3D SCBF equations driven by irregular (rough) Hilbert space valued additive white noise in unbounded domains. Moreover, the results on the existence and uniqueness of invariant measures for 2D and 3D SCBF model  on unbounded domains are also new. Likewise 3D NSE, the global existence and uniqueness of  strong solutions for the equations \eqref{1} (for $d=3$) with $r\in[1,3)$ and $r=3$ (when $2\beta\mu<1$) is still an open problem. {Therefore, there are three distinct cases to be considered and we summarize them as follows (see Table \ref{Table1} below):}
	\begin{table}[ht]
		{	\begin{tabular}{|c|c|c|c|c|}
				\hline
				\textbf{Cases}& $d$ &$ r$& conditions on $\mu$ \& $\beta$ \\
				\hline
				\textbf{I}& $d=2$ &$r\in[1,\infty)$&  for any   $\mu>0$ and $\beta>0$  \\
				\hline
				\textbf{II}& $d=3$ &$r\in(3,\infty)$& for any $\mu>0$ and $\beta>0$ \\
				\hline
				\textbf{III}& $d=3$ &$r=3$&for $\mu>0$ and $\beta>0$ with $2\beta\mu\geq1$ \\
				\hline
			\end{tabular}
			\vskip 0.1 cm
			\caption{Values of $\mu,\beta$ and $r$ for $d=2,3$.}
			\label{Table1}}
	\end{table}

	\begin{assumption}\label{assumpO}
	Let $\mathcal{O}$ be an open and connected subset of $\R^d$ $(d=2,3)$, the boundary of which is uniformly of class $\mathrm{C}^3$ (cf. \cite{Heywood}).  For the domain $\mathcal{O}$, we also assume that, there exists a positive constant $\lambda_1 $ such that the following Poincar\'e inequality is satisfied:
			\begin{align}\label{2.1}
				\lambda_1\int_{\mathcal{O}} |\psi(x)|^2 \d x \leq \int_{\mathcal{O}} |\nabla \psi(x)|^2 \d x,  \ \text{ for all } \  \psi \in \H^{1}_0 (\mathcal{O}).
			\end{align}
	\end{assumption}
	A domain in which Poincar\'e's inequality is satisfied, we call it as a \emph{Poincar\'e domain} (cf. \cite[p.306]{R.Temam} and \cite[p.117]{Robinson2}). It can be easily seen that if $\mathcal{O}$ is bounded in some direction, then the Poincar\'e inequality holds.  For example, in two dimensions, if $x=(x_1,x_2)\in\mathbb{R}^2$, then one can take $\mathcal{O}$ is included in a region of the form  $0<x_1<L$. 
	
	\subsection{Difficulties and approaches}
	 For 2D as well as 3D CBF/SCBF equations with $r\geq3$, it is not easy to show that the solution satisfies the energy equality (unless the existence of strong solution is known). One needs to construct  a sequence which converges in both Sobolev space ($\H^1(\mathcal{O})$) and Lebesgue space ($\L^{r+1}(\mathcal{O})$) simultaneously. In \cite{FHR,HR1}, the authors presented an idea  to solve this problem by using the eigenfunctions of the Stokes operator on periodic and bounded domains. Later, the author in \cite{Mohan1,Mohan} used this method to prove the energy equality for CBF and SCBF equations on bounded domains. But this method is no longer applicable in unbounded domains like Poincar\'e domains due to the lack of eigenvalues and eigenfunctions of the Stokes operator.  In the case of unbounded domains, we know that $\C_0^{\infty}(\mathcal{O};\R^d)$ is dense $\H^2(\mathcal{O})$ and $\H^2(\mathcal{O})$ is continuously embedded in $\H^1(\mathcal{O})$ as well as in $\L^{r+1}(\mathcal{O})$ (for $d=2,3$), there is always a sequence in $\C_0^{\infty}(\mathcal{O};\R^d)$ (or even in $\H^2(\mathcal{O})$) such that it will converge in both $\H^1(\mathcal{O})$ and $\L^{r+1}(\mathcal{O})$ simultaneously. But the construction of such sequences satisfying the divergence free condition is the main task (see the spaces $\V$ and $\wi\L^{r+1}$ defined in Section \ref{sec2}). From the work \cite{BM}, we know the existence of a self-adjoint operator $\mathcal{L}$ in unbounded domains whose inverse is compact and the eigenfunctions of operator $\mathcal{L}$ form an orthonormal basis in $\L^2(\mathcal{O})$ (see Subsection \ref{C_O} below). Using the eigenfunctions of operator $\mathcal{L}$, we construct a sequence which converges in both $\V$ and $\wi\L^{r+1},$ simultaneously, and help us to obtain the energy inequality.
	
	A major prevailing result on the random attractors for SPDEs (associated with the Gelfand triple $\V\hookrightarrow \H\hookrightarrow\V'$, where $\V$ is a separable Banach space  with its topological dual $\V'$ and $\H$ is a separable Hilbert space) depends heavily on the existence of a random compact attracting set (cf. \cite{CDF}). But in the case of unbounded domains, the embedding $\V\hookrightarrow \H$ is no longer compact. Therefore, we are not able to prove the existence of random attractors using the compactness criterion. In the deterministic case, this difficulty (in unbounded domains) was resolved by different methods, cf.  \cite{Abergel, Ghidaglia,Rosa}, etc., for the autonomous case and \cite{CLR1, CLR2}, etc., for the non-autonomous case. For SPDEs, the methods available in the deterministic case have also been  generalized by several authors (see for example, \cite{BLL,BLW, BL, Wang}, etc.). In particular, the authors in \cite{BL} considered the 2D stochastic NSE  in Poincar\'e domains perturbed by a very general irregular additive white noise and the existence of stochastic flow (or RDS) is provided. Apart from that, they proposed sufficient conditions for the existence of a unique random attractor in \cite{BCLLLR}. The existence of a unique random attractor for the 2D stochastic NSE  in Poincar\'e domains is proved in \cite{BCLLLR}.  
	
	The concept of an asymptotically compact cocycle was introduced in \cite{CLR1} and the authors have established the existence of attractors for the non-autonomous 2D Navier-Stokes equations. Later, this concept has been  utilized to prove the existence of random attractors for several SPDEs like 1D stochastic lattice differential equation \cite{BLL}, stochastic NSE on the 2D unit sphere \cite{BGT}, stochastic $g$-NSE \cite{FY,LL,LXS}, stochastic non-autonomous Kuramoto-Sivashinsky equations \cite{LYZ}, stochastic heat equations in materials with memory on thin domains \cite{SLHZ}, stochastic reaction-diffusion equations \cite{BLW,Slavik}, 3D stochastic Benjamin-Bona-Mahony equations \cite{Wang}, etc., and references therein. 
	
	The existence of a random attractor for stochastic 3D NSE with damping driven by a multiplicative noise is established in \cite{LG}. The authors in \cite{You} and \cite{HZ} showed the existence of a random attractor and exponential attractor, respectively, for 3D damped NSE in bounded domains with additive noise by verifying the pullback flattening property. But in 3D bounded domains, due to the technical difficulties described in the works \cite{KZ,Mohan}, etc.,  (commutativity of the projection operator with $-\Delta$ and the nonzero boundary condition of projected nonlinear damping term), some of the results obtained in the above mentioned works may not hold true.

	Recently, authors in \cite{GLS} proved the existence of random attractors for SPDEs having locally monotone terms by assuming that the embedding $\V\hookrightarrow \H$ is compact. Even though our system satisfies a local monotonicity property for $d=2$ with $r\in[1,3]$ (see \eqref{fe2_1} below), this work does not fall in the framework of \cite{GLS}, as the embedding $\V\hookrightarrow \H$ is not compact in unbounded domains. 
	
	\subsection{Novelties of the work}
	In this paper, our  aim is to prove the existence  and uniqueness of weak solutions, and the existence and uniqueness of random attractors of the following stochastic convective Brinkman-Forchheimer equations perturbed by irregular additive white noise in unbounded domain $\mathcal{O}$ satisfying Assumption \ref{assumpO}:
	\begin{equation}\label{SCBF}
		\left\{
		\begin{aligned}
			\d\u+[-\mu \Delta\u+(\u\cdot\nabla)\u+\alpha\u+\beta|\u|^{r-1}\u+\nabla p]\d t&=\boldsymbol{f}\d t + \d\mathrm{W}, \ \text{ in } \ \mathcal{O}\times(0,\infty), \\ \nabla\cdot\u&=0, \ \text{ in } \ \mathcal{O}\times(0,\infty), \\
			\u&=\mathbf{0},\ \ \text{ on } \ \partial\mathcal{O}\times[0,\infty), \\
			\u(0)&=\boldsymbol{x}, \ \text{ in } \ \mathcal{O},
		\end{aligned}
		\right.
	\end{equation} for $d=2$ with $r\geq1$, $d=3$ with $r\geq3$ and $d=r=3$ with $2\beta\mu\geq1$, where $\W(\cdot)$ is an $\H$-valued Wiener process on some given filtered probability space $(\Omega, \mathcal{F}, (\mathcal{F}_t)_{t\in \R}, \mathbb{P})$, whose properties will be specified in Section \ref{sec2} below.	The existence of a unique weak solution satisfying the energy equality to SCBF equations (the transformed system \eqref{cscbf}) is proved by using a Faedo-Galerkin approximation technique. We use the concepts developed in \cite{BCLLLR} to prove the existence of a unique global random attractor for SCBF equations (with irregular white noise) in unbounded domains. As discussed in \cite{BCLLLR},  we provide a special attention to the noise with low spatial regularity. The asymptotic compactness of RDS generated by \eqref{SCBF} is proved using the method of energy equations introduced in \cite{Ball}. We consider an infinite dimensional driving Wiener process with minimal assumptions on its Cameron-Martin space (or Reproducing Kernel Hilbert space). We also point out  that the regularity of the noise needed to obtain random attractors for SCBF equations for $d=2$ with $r\in[1,3]$ and $d=r=3$ with $2\beta\mu\geq1$, is the same as that of 2D NSE  (cf. \cite{BL} and see Assumptions \ref{assump1} and \ref{assump2}), whereas for the case $d=2,3$ with $r\in(3,\infty)$, we require more spatial regularity on the noise (see Assumption \ref{assump2}). 
	 
	 In \cite{CF}, the authors proved that the existence of compact invariant random set is a sufficient condition for the existence of invariant measures. They have applied this concept  to prove the existence of invariant measures for reaction-diffusion equations and 2D stochastic NSE in bounded domains.   The authors in \cite{BL,KM8}, etc., used this idea to prove the existence of random attractors for 2D stochastic NSE in unbounded domains.   Since, the random attractor itself is a compact invariant set, the existence of invariant measures is assured. In addition, we prove the uniqueness of  invariant measures for system \eqref{SCBF} in Poincar\'e domains by using the exponential stability of solutions. 
		\subsection{Outline} 
		The rest of the paper is organized as follows: In the next section, we provide the necessary function spaces needed to obtain the existence and uniqueness of random attractors for the system \eqref{SCBF}. Also, we define the linear and nonlinear operators, and explain  their properties. Moreover, we provide an abstract formulation to the system \eqref{SCBF} in the same section. The metric dynamical system (MDS) and random dynamical system (RDS) corresponding to SCBF equations is constructed in Section \ref{sec5}. The existence and uniqueness of a weak solution satisfying the energy equality to the transformed SCBF equations (see \eqref{cscbf}) by using a Faedo-Galerkin approximation technique is also established (Theorem \ref{solution}) in the same section.  Section \ref{sec6} is devoted for establishing  the main result of this paper, that is, the existence of a random attractor for 2D and 3D SCBF equations on  Poincar\'e domains. In order to do this, we first present Lemma \ref{RA1}, which provides  us the energy estimates for SCBF equations. Then, we prove the weak continuity of the RDS generated by SCBF equations in Lemmas \ref{weak_topo1} and \ref{weak_topo2}. Based on Lemma \ref{RA1}, we introduce new classes of functions $\mathfrak{K}_1$ and $\mathfrak{K}_2$, which are defined in Definition \ref{RA2}. Then, we define two classes $\mathfrak{DK}_1$ and $\mathfrak{DK}_2$ of closed and bounded random sets using functions in the classes $\mathfrak{K}_1$ and $\mathfrak{K}_2$, respectively. We achieve the goal of this work by proving  Theorem \ref{Main_theorem_1}, which affirms that the RDS $\varphi$ generated by SCBF equations on Poincar\'e domains is $\mathfrak{DK}_1$-asymptotically compact (for $d=2$ with $r\in[1,3)$) and $\mathfrak{DK}_2$-asymptotically compact (for $d=2,3$ with $r\geq3$). Hence, in view of \cite[Theorem 2.8 ]{BCLLLR}, the existence of a random attractor of $\varphi$ is deduced. In the final section, we show the existence of a unique invariant measure for the system \eqref{SCBF} in Poincar\'e domains  (Theorem \ref{UIM2}).

	\section{Mathematical Formulation}\label{sec2}\setcounter{equation}{0}
	In this section, we provide the necessary function spaces needed to obtain the existence of random attractors for SCBF equations. Furthermore, we define some operators and their properties to get an abstract formulation for the system \eqref{SCBF} and main result of this work.
	\subsection{Function spaces} Let $\C_0^{\infty}(\mathcal{O};\R^d)$ denote the space of all infinite times differentiable functions  ($\R^d$-valued) with compact support in $\mathcal{O}\subset\R^d$. We define 
	\begin{align*} 
		\mathcal{V}&:=\{\u\in\C_0^{\infty}(\mathcal{O};\R^d):\nabla\cdot\u=0\},\\
		\mathbb{H}&:=\text{the closure of }\ \mathcal{V} \ \text{ in the Lebesgue space } \L^2(\mathcal{O})=\mathrm{L}^2(\mathcal{O};\R^d),\\
		\mathbb{V}&:=\text{the closure of }\ \mathcal{V} \ \text{ in the Sobolev space } \H^1(\mathcal{O})=\mathrm{H}^1(\mathcal{O};\R^d),\\
			\mathbb{V}_s&:=\text{the closure of }\ \mathcal{V} \ \text{ in the Sobolev space } \H^s(\mathcal{O})=\mathrm{H}^s(\mathcal{O};\R^d),
			&\text{  for  } s>1,\\
		\widetilde{\L}^{p}&:=\text{the closure of }\ \mathcal{V} \ \text{ in the Lebesgue space } \L^p(\mathcal{O})=\mathrm{L}^p(\mathcal{O};\R^d),
	&\text{	for  }p\in(2,\infty).
	\end{align*}
	 Then, we characterize the spaces $\H$, $\V$ and $\widetilde{\L}^p$   with norms $$\|\u\|_{\H}^2:=\int_{\mathcal{O}}|\u(x)|^2\d x, \ \|\u\|_{\V}^2:=\int_{\mathcal{O}}|\nabla\u(x)|^2\d x	\ \text{ and }\ \|\u\|_{\widetilde{\L}^p}^p=\int_{\mathcal{O}}|\u(x)|^p\d x,$$ respectively. Let $(\cdot,\cdot)$ and $(\!(\cdot,\cdot)\!)$ denote the inner product in the Hilbert space $\H$ and $\V$, respectively, and $\langle \cdot,\cdot\rangle $ denote the induced duality between the spaces $\V$  and its dual $\V'$ as well as $\widetilde{\L}^p$ and its dual $\widetilde{\L}^{p'}$, where $\frac{1}{p}+\frac{1}{p'}=1$. We endow the space $\V\cap\widetilde{\L}^{p}$ with the norm $\|\u\|_{\V}+\|\u\|_{\widetilde{\L}^{p}},$ for $\u\in\V\cap\widetilde{\L}^p$ and its dual $\V'+\widetilde{\L}^{p'}$ with the norm $$\inf\left\{\|\v_1\|_{\V'}+\|\v_1\|_{\widetilde{\L}^{p'}}:\v=\v_1+\v_2, \ \v_1\in\V', \ \v_2\in\widetilde{\L}^{p'}\right\}.$$
	Moreover, we have the continuous embedding $\V\cap\widetilde{\L}^p\hookrightarrow\V\hookrightarrow\H\cong\H^{'}\hookrightarrow\V'\hookrightarrow\V'+\widetilde{\L}^{p'}$.
	\subsection{Linear operator}
	Let $\mathcal{P}: \L^2(\mathcal{O}) \to\H$ denote the Helmholtz-Hodge orthogonal projection (cf.  \cite{OAL}). Let us define the Stokes operator 
	\begin{equation*}
		\A\u:=-\mathcal{P}\Delta\u,\;\u\in\D(\A).
	\end{equation*}
	The operator $\A$ is a linear continuous operator from $\V$ into $\V'$, satisfying
	\begin{equation*}
		\langle\A\u,\v\rangle=(\!(\u,\v)\!), \ \ \ \u,\v\in\V.
	\end{equation*}
	Since the boundary of $\mathcal{O}$ is uniformly of class $\mathrm{C}^3$, we infer that $\D(\A)=\V\cap\H^2(\mathcal{O})$ and $\|\A\u\|_{\H}$ defines a norm in $\D(\A),$ which is equivalent to the one in $\H^2(\mathcal{O})$ (cf. \cite[Lemma 1]{Heywood}). Above argument implies that $\mathcal{P}:\H^2(\mathcal{O})\to\H^2(\mathcal{O})$ is a bounded operator. Note that the operator $\A$ is a non-negative self-adjoint operator in $\H$ and 
	\begin{align}\label{2.7a}
		\langle\A\u,\u\rangle =\|\u\|_{\V}^2,\ \textrm{ for all }\ \u\in\V, \ \text{ so that }\ \|\A\u\|_{\V'}\leq \|\u\|_{\V}.
	\end{align}
	\begin{remark}
		Since $\mathcal{O}$ is a Poincar\'e domain, then $\A$ is invertible and its inverse $\A^{-1}$ is bounded. Moreover, for $\u\in\D(\A)$, we have 
		\begin{align*}
			\|\u\|_{\V}^2=(\nabla\u,\nabla\u)=(\A\u,\u)\leq\|\A\u\|_{\H}\|\u\|_{\H}\leq\frac{1}{\lambda_1^{1/2}}\|\A\u\|_{\H}\|\u\|_{\V},
		\end{align*}
		so that we get $\|\A\u\|_{\H}\geq\lambda_1^{-1/2}	\|\u\|_{\V},$ for all $\u\in\D(\A).$
	\end{remark}
	\subsection{Bilinear operator}
	Next, we define the \emph{trilinear form} $b(\cdot,\cdot,\cdot):\V\times\V\times\V\to\R$ by $$b(\u,\v,\w)=\int_{\mathcal{O}}(\u(x)\cdot\nabla)\v(x)\cdot\w(x)\d x=\sum_{i,j=1}^d\int_{\mathcal{O}}\u_i(x)\frac{\partial \v_j(x)}{\partial x_i}\w_j(x)\d x.$$ If $\u, \v$ are such that the linear map $b(\u, \v, \cdot) $ is continuous on $\V$, the corresponding element is denoted by $\B(\u, \v)\in \V'$. We also denote $\B(\u) = \B(\u, \u)=\mathcal{P}[(\u\cdot\nabla)\u]$. Using an integration by parts, we obtain 
	\begin{equation}\label{b0}
		\left\{
		\begin{aligned}
			b(\u,\v,\v) &= 0,\ \text{ for all }\ \u,\v \in\V,\\
			b(\u,\v,\w) &=  -b(\u,\w,\v),\ \text{ for all }\ \u,\v,\w\in \V.
		\end{aligned}
		\right.\end{equation}
	The following interpolation inequality is used frequently in the upcoming sections. 
	\begin{lemma}[Interpolation inequality] \label{Interpolation}
		Assume $1\leq s_1\leq s\leq s_2\leq \infty$, $a\in(0,1)$ such that $\frac{1}{s}=\frac{a}{s_1}+\frac{1-a}{s_2}$ and $\u\in\L^{s_1}(\mathcal{O})\cap\L^{s_2}(\mathcal{O})$, then we have 
		\begin{align*}
			\|\u\|_{\L^s(\mathcal{O})}\leq\|\u\|_{\L^{s_1}(\mathcal{O})}^{a}\|\u\|_{\L^{s_2}(\mathcal{O})}^{1-a}. 
		\end{align*}
	\end{lemma}
	\begin{remark}
		The following well-known inequality is due to Ladyzhenskaya (Lemmas 1 and 2 of \cite[Chapter I]{OAL}):
		\begin{align}\label{lady}
			\|\v\|_{\L^{4}(\mathcal{O}) } \leq \begin{cases}
				2^{1/4} \|\v\|^{1/2}_{\L^{2}(\mathcal{O}) } \|\nabla \v\|^{1/2}_{\L^{2}(\mathcal{O}) }, \ \ \ \v\in \H^{1,2}_{0} (\mathcal{O}), \text{ for } d=2,\\
				2^{1/2} \|\v\|^{1/4}_{\L^{2}(\mathcal{O}) } \|\nabla \v\|^{3/4}_{\L^{2}(\mathcal{O}) }, \ \ \ \v\in \H^{1,2}_{0} (\mathcal{O}), \text{ for } d=3.
			\end{cases}
		\end{align}
	\end{remark}
	\begin{remark}
		1. 	In the trilinear form, using H\"older's inequality, we obtain
		\begin{align}\label{exten1}
			|b(\u,\v,\w)|=|b(\u,\w,\v)|\leq \|\u\|_{\widetilde{\L}^{4}}\|\v\|_{\widetilde{\L}^{4}}\|\w\|_{\V},
		\end{align}
		for all $\u, \v,  \w\in\V$ and
		\begin{align}\label{2p9 1}
			\|\B(\u,\v)\|_{\V'}\leq \|\u\|_{\widetilde{\L}^{4}}\|\v\|_{\widetilde{\L}^{4}}.
		\end{align}
		2. 	If $\u\in \mathrm{L}^4 (0, T; \widetilde{\L}^4),$ then $\B(\u)\in \mathrm{L}^2(0, T; \V').$ Indeed, by \eqref{2p9 1} we have 
		\begin{align}\label{2.8}
			\int_{0}^{T} \|\B(\u(t))\|^2_{\V'}\d t \leq \int_{0}^{T}\|\u(t)\|^4_{\widetilde{\L}^4}\d t < \infty. 
		\end{align}
		
	\end{remark}
	\begin{remark}
		For $r> 3$, using interpolation inequality (Lemma \ref{Interpolation}), we find  
		\begin{align*}
			\left|\langle \B(\u,\u),\v\rangle \right|=\left|b(\u,\v,\u)\right|\leq \|\u\|_{\widetilde{\L}^{\frac{2(r+1)}{r-1}}}\|\u\|_{\widetilde{\L}^{r+1}}\|\v\|_{\V}\leq\|\u\|_{\widetilde{\L}^{r+1}}^{\frac{r+1}{r-1}}\|\u\|_{\H}^{\frac{r-3}{r-1}}\|\v\|_{\V},
		\end{align*}
		for all $\v\in\V$. Thus, we have 
		\begin{align}\label{2.9a}
			\|\B(\u)\|_{\V'}\leq\|\u\|_{\widetilde{\L}^{r+1}}^{\frac{r+1}{r-1}}\|\u\|_{\H}^{\frac{r-3}{r-1}}.
		\end{align}
		Moreover, for $r>3$, if $\u\in \mathrm{L}^2(0,T; \H)\cap\mathrm{L}^{r+1} (0, T; \widetilde{\L}^{r+1}),$ then $\B(\u)\in \mathrm{L}^2(0, T; \V').$ Indeed, making use of \eqref{2.9a}, we get 
		\begin{align}\label{2.13}
			\int_{0}^{T} \|\B(\u(t))\|^2_{\V'}\d t \leq  \int_{0}^{T}\|\u(t)\|_{\widetilde{\L}^{r+1}}^{\frac{2(r+1)}{r-1}}\|\u(t)\|_{\H}^{\frac{2(r-3)}{r-1}}\d t \leq \|\u\|_{\mathrm{L}^{r+1}(0, T; \widetilde{\L}^{r+1})}^{\frac{2(r+1)}{r-1}}\|\u\|_{\mathrm{L}^2(0, T; \H)}^{\frac{2(r-3)}{r-1}}< \infty. 
		\end{align}
	\end{remark}
	\begin{remark}\label{RemarkI}
		Using interpolation inequality (Lemma \ref{Interpolation}), we have the following observation:\\
		1. For $r\in[1,3]$, if $\y\in \mathrm{L}^2(0,T; \H)\cap\mathrm{L}^{4} (0, T; \widetilde{\L}^{4}),$ then $\y\in \mathrm{L}^{r+1} (0, T; \widetilde{\L}^{r+1}).$ Indeed 
		
		\begin{align*}
			\int_{0}^{T} \|\y(t)\|^{r+1}_{\wi \L^{r+1}}\d t \leq  \int_{0}^{T}\|\y(t)\|_{\widetilde{\L}^{4}}^{2(r-1)}\|\y(t)\|_{\H}^{3-r}\d t \leq \|\y\|_{\mathrm{L}^{4}(0, T; \widetilde{\L}^{4})}^{2(r-1)}\|\y\|_{\mathrm{L}^2(0, T; \H)}^{3-r}< \infty.
		\end{align*}
		2.	For $r>3$, if $\y\in \mathrm{L}^2(0,T; \H)\cap\mathrm{L}^{r+1} (0, T; \widetilde{\L}^{r+1}),$ then $\y\in \mathrm{L}^{4} (0, T; \widetilde{\L}^{4}).$ Indeed 
		
		\begin{align*}
			\int_{0}^{T} \|\y(t)\|^4_{\wi \L^4}\d t \leq  \int_{0}^{T}\|\y(t)\|_{\widetilde{\L}^{r+1}}^{\frac{2(r+1)}{r-1}}\|\y(t)\|_{\H}^{\frac{2(r-3)}{r-1}}\d t \leq \|\y\|_{\mathrm{L}^{r+1}(0, T; \widetilde{\L}^{r+1})}^{\frac{2(r+1)}{r-1}}\|\y\|_{\mathrm{L}^2(0, T; \H)}^{\frac{2(r-3)}{r-1}}< \infty.
		\end{align*}
	\end{remark}
	Let us now provide some convergence results regarding the operator $b(\cdot,\cdot,\cdot)$, which will be used in the subsequent sections of the paper. 
	\begin{lemma}[{\cite[Ch. III, Lemma 3.2]{Temam}}]\label{convergence_b*}
		Let $\mathcal{O}_1\subset\mathcal{O}$, which is bounded, and $\psi: [0, T]\times \mathcal{O} \to \R^d$ ($d=2,3$) be a $\mathrm{C}^1$-class function such that $\mathrm{supp} (\psi (t, \cdot)) \subset \mathcal{O}_1,$ for $t\in[0, T],$ and 
		$$ \sup_{1\leq i, j\leq d} \sup_{(t, x)\in [0, T] \times \mathcal{O}_1} |D_i \psi^j (t, x)| = C < \infty.$$ 
		Let $\v_m$ converges to $\v$ in $\mathrm{L}^2(0, T; \V)$ weakly and in $\mathrm{L}^2(0, T; \L^2(\mathcal{O}_1))$ strongly. Then, we have 
		$$\int_{0}^{T} b(\v_m(t), \v_m(t), \psi(t)) \d t \to \int_{0}^{T} b(\v(t), \v(t), \psi(t)) \d t\ \text{ as }\ m\to\infty.$$
	\end{lemma}
	\begin{corollary}[{\cite[Corollary 5.3]{BL}}]\label{convergence_b1}
		For $d=2$ with $r\in[1,3)$, assume that $\{\v_m\}_{m\in \mathbb{N}}$ is a bounded sequence in $\mathrm{L}^{\infty}(0, T; \H),\ \v \in \mathrm{L}^{\infty}(0, T; \H),$ $\v_m$ converges to $\v$ weakly and strongly in $\mathrm{L}^2(0, T;\V)$ and $\mathrm{L}^2(0, T; {\L}^2_{\emph{loc}} (\mathcal{O}))$, respectively. Then for any $\y \in \mathrm{L}^{4}(0,T;\widetilde{\L}^{4}),$ $$\int_{0}^{T} b(\v_m(t), \v_m(t), \y(t)) \d t \to \int_{0}^{T} b(\v(t), \v(t), \y(t)) \d t\ \text{ as }\ m\to\infty.$$
	\end{corollary}
	\begin{corollary}\label{convergence_b2}
		For $d=2,3$ with $r\geq3$, assume that $\{\v_m\}_{m\in \mathbb{N}}$ is a bounded sequence in $\mathrm{L}^{\infty}(0, T; \H), \v \in \mathrm{L}^{\infty}(0, T; \H),$ $\v_m$ converges to $\v$ weakly and strongly in $\mathrm{L}^2(0, T;\V)\cap\mathrm{L}^{r+1} (0, T; \widetilde{\L}^{r+1})$ and $\mathrm{L}^2(0, T; \L^2_{\emph{loc}}(\mathcal{O}))$, respectively. Then, for any \ \ $\y \in \mathrm{L}^{2}(0,T;\H)\cap\mathrm{L}^{r+1}(0,T; \wi\L^{r+1})$, $$\int_{0}^{T} b(\v_m(t), \v_m(t), \y(t)) \d t \to \int_{0}^{T} b(\v(t), \v(t), \y(t)) \d t\ \text{ as }\ m\to\infty.$$
	\end{corollary}
	\begin{proof}
		From the assumptions, we can find a  constant $L>0,$ such that
		\begin{align*}
			\left(\int_{0}^{T}\|\v_m(t)\|^{2}_{\V}\d t\right)^{\frac{1}{2}}+\left(\int_{0}^{T}\|\v(t)\|^{2}_{\V}\d t\right)^{\frac{1}{2}}&+\left(\int_{0}^{T}\|\v_m(t)\|^{r+1}_{\wi\L^{r+1}}\d t\right)^{\frac{1}{r+1}}\nonumber\\&+\left(\int_{0}^{T}\|\v(t)\|^{r+1}_{\wi\L^{r+1}}\d t\right)^{\frac{1}{r+1}} \leq L.
		\end{align*}
		Let us choose $\epsilon > 0$. Since $\y \in \mathrm{L}^{2}(0,T;\H)\cap\mathrm{L}^{r+1}(0,T; \wi\L^{r+1}),$ by a standard regularization method, there exists a function $\psi$ satisfying the assumptions of Lemma \ref{convergence_b*} such that $(\int_{0}^{T} \|\y(s)- \psi(s)\|^{r+1}_{\wi{\L}^{r+1}}\d s)^{\frac{1}{r+1}} < \frac{\epsilon}{3L^2}$ and $(\int_{0}^{T} \|\y(s)- \psi(s)\|^{2}_{\H}\d s)^{\frac{1}{2}} < \frac{\epsilon}{3L^2}$. Making use of Lemma \ref{convergence_b*}, we can find $M_{\epsilon} \in \mathbb{N}$ such that  $$\left|\int_{0}^{T} b(\v_m(t), \v_m(t), \psi(t)) \d t - \int_{0}^{T} b(\v(t), \v(t), \psi(t)) \d t\right| < \frac{\epsilon}{3},$$ for all $m \geq M_{\epsilon}.$ Hence,  for $m > M_{\epsilon},$ using H\"older's and interpolation (see \eqref{Interpolation}) inequalities, we obtain (for $r>3$)
		\begin{align*}
			& \left|\int_{0}^{T} b(\v_m(t), \v_m(t), \y(t)) \d t - \int_{0}^{T} b(\v(t), \v(t), \y(t)) \d t\right|\\
			&	\leq \left|\int_{0}^{T} b(\v_m(t), \v_m(t), \y(t)-\psi(t)) \d t\right| + \left|\int_{0}^{T} b(\v(t), \v(t),\y(t)- \psi(t)) \d t\right| \\
			&\quad+ \left|\int_{0}^{T} b(\v_m(t), \v_m(t), \psi(t)) \d t - \int_{0}^{T} b(\v(t), \v(t), \psi(t)) \d t\right|\\
			&	<  \frac{\epsilon}{3} + \int_{0}^{T} \|\v_m(t)\|_{\widetilde{\L}^{r+1}} \|\v_m(t)\|_{\V} \|\y(t)-\psi(t)\|_{\wi{\L}^{\frac{2(r+1)}{r-1}}} \d t\\&\quad+ \int_{0}^{T} \|\v(t)\|_{\widetilde{\L}^{r+1}} \|\v(t)\|_{\V} \|\y(t)-\psi(t)\|_{\wi{\L}^{\frac{2(r+1)}{r-1}}}\d t \\
			&	\leq  \frac{\epsilon}{3} + \int_{0}^{T} \|\v_m(t)\|_{\widetilde{\L}^{r+1}} \|\v_m(t)\|_{\V} \|\y(t)-\psi(t)\|^{\frac{2}{r-1}}_{\widetilde{\L}^{r+1}}\|\y(t)-\psi(t)\|^{\frac{r-3}{r-1}}_{\H} \d t\\&\quad+ \int_{0}^{T} \|\v(t)\|_{\widetilde{\L}^{r+1}} \|\v(t)\|_{\V} \|\y(t)-\psi(t)\|^{\frac{2}{r-1}}_{\wi{\L}^{r+1}}\|\y(t)-\psi(t)\|^{\frac{r-3}{r-1}}_{\H}\d t \\
			&\leq \frac{\epsilon}{3} + \left(\int_{0}^{T}\|\v_m(t)\|^{r+1}_{\wi\L^{r+1}}\d t\right)^{\frac{1}{r+1}}\left(\int_{0}^{T}\|\v_m(t)\|^{2}_{\V}\d t\right)^{\frac{1}{2}}\left(\int_{0}^{T} \|\y(t)- \psi(t)\|^{r+1}_{{\L}^{r+1}}\d t\right)^{\frac{2}{(r+1)(r-1)}}\nonumber\\&\quad\times\left(\int_{0}^{T} \|\y(t)- \psi(t)\|^{2}_{\H}\d t\right)^{\frac{r-3}{2(r-1)}} + \left(\int_{0}^{T}\|\v(t)\|^{r+1}_{\wi\L^{r+1}}\d t\right)^{\frac{1}{r+1}}\left(\int_{0}^{T}\|\v(t)\|^{2}_{\V}\d t\right)^{\frac{1}{2}}\nonumber\\&\quad\times\left(\int_{0}^{T} \|\y(t)- \psi(t)\|^{r+1}_{{\L}^{r+1}}\d t\right)^{\frac{2}{(r+1)(r-1)}}\left(\int_{0}^{T} \|\y(t)- \psi(t)\|^{2}_{\H}\d t\right)^{\frac{r-3}{2(r-1)}}\nonumber\\&<\epsilon,
		\end{align*}
		which completes the proof. For $r=3$, proof is similar as previous case and hence we omit it here.
	\end{proof}
	\subsection{Nonlinear operator}
	Let us now consider the nonlinear operator $\mathcal{C}(\u):=\mathcal{P}(|\u|^{r-1}\u)$. It is immediate that $\langle\mathcal{C}(\u),\u\rangle =\|\u\|_{\widetilde{\L}^{r+1}}^{r+1}$ and the map $\mathcal{C}(\cdot):\V\cap\widetilde{\L}^{r+1}\to\V'+\widetilde{\L}^{\frac{r+1}{r}}$. Also, for any $r\in [1, \infty)$ and $\u_1, \u_2 \in \V\cap\wi\L^{r+1}$, we have (cf. \cite[Subsection 2.4]{Mohan}),
	\begin{align}\label{MO_c}
		\langle\mathcal{C}(\u_1)-\mathcal{C}(\u_2),\u_1-\u_2\rangle \geq\frac{1}{2}\||\u_1|^{\frac{r-1}{2}}(\u_1-\u_2)\|_{\H}^2+\frac{1}{2}\||\u_2|^{\frac{r-1}{2}}(\u_1-\u_2)\|_{\H}^2\geq 0,
	\end{align}
	and 
	\begin{align}\label{a215}
		\|\u_1-\u_2\|_{\wi\L^{r+1}}^{r+1}\leq 2^{r-2}\||\u_1|^{\frac{r-1}{2}}(\u_1-\u_2)\|_{\H}^2+2^{r-2}\||\u_2|^{\frac{r-1}{2}}(\u_1-\u_2)\|_{\H}^2,
	\end{align}
	for $r\geq 1$ (replace $2^{r-2}$ with $1,$ for $1\leq r\leq 2$).
	Let us now provide some convergence results regarding the operator $\mathcal{C}(\cdot)$, which will be used in the sequel. 
	\begin{lemma}\label{convergence_c2_1}
		Let $\mathcal{O}_1\subset\mathcal{O}$, which is bounded, and $\psi: [0, T]\times \mathcal{O} \to \R^d$ be a continuous  function such that $\mathrm{supp} (\psi (t, \cdot)) \subset \mathcal{O}_1,$ for $t\in[0, T],$ and 
		$$ \sup_{(t, x)\in [0, T] \times \mathcal{O}_1} |\psi (t, x)| = C < \infty.$$   Assume that $\{\v_m\}_{m\in \mathbb{N}}$ is a bounded sequence in the space $\mathrm{L}^{\infty}(0, T; \H), \v \in \mathrm{L}^{\infty}(0, T; \H),$ $\v_m$ converges to $\v$ weakly and strongly in $\mathrm{L}^2(0, T; \V)\cap\mathrm{L}^{r+1} (0, T; \widetilde{\L}^{r+1})$ and $\mathrm{L}^2(0, T; \L^2(\mathcal{O}_1))$, respectively. Then for any $r\in[1,3)$ with $\y\in\mathrm{L}^{4} (0, T; \widetilde{\L}^{4})\cap \mathrm{L}^2 (0, T; \H)$ and for any $r\geq3$ with $\y\in\mathrm{L}^{r+1} (0, T; \widetilde{\L}^{r+1})\cap \mathrm{L}^2 (0, T; \H)$, 
		\begin{align}\label{217}
			\int_{0}^{T} \big\langle\mathcal{C}(\v_m(t)+\y(t)) ,\psi(t)  \big\rangle \d t \to \int_{0}^{T} \big\langle\mathcal{C}(\v(t)+\y(t)) ,\psi(t)  \big\rangle \d t\ \text{ as }\ m\to\infty.
		\end{align}
	\end{lemma}
	\begin{proof}
		It is given that $\v_m$ converges to $\v$ weakly and strongly in $\mathrm{L}^2(0, T; \V)\cap\mathrm{L}^{r+1} (0, T; \widetilde{\L}^{r+1})$ and $\mathrm{L}^2(0, T; \L^2(\mathcal{O}_1))$, respectively.
		
		The case $r=1$ is obvious. Let us first consider $1< r<3$.	Using Taylor's formula (\cite[Theorem 7.9.1]{PGC}) and H\"older's inequality, we obtain 
		\begin{align*}
			& \left|\int_{0}^{T} \big\langle\mathcal{C}(\v_m(t)+\y(t)) ,\psi(t)  \big\rangle \d t - \int_{0}^{T} \big\langle\mathcal{C}(\v(t)+\y(t)) ,\psi(t)  \big\rangle \d t\right|\\&\leq \int_{0}^{T}\left| \big\langle\mathcal{C}(\v_m(t)+\y(t))  - \mathcal{C}(\v(t)+\y(t)) ,\psi(t)  \big\rangle\right| \d t \\
			&\leq  \sup_{(t, x)\in [0, T] \times \mathcal{O}_1} |\psi(t, x)| \bigg[\int_{0}^{T} \|\v_m(t)- \v(t)\|_{{\L}^{2}(\mathcal{O}_1)} \|\v_m(t)+\y(t)\|^{r-1}_{{\L}^{2(r-1)}(\mathcal{O}_1)} \d t\\&\qquad\qquad+\int_{0}^{T} \|\v_m(t)- \v(t)\|_{{\L}^{2}(\mathcal{O}_1)} \|\v(t)+\y(t)\|^{r-1}_{{\L}^{2(r-1)}(\mathcal{O}_1)} \d t\bigg] \\
			&\leq\sup_{(t,x)\in[0,T]\times\mathcal{O}_1}|\psi(t,x)|\|\v_m-\v\|_{\mathrm{L}^{2}(0,T;{\L}^{2}(\mathcal{O}_1))} \bigg[  \|\v_m+\y\|^{r-1}_{\mathrm{L}^{2(r-1)}(0,T;{\L}^{2(r-1)}(\mathcal{O}_1))}\\&\qquad\qquad+\|\v+\y\|^{r-1}_{\mathrm{L}^{2(r-1)}(0,T;{\L}^{2(r-1)}(\mathcal{O}_1))} \bigg] \\&\leq C\sup_{(t,x)\in[0,T]\times\mathcal{O}_1}|\psi(t,x)|\|\v_m-\v\|_{\mathrm{L}^{2}(0,T;{\L}^{2}(\mathcal{O}_1))} \bigg[  \|\v_m+\y\|^{r-1}_{\mathrm{L}^{4}(0,T;{\L}^{4}(\mathcal{O}_1))}+\|\v+\y\|^{r-1}_{\mathrm{L}^{4}(0,T;{\L}^{4}(\mathcal{O}_1))} \bigg] \\
			&	\to   0 \ \text{ as }\   m \to \infty.
		\end{align*}
		Finally, we consider $r\geq3$.	Using Taylor's formula (\cite[Theorem 7.9.1]{PGC}), H\"older's and interpolation (see \eqref{Interpolation}) inequalities, we obtain 
		\begin{align*}
			& \left|\int_{0}^{T} \big\langle\mathcal{C}(\v_m(t)+\y(t)) ,\psi(t)  \big\rangle \d t - \int_{0}^{T} \big\langle\mathcal{C}(\v(t)+\y(t)) ,\psi(t)  \big\rangle \d t\right|\\&\leq \int_{0}^{T}\left| \big\langle\mathcal{C}(\v_m(t)+\y(t))  - \mathcal{C}(\v(t)+\y(t)) ,\psi(t)  \big\rangle\right| \d t \\
			&\leq  \sup_{(t, x)\in [0, T] \times \mathcal{O}_1} |\psi(t, x)| \bigg[\int_{0}^{T} \|\v_m(t)- \v(t)\|_{{\L}^{\frac{r+1}{2}}(\mathcal{O}_1)} \|\v_m(t)+\y(t)\|^{r-1}_{{\L}^{r+1}(\mathcal{O}_1)} \d t\\&\qquad\qquad+\int_{0}^{T} \|\v_m(t)- \v(t)\|_{{\L}^{\frac{r+1}{2}}(\mathcal{O}_1)} \|\v(t)+\y(t)\|^{r-1}_{{\L}^{r+1}(\mathcal{O}_1)} \d t\bigg] \\
			&\leq\sup_{(t,x)\in[0,T]\times\mathcal{O}_1}|\psi(t,x)|\|\v_m-\v\|^{\frac{2}{r-1}}_{\mathrm{L}^{2}(0,T;{\L}^{2}(\mathcal{O}_1))}\|\v_m-\v\|^{\frac{r-3}{r-1}}_{\mathrm{L}^{r+1}(0,T;{\L}^{r+1}(\mathcal{O}_1))}\\&\quad\times \bigg[  \|\v_m+\y\|^{r-1}_{\mathrm{L}^{r+1}(0,T;{\L}^{r+1}(\mathcal{O}_1))}+\|\v+\y\|^{r-1}_{\mathrm{L}^{r+1}(0,T;{\L}^{r+1}(\mathcal{O}_1))} \bigg] \\
			&	\leq C\sup_{(t,x)\in[0,T]\times\mathcal{O}_1}|\psi(t,x)|\|\v_m-\v\|^{\frac{2}{r-1}}_{\mathrm{L}^{2}(0,T;{\L}^{2}(\mathcal{O}_1))}\bigg[\|\v_m|^{\frac{r-3}{r-1}}_{\mathrm{L}^{r+1}(0,T;{\L}^{r+1}(\mathcal{O}_1))}+\|\v\|^{\frac{r-3}{r-1}}_{\mathrm{L}^{r+1}(0,T;{\L}^{r+1}(\mathcal{O}_1))}\bigg] \\
			&\quad\times\bigg[  \|\v_m+\y\|^{r-1}_{\mathrm{L}^{r+1}(0,T;{\L}^{r+1}(\mathcal{O}_1))}+\|\v+\y\|^{r-1}_{\mathrm{L}^{r+1}(0,T;{\L}^{r+1}(\mathcal{O}_1))} \bigg] \\
			&	\to   0 \ \text{ as }\   m \to \infty,
		\end{align*}
		which completes the proof. 
	\end{proof}
	\begin{corollary}\label{convergence_c3_1}
		If $\{\v_m\}_{m\in \mathbb{N}}$ is a bounded sequence in $\mathrm{L}^{\infty}(0, T; \H), \v \in \mathrm{L}^{\infty}(0, T; \H)$, $\v_m$ converges to $\v$ weakly and strongly in $\mathrm{L}^2(0, T; \V)\cap\mathrm{L}^{r+1} (0, T; \widetilde{\L}^{r+1})$ and $\mathrm{L}^2(0, T; {\L}^2_{\emph{loc}}(\mathcal{O}))$, respectively. Then for any $r\in[1,3)$ with $\y\in \mathrm{L}^{4} (0, T; \widetilde{\L}^{4}) \cap\mathrm{L}^2 (0, T; \H)$ and for any $r\geq3$ with $\y\in\mathrm{L}^{r+1} (0, T; \widetilde{\L}^{r+1})\cap \mathrm{L}^2 (0, T; \H)$, we have 
		\begin{align*}
			\int_{0}^{T} \big\langle\mathcal{C}(\v_m(t)+\y(t)) ,\y(t)  \big\rangle \d t \to \int_{0}^{T} \big\langle\mathcal{C}(\v(t)+\y(t)) ,\y(t)  \big\rangle \d t .
		\end{align*}
	\end{corollary}
	\begin{proof}
		From Remark \ref{RemarkI} , we infer that for $r\in[1,3]$, $\y \in \mathrm{L}^{2}(0,T;\H)\cap\mathrm{L}^{4}(0,T; \wi\L^{4})$ implies $\y \in\mathrm{L}^{r+1}(0,T; \wi\L^{r+1})$. Hence, $\y \in\mathrm{L}^{r+1}(0,T; \wi\L^{r+1}),$ for all $r\in[1,\infty)$. Moreover, we can find a constant $L>0$ such that $\|\v_m+\y\|^{r}_{\mathrm{L}^{r+1}(0,T;\widetilde{\L}^{r+1})}+\|\v+\y\|^{r}_{\mathrm{L}^{r+1}(0,T;\widetilde{\L}^{r+1})} \leq L.$ 	Let us choose an $\epsilon > 0$. Since $\y \in \mathrm{L}^{r+1} (0, T; \widetilde{\L}^{r+1}),$ for all $r\in[1,\infty)$,  by a standard regularization method, we can find a function $\psi_1$ satisfying the assumptions of  Lemma \ref{convergence_c2_1} such that $\left(\int_{0}^{T} \|\y(t)- \psi_1(t)\|^{r+1}_{\widetilde{\L}^{r+1}}\d t\right)^{\frac{1}{r+1}} < \frac{\epsilon}{3L}$. Hence, by Lemma \ref{convergence_c2_1}, we can find $M_{\epsilon} \in \mathbb{N}$ such that   $$\left|\int_{0}^{T} \big\langle\mathcal{C}(\v_m(t)+\y(t)) ,\psi_1(t)  \big\rangle \d t - \int_{0}^{T} \big\langle\mathcal{C}(\v(t)+\y(t)) ,\psi_1(t)  \big\rangle \d t\right| < \frac{\epsilon}{3},$$  for all $m \geq M_{\epsilon}.$ Hence,  for $m > M_{\epsilon}$, we have
		\begin{align*}
			&\left|\int_{0}^{T} \big\langle\mathcal{C}(\v_m(t)+\y(t)) ,\y(t)  \big\rangle \d t - \int_{0}^{T} \big\langle\mathcal{C}(\v(t)+\y(t)) ,\y(t)  \big\rangle \d t\right| \\
			&	\leq  \int_{0}^{T}\left| \big\langle\mathcal{C}(\v_m(t)+\y(t)) ,\y(t) - \psi_1(t)  \big\rangle\right|  \d t+ \int_{0}^{T} \left|\big\langle\mathcal{C}(\v(t)+\y(t)) ,\y(t) - \psi_1(t)  \big\rangle\right| \d t \\&\quad+ \left|\int_{0}^{T} \big\langle\mathcal{C}(\v_m(t)+\y(t)) ,\psi_1(t) \big\rangle \d t - \int_{0}^{T} \big\langle\mathcal{C}(\v(t)+\y(t)) ,\psi_1(t)  \big\rangle \d t\right|\\
			&	<  \frac{\epsilon}{3} + \int_{0}^{T} \left[\|\v_m(t)+\y(t)\|^r_{\widetilde{\L}^{r+1}}  + \|\v(t)+\y(t)\|^r_{\widetilde{\L}^{r+1}}\right] \|\y(t)-\psi(t)_1\|_{\widetilde{\L}^{r+1}}\d t \\
			&	\leq  \frac{\epsilon}{3}+ \left(\|\v_m+\y\|^{r}_{\mathrm{L}^{r+1}(0,T;\widetilde{\L}^{r+1})} +\|\v+\y\|^{r}_{\mathrm{L}^{r+1}(0,T;\widetilde{\L}^{r+1})}\right)\|\y-\psi_1\|_{\mathrm{L}^{r+1}(0,T;\widetilde{\L}^{r+1})}	\nonumber\\&<  \frac{\epsilon}{3}+\frac{2\epsilon}{3}=\epsilon,
		\end{align*}
		which completes the proof for all $r\in[1,\infty)$.
	\end{proof}
	\begin{theorem}[\cite{Mohan}]\label{LocMon}
		Let $d=2$ with $ r\in[1, 3]$, $d=2,3$ with $ r> 3$, $d=r=3$ with $2\beta\mu\geq1$ and $\u_1, \u_2 \in \V\cap\wi\L^{r+1}.$ Then, for the operator $\G(\u) = \mu\A\u +\B(\u)+\alpha\u+\beta\mathcal{C}(\u),$ we have 
		\begin{align}\label{fe2_1}
			\langle\G(\u_1)-\G(\u_2),\u_1-\u_2\rangle+ \frac{27}{32\mu ^3}\|\u_2\|^4_{\widetilde{\L}^4}\|\u_2-\u_2\|_{\H}^2&\geq 0, \text{ for } d=2 \text{ with } r\in[1,3],
		\end{align}
		\begin{align}\label{fe2_2}
			\langle\G(\u_1)-\G(\u_2),\u_1-\u_2\rangle+ \eta\|\u_2-\u_2\|_{\H}^2&\geq 0, \text{ for } d=2,3 \text{ with } r>3,
		\end{align}
		where $\eta=\frac{r-3}{2\mu(r-1)}\left(\frac{2}{\beta\mu (r-1)}\right)^{\frac{2}{r-3}}$ and 
		\begin{align}\label{fe2_3}
			\langle\G(\u_1)-\G(\u_2),\u_1-\u_2\rangle \geq 0, \text{ for } d=r=3 \text{ with } 2\beta\mu\geq1.
		\end{align} 
	\end{theorem}

	\subsection{A compact operator}(See \cite[Subsection 2.3]{BM} for more details)\label{C_O}
	Consider the natural embedding $j:\V\hookrightarrow\H$ and its adjoint $j^*:\H\hookrightarrow\V$. Since the range of $j$ is dense in $\H$, the map $j^*$ is one-to-one. Let us define
	\begin{align}\label{L1}
		\D(\mathcal{A})&:=j^*(\H)\subset\V,\nonumber \\
		\mathcal{A}\u&:=(j^*)^{-1}\u, \ \ \u\in\D(\mathcal{A}).
	\end{align}
Note that for all $\u\in\D(\mathcal{A})$ and $\v\in\V$
\begin{align*}
	(\mathcal{A}\u,\v)_{\H}=(\u,\v)_{\V}.
\end{align*}
Let us assume that $s>2$. It is clear that $\V_s$ is dense in $\V$ and the embedding $j_s:\V_s\hookrightarrow\V$ is continuous. Then, there exists a Hilbert space $\mathbb{U}$ (cf. \cite{HW}, \cite[Lemma C.1]{BM}) such that $\mathbb{U}\subset\V_s$, $\mathbb{U}$ is dense in $\V_s$ and 
\begin{align*}
	\text{  the natural embedding }\iota_s:\mathbb{U}\hookrightarrow\V_s \text{ is compact.}
\end{align*}
It implies that 
\begin{align*}
	\mathbb{U} \xhookrightarrow[\iota_s]{} \V_s\xhookrightarrow[j_s]{}\V\xhookrightarrow[j]{}\H\cong{\H}{'}\xhookrightarrow[j']{}\V'\xhookrightarrow[j'_s]{}  {\V}'_{s}\xhookrightarrow[\iota'_s]{}\mathbb{U}'.
\end{align*}
Consider the composition $$\iota:=j\circ j_s\circ\iota_s:\mathbb{U}\to\H$$ and its adjoint $$\iota^*:=(j\circ j_s\circ\iota_s)^*=\iota_s^*\circ j^*_s\circ j^*:\H\to \mathbb{U}.$$ We have that $\iota$ is compact and since its range is dense in $\H$, $\iota^*:\H\to\mathbb{U}$ is one-one. Let us define 
\begin{align}\label{L2}
	\D(\mathcal{L})&:=\iota^*(\H)\subset\mathbb{U},\nonumber\\
	\mathcal{L}\u&:=(\iota^*)^{-1}\u, \ \ \u\in\D(\mathcal{L}).
\end{align}
Also we have that $\mathcal{L}:\D(\mathcal{L})\to\H$ is onto, $\D(\mathcal{L})$ is dense in $\H$ and 
\begin{align*}
	(\mathcal{L}\u,\w)_{\H}=(\u,\w)_{\mathbb{U}}, \ \ \ \ \u\in\D(\mathcal{L}), \ \w\in\mathbb{U}.
\end{align*}
Furthermore, for $\u\in\D(\mathcal{L})$,
\begin{align*}
	\mathcal{L}\u=(\iota^*)^{-1}\u=(j^*)^{-1}\circ (j^*_s)^{-1}\circ(\iota^*_s)^{-1}\u=\mathcal{A}\circ (j^*_s)^{-1}\circ(\iota^*_s)^{-1}\u,
\end{align*}
where $\mathcal{A}$ is defined in \eqref{L1}. Since the operator $\mathcal{L}$ is self-adjoint and $\mathcal{L}^{-1}$ is compact, there exists an orthonormal basis $\{\boldsymbol{e}_i\}_{i\in\N}$ of $\H$ such that 
\begin{align}\label{L3}
	\mathcal{L}\boldsymbol{e}_i=\mu_i\boldsymbol{e}_i,\ \ \ \ i\in\N,
\end{align}
that is, $\boldsymbol{e}_i$ are the eigenfunctions $\mu_i$ are the corresponding eigenvalues of operator $\mathcal{L}$. Note that $\boldsymbol{e}_i\in\mathbb{U}$, \ $i\in\N$, because $\D(\mathcal{L})\subset \mathbb{U}$. 

Let us fix $m\in\N$ and let $\P_m$ be the operator from $\mathbb{U}'$ to $\mathrm{span}\{\boldsymbol{e}_1,\ldots,\boldsymbol{e}_m\}=:\H_{m}$ defined by 
\begin{align}
	\P_m\u^*:=\sum_{i=1}^{m}\langle\u^*,\boldsymbol{e}_i\rangle_{\mathbb{U}'\times\mathbb{U}}\boldsymbol{e}_i, \ \ \ \ \ \ \u^*\in\mathbb{U}'.
\end{align}
We will consider the restriction of operator $\P_m$ to the space $\H$ denoted still by the  same. In particular, we have $\H\hookrightarrow\mathbb{U}'$, that is, every element $\u\in\H$ induces a functional $\u^*\in\mathbb{U}'$ by 
\begin{align}
	\langle\u^*,\v\rangle_{\mathbb{U}'\times\mathbb{U}}:=(\u,\v), \ \ \ \ \v\in\mathbb{U}.
\end{align}
Thus the restriction of $\P_m$ to $\H$ is given by 
\begin{align}
	\P_m\u:=\sum_{i=1}^{m}(\u,\boldsymbol{e}_i)\boldsymbol{e}_i, \ \ \ \ \ \ \u\in\H.
\end{align}
Hence particularly, $\P_m$ is the orthogonal projection from $\H$ onto $\H_m$. 
\begin{lemma}[{\cite[Lemma 2.4]{BM}}]
	For every $\u\in\mathbb{U}$ and $s>2$, we have 
	\begin{itemize}
	\item [(i)] $\lim\limits_{m\to\infty}\|\P_m\u-\u\|_{\mathbb{U}}=0$,
	\item [(ii)] $\lim\limits_{m\to\infty}\|\P_m\u-\u\|_{\mathbb{V}_s}=0$, 
	\item [(iii)] $\lim\limits_{m\to\infty}\|\P_m\u-\u\|_{\mathbb{V}}=0$.
	\end{itemize}
\end{lemma}

	\subsection{Stochastic convective Brinkman-Forchheimer equations}
	In this subsection, we provide an abstract formulation of the system \eqref{SCBF} and assumptions on the noise.
	On taking projection $\mathcal{P}$ onto the first equation in \eqref{SCBF}, we obtain 
	\begin{equation}\label{S-CBF}
		\left\{
		\begin{aligned}
			\d\u(t)+\{\mu \A\u(t)+\B(\u(t))+\alpha\u(t)+\beta \mathcal{C}(\u(t))\}\d t&=\f \d t + \d\mathrm{W}(t), \ \ \ t> 0, \\ 
			\u(0)&=\boldsymbol{x},
		\end{aligned}
		\right.
	\end{equation}
 where $\boldsymbol{x}\in \H,\ \f\in \V'$ and $\{\mathrm{W}(t), \ t\in \R\},$ is a two-sided cylindrical Wiener process in $\H$ with its RKHS $\mathrm{K}. $  	For $d=2$ with $r\in[1,\infty)$, $d=3$ with $r\in[3,\infty)$ and $d=r=3$ with $2\beta\mu\geq1$,  RKHS $\mathrm{K}$   satisfies the following Assumptions \ref{assump1} (for $r\in [1,3)$) and \ref{assump2} (for $r\geq3$): 
	\begin{assumption}\label{assump1}
		For $r\in[1,3)$,	$ \mathrm{K} \subset \H \cap \widetilde{\L}^{4} $ is a Hilbert space such that for some $\delta\in (0, 1/2),$
		\begin{align}\label{A1}
			\A^{-\delta} : \mathrm{K} \to \H \cap \widetilde{\L}^{4}  \  \text{ is }\ \gamma \text{-radonifying.}
		\end{align}
	\end{assumption}
	\begin{assumption}\label{assump2}
		For $r\geq3$,	$ \mathrm{K} \subset \H \cap \widetilde{\L}^{r+1} $ is a Hilbert space such that for some $\delta\in (0, 1/2),$
		\begin{align}\label{A2}
			\A^{-\delta} : \mathrm{K} \to \H \cap \widetilde{\L}^{r+1} \   \text{ is }\ \gamma \text{-radonifying}.
		\end{align}
	\end{assumption}
	\begin{remark}
		1.	Let $\mathrm{K}$ be a separable Hilbert space and $\mathrm{X}$ be a separable Banach space. We denote by $\gamma(\mathrm{K}, \mathrm{X}),$ the completion of the finite rank operators from $\mathrm{K}$ to $\mathrm{X}$ with respect to the norm 
		\begin{align}\label{gamma}
			\left\|\sum_{i=1}^{k} h_i \otimes x_i  \right\|_{\gamma(\mathrm{K}, \mathrm{X})} := \left(\mathbb{E} \left\|\sum_{i=1}^{k} \gamma_i x_i  \right\|_{\mathrm{X}}^2\right)^{1/2},
		\end{align}
		where $h_1, \ldots, h_k$ are orthogonal in $\mathrm{K}$ and $\{\gamma_i\}_{i\geq 1}$ is a sequence of independent standard Gaussian random variables defined on some probability space $(\Omega, \mathcal{F}, \mathbb{P}).$ A linear operator $L: \mathrm{K} \to \mathrm{X}$ belongs to $\gamma(\mathrm{K}, \mathrm{X})$ is called \emph{$\gamma$-radonifying}. Also, $\gamma(\mathrm{K}, \mathrm{X})$ is a separable Banach space as well as an operator ideal (cf. \cite{BH}).
		
		2. The conditions \eqref{A1} and \eqref{A2} mean that the operator $\A^{-\delta}:\mathrm{K}\to\H$ is Hilbert-Schmidt and $\A^{-\delta}:\mathrm{K}\to  \H\cap\wi\L^{r+1}$ ($r\geq 3$) is $\gamma$-radonifying. 
		
		3. Because $\A^{-s}$ is bounded operator in $\H\cap\widetilde{\L}^{r+1}$ ($r\geq3$), for $s>0$, if the conditions \eqref{A1} and \eqref{A2} are satisfied for some $\delta_1$, then it is also satisfied for any $\delta_2\geq\delta_1$ (using ideal property).
		
		4. Let us fix $p\in(1,\infty)$. Let $(X_i, \mathcal{A}_i, \nu_i),\ i=1,2,$ be $\sigma$-finite measure spaces. A bounded linear operator $R:\mathrm{L}^2(X_1)\to\mathrm{L}^p(X_2)$ is $\gamma$-radonifying if and only if  there exists  a measurable function $\kappa:X_1\times X_2\to \R$ such that $\int_{X_2}\big[\int_{X_1}|\kappa(x_1,x_2)|^2\d\nu_1(x_1)\big]^{p/2}\d\nu_2(x_2)<\infty,$ and for all $\nu_2$-almost  all $x_2\in X_2, (R(f))(x_2)=\int_{X_1} \kappa(x_1,x_2)f(x_1)\d\nu_1(x_1), f\in \mathrm{L}^2(X_1)$ (cf. \cite[Theorem 2.3]{BN}). Thus, it can be easily seen that if $\mathcal{O}$ is a bounded domain, then $\A^{-s}:\H\to\widetilde{\L}^p$ is $\gamma$-radonifying if and only if $\int_{\mathcal{O}}\big[\sum_{j} \lambda_j^{-2s}|e_j(x)|^2\big]^{p/2}\d x<\infty,$ where $\{e_j\}$ is an orthogonal basis of $\H$ and $\A e_j=\lambda_j e_j, j\in \N.$ 
		\begin{itemize}
			\item [(i)] In 2D bounded domains, we know that $\lambda_j\sim j$ and hence $\A^{-s}$ is $\gamma$-radonifying if and only if $s>\frac{1}{2}.$ In other words, with $\mathrm{K}=\D(\A^{s}),$ the embedding $\mathrm{K}\hookrightarrow\H\cap\widetilde{\L}^{r+1} (r\geq3)$ is $\gamma$-radonifying if and only if $s>\frac{1}{2}.$ Thus, Assumptions \ref{assump1} and  \ref{assump2} are satisfied for any $\delta>0.$ In fact, the conditions \eqref{A1} (for $r\in[1,3)$) and \eqref{A2} (for $r\geq3$) hold if and only if the operator $\A^{-(s+\delta)}:\H\to \H\cap\widetilde{\L}^{r+1}$ is $\gamma$-radonifying.
			\item [(ii)] In 3D bounded domains, we know that $\lambda_j\sim j^{2/3},$ for large $j$ (growth of eigenvalues, see \cite{FMRT}) and hence $\A^{-s}$ is Hilbert-Schmidt if and only if $s>\frac{3}{4}.$ In other words, with $\mathrm{K}=\D(\A^{s}),$ the embedding $\mathrm{K}\hookrightarrow\H\cap\widetilde{\L}^{r+1} (r\geq3)$ is $\gamma$-radonifying if and only if $s>\frac{3}{4}.$ Thus, Assumption \ref{assump2} is satisfied for any $\delta>0.$ In fact, the condition \eqref{A1} holds if and only if the operator $\A^{-(s+\delta)}:\H\to \H\cap\widetilde{\L}^{r+1} (r\geq3)$ is $\gamma$-radonifying.
		\end{itemize}

		5. The requirement of  $\delta<\frac{1}{2}$ in Assumptions \ref{assump1} and \ref{assump2} is necessary because we need (see subsection \ref{O-Up}) the corresponding Ornstein-Uhlenbeck process to take values in $\H\cap\widetilde{\L}^{4}$ and $\H\cap\widetilde{\L}^{r+1}$, for $r\in[1,3)$ and $r\geq 3$,  respectively.
	\end{remark}

	\section{RDS generated by SCBF equations}\label{sec5}\setcounter{equation}{0} In this section, we discuss the random dynamical system generated by SCBF equations. Let us represent $\mathrm{X}_1 = \H \cap \widetilde{\mathbb{L}}^{4} $ and $\mathrm{X}_2 = \H \cap \widetilde{\mathbb{L}}^{r+1} $. Let $\mathrm{E}_i$ denote the completion of $\A^{-\delta}(\mathrm{X}_i)$ with respect to the graph norm $\|x_i\|_{\mathrm{E}_i}=\|\A^{-\delta} x_i\|_{\mathrm{X}_i}, \text{ for } x_i\in \mathrm{X}_i , \ i\in\{1,2\}$,  where $ \|\cdot\|_{\mathrm{X}_1} = \|\cdot\|_{\H} +  \|\cdot\|_{\widetilde{\L}^4 }  \text{ and } \|\cdot\|_{\mathrm{X}_2} = \|\cdot\|_{\H} + \|\cdot\|_{\widetilde{\L}^{r+1} }.$ Note that $\mathrm{E}_1$ and $\mathrm{E}_2$ are separable Banach spaces (cf. \cite{Brze2}).
	
	For $\xi \in(0, 1/2)$, let us set 
	$$ \|\omega\|_{C^{\xi}_{1/2} (\mathbb{R}, \mathrm{E}_i)} = \sup_{t\neq s \in \mathbb{R}} \frac{\|\omega(t) - \omega(s)\|_{\mathrm{E}_i}}{|t-s|^{\xi}(1+|t|+|s|)^{1/2}}, \ i\in\{1,2\}.$$
	Furthermore, we define
	\begin{align*}
		C^{\xi}_{1/2} (\mathbb{R}, \mathrm{E}_i) &= \left\{ \omega \in C(\mathbb{R}, \mathrm{E}_i) : \omega(0)=0,\ \|\omega\|_{C^{\xi}_{1/2} (\mathbb{R}, \mathrm{E}_i)} < \infty \right\},\\ \Omega(\xi, \mathrm{E}_i)&=\text{the closure of }\  \{ \omega \in C^\infty_0 (\mathbb{R}, \mathrm{E}_i) : \omega(0) = 0 \}\ \text{ in } \ C^{\xi}_{1/2} (\mathbb{R}, \mathrm{E}_i).
	\end{align*}
	The space $\Omega(\xi, \mathrm{E}_i)$ is a separable Banach space. We also define
	$$C_{1/2} (\mathbb{R}, \mathrm{E}_i) = \left\{ \omega \in C(\mathbb{R}, \mathrm{E}_i) : \omega(0)=0, \|\omega\|_{C_{1/2} (\mathbb{R}, \mathrm{E}_i)} = \sup_{t \in \mathbb{R}} \frac{\|\omega(t) \|_{\mathrm{E}_i}}{1+|t|^{1/2}} < \infty \right\}.$$

	Let us  denote $\mathcal{F}_i$ for the Borel $\sigma$-algebra on $\Omega(\xi, \mathrm{E}_i).$ For $\xi\in (0, 1/2)$, there exists a Borel probability measure $\mathbb{P}_i$ on $\Omega(\xi, \mathrm{E}_i)$ (cf. \cite{Brze}) such that the canonical process $\{w^i_t, \ t\in \mathbb{R}\}$ is defined by 
	\begin{align}\label{Wp}
		w^i_t(\omega) := \omega(t), \ \ \ \omega \in \Omega(\xi, \mathrm{E}_i),
	\end{align}
	is an $\mathrm{E}_i$-valued two sided Wiener process such that the RKHS of the Gaussian measure $\mathscr{L}(w_1)$ on $\mathrm{E}_i$ is $\mathrm{K}$. For $t\in \mathbb{R},$ let $\mathcal{F}^i_t := \sigma \{ w^i_s : s \leq t \}.$ Then  there exists a unique bounded linear map $\mathrm{W}^i(t): \mathrm{K} \to \mathrm{L}^2(\Omega(\xi, \mathrm{E}_i), \mathcal{F}^i_t  ,  \mathbb{P}_i).$ Moreover, the family $(\mathrm{W}^i(t))_{t\in \mathbb{R}}$ is a $\mathrm{K}$-cylindrical Wiener process on a filtered probability space $(\Omega(\xi, \mathrm{E}_i), \mathcal{F}_i, (\mathcal{F}^i_t)_{t \in \mathbb{R}} , \mathbb{P}_i)$ (cf. \cite{BP} for more details).
	
	We consider a flow $\theta = (\theta_t)_{t\in \mathbb{R}}$, on the space $C_{1/2} (\mathbb{R}, \mathrm{E}_i),$  defined by
	$$ \theta_t \omega(\cdot) = \omega(\cdot + t) - \omega(t), \ \ \ \omega\in C_{1/2} (\mathbb{R}, \mathrm{E}_i), \ \ t\in \mathbb{R}.$$ 
	This flow keeps the spaces $C^{\xi}_{1/2} (\mathbb{R}, \mathrm{E}_i)$ and $\Omega(\xi, \mathrm{E}_i)$ invariant and preserves $\mathbb{P}_i.$

	\subsection{Analytic preliminaries} Let us first recall some analytic preliminaries from \cite{BL} which will help us to define an Ornstein-Uhlenbeck process and all the results of this subsection are valid for the space $C^{\xi}_{1/2} (\mathbb{R}, \mathrm{Y})$ replaced by $\Omega(\xi, \mathrm{Y}).$ 
	\begin{proposition}[{\cite[Proposition 2.11]{BL}}]\label{Ap}
		Let us assume that $A$ is the generator of an analytic semigroup $\{e^{-tA}\}_{t\geq 0}$ on a separable Banach space $\mathrm{Y}$ such that for some $C>0\ \text{and}\ \gamma>0$
		\begin{align}
			\| A^{1+\delta}e^{-tA}\|_{\mathfrak{L}(\mathrm{Y})} \leq C t^{-1-\delta} e^{-\gamma t}, \ \ t\geq 0,
		\end{align}
	where $\mathfrak{L}(\mathrm{Y})$ denotes the space of all bounded linear operators from $\mathrm{Y}$ to $\mathrm{Y}$.	For $\xi \in (\delta, 1/2)$ and $\widetilde{\omega} \in  C^{\xi}_{1/2} (\mathbb{R}, \mathrm{Y}),$ we define 
		\begin{align}
			\hat{z}(t) = \hat{z} (A; \widetilde{\omega})(t) := \int_{-\infty}^{t} A^{1+\delta} e^{-(t-r)A} (\widetilde{\omega}(t) - \widetilde{\omega}(r))\d r, \ \ t\in \mathbb{R}.
		\end{align}
		If $t\in \mathbb{R},$ then $\hat{z}(t)$ is a well-defined element of $\mathrm{Y}$ and the mapping 
		$$C^{\xi}_{1/2} (\mathbb{R}, \mathrm{Y}) \ni \widetilde{\omega}  \mapsto \hat{z}(t) \in \mathrm{Y} $$
		is continuous. Moreover, the map $\hat{z} :  C^{\xi}_{1/2} (\mathbb{R}, \mathrm{Y}) \to  C_{1/2} (\mathbb{R}, \mathrm{Y})$  is well defined, linear and bounded. In particular, there exists a constant $C < \infty$ such that for any $\widetilde{\omega} \in C^{\xi}_{1/2} (\mathbb{R}, \mathrm{Y})$ 
		\begin{align}\label{X_bound_of_z}
			\|\hat{z}(\widetilde{\omega})(t)\|_{\mathrm{Y}} \leq C(1 + |t|^{1/2})\|\widetilde{\omega}\|_{C^{\xi}_{1/2} (\mathbb{R}, \mathrm{Y})}, \ \ \ t \in \R.
		\end{align}
		Furthermore, under the same assumption, following results hold (Corollary 6.4, Theorem 6.6, Corollary 6.8 in \cite{BL}):
		\begin{itemize}
			\item [1.]For all $-\infty<a<b<\infty$ and $t\in \R$, the map 
			\begin{align}\label{O-U_conti}
				C^{\xi}_{1/2} (\mathbb{R}, \mathrm{Y}) \ni \widetilde{\omega} \mapsto (\hat{z}(\widetilde{\omega})(t), \hat{z}(\widetilde{\omega})) \in \mathrm{Y} \times \mathrm{L}^{q} (a, b; \mathrm{Y})
			\end{align}
			where $q\in [1, \infty]$, is continuous.
			\item [2.] For any $\omega \in C^{\xi}_{1/2} (\mathbb{R}, \mathrm{Y}),$
			\begin{align}\label{stationary}
				\hat{z}(\theta_s \omega)(t) = \hat{z}(\omega)(t+s), \ \ t, s \in \mathbb{R}.
			\end{align}
			\item [3.] For $\zeta \in C_{1/2}(\R, \mathrm{Y}),$ if we put $(\tau_s\zeta(t))=\zeta(t+s), \ t,s \in \R,$ then, for $t \in \R ,\  \tau_s \circ \hat{z} = \hat{z}\circ\theta_s$, that is, 
			\begin{align}\label{IS}
				\tau_s\big(\hat{z}(\omega)\big)= \hat{z}\big(\theta_s(\omega)\big), \ \ \ \omega\in C^{\xi}_{1/2} (\mathbb{R}, \mathrm{Y}).
			\end{align}
		\end{itemize} 
	\end{proposition}

	\subsection{Ornstein-Uhlenbeck process}\label{O-Up}
	In this subsection, we define the Ornstein-Uhlenbeck processes under Assumptions \ref{assump1} and \ref{assump2}.	For $\delta$ as in Assumptions \ref{assump1} and \ref{assump2}, $\mu, \alpha, \beta > 0, \ \chi \geq 0, \ \xi \in (\delta, 1/2)$ and $ \omega \in C^{\xi}_{1/2} (\mathbb{R}, \mathrm{E}_i)$ (so that $(\mu \A + \chi I)^{-\delta}\omega \in C^{\xi}_{1/2} (\mathbb{R}, \mathrm{X}_i)$), we define $$ \z_{\chi}(\omega) := \hat{z}((\mu \A + \chi I); (\mu \A + \chi I)^{-\delta}\omega) \ \in C_{1/2}(\mathbb{R}, \mathrm{X}_i),$$ for $i=1,2$, that is, for any $t\geq 0,$ 
	\begin{align}\label{DOu1}
		\z_{\chi}(\omega)(t)&:=\int_{-\infty}^{t} (\mu \A + \chi I)^{1+\delta} e^{-(t-\tau)(\mu \A + \chi I)} [(\mu \A + \chi I)^{-\delta}\omega(t) - (\mu \A + \chi I)^{-\delta}\omega(\tau)]\d \tau  \nonumber\\
		&=\int_{-\infty}^{t} (\mu \A + \chi I)^{1+\delta} e^{-(t-\tau)(\mu \A + \chi I)} ((\mu \A + \chi I)^{-\delta}\theta_{\tau} \omega)(t-\tau)\d \tau.
	\end{align}
	For $\omega \in C^{\infty}_0 (\mathbb{R}, \mathrm{E}_i)$ with $\omega(0)= 0,$ using  integration by parts, we obtain 
	\begin{align*}
		\frac{\d\z_\chi(t)}{\d t} &= -(\mu \A + \chi I )\int_{-\infty}^{t} (\mu \A + \chi I)^{1+\delta} e^{-(t-r)(\mu \A + \chi I)} [(\mu \A + \chi I)^{-\delta}\omega(t) \\&\qquad\qquad - (\mu \A + \chi I)^{-\delta}\omega(r)]\d r +  \dot{\omega}(t).
	\end{align*}
	Thus $\z_{\chi}(\cdot)$ is the solution of the following equation:
	\begin{align}\label{OuE1}
		\frac{\d\z_{\chi} (t)}{\d t} + (\mu \A + \chi I)\z_{\chi}(t) = \dot{\omega} (t), \ \ t\in \mathbb{R}.
	\end{align}
	Therefore, from the definition of the space $\Omega(\xi, \mathrm{E}_i),$ we have 
	\begin{corollary}\label{Diff_z1}
		If $\chi_1, \chi_2 \geq 0,$ then the difference $\z_{\chi_1} - \z_{\chi_2}$ is a solution to 
		\begin{align}\label{Dif_z1}
			\frac{\d(\z_{\chi_1} - \z_{\chi_2})(t)}{\d t} + \mu\A(\z_{\chi_1} - \z_{\chi_2})(t) = -(\chi_1 \z_{\chi_1} - \chi_2\z_{\chi_2})(t), \ \ \ t \in \R.
		\end{align}
	\end{corollary}
	
	According to the  definition \eqref{Wp} of Wiener process $\{w^i_t, \ t\in \R\},$ one can view the formula \eqref{DOu1} as a definition of a process $\{\z_{\chi}(t), \ t\in \R\},$ on the probability space $(\Omega(\xi, \mathrm{E}_i), \mathcal{F}_i, \mathbb{P}_i),$ for $i=1,2$. Equation \eqref{OuE1} clearly tells that the process $\z_{\chi}(\cdot)$ is an Ornstein-Uhlenbeck process. Furthermore, the following results hold for $\z_{\chi}(\cdot)$.
	\begin{proposition}[{\cite[Proposition 6.10]{BL}}]\label{SOUP1}
		The process $\{\z_{\chi}(t), \ t\in \mathbb{R}\},$ is a stationary Ornstein-Uhlenbeck process on $(\Omega(\xi, \mathrm{E}_i), \mathcal{F}_i, \mathbb{P}_i)$, for $i=1,2$. It is a solution of the equation 
		\begin{align}\label{OUPe1}
			\d\z_{\chi}(t) + (\mu \A + \chi I)\z_{\chi} \d t = \d\mathrm{W}(t), \ \ t\in \mathbb{R},
		\end{align}
		that is, for all $t\in \mathbb{R},$ $\mathbb{P}_i$-a.s.
		\begin{align}\label{oup1}
			\z_\chi (t) = \int_{-\infty}^{t} e^{-(t-\xi)(\mu \A + \chi I)} \d\mathrm{W}(\xi),
		\end{align}
		where the integral is an It\^o integral on the M-type 2 Banach space $\mathrm{X}_i$  (cf. \cite{Brze1}). 	In particular, for some $C$ depending on $\mathrm{X}_i$,
		\begin{align}\label{E-OUP1}
			\mathbb{E}\left[\|\z_{\chi} (t)\|^2_{\mathrm{X}_i} \right]&= \mathbb{E}\left[\left\|\int_{-\infty}^{t} e^{-(t-\xi)(\mu \A + \chi I)} \d\mathrm{W}(\xi)\right\|^2_{\mathrm{X}_i}\right] \leq C\int_{-\infty}^{t} \|e^{-(t-\xi)(\mu \A +  \chi I)}\|^2_{\gamma(\mathrm{K},\mathrm{X}_i)} \d \xi \nonumber\\&=C \int_{0}^{\infty}  e^{-2\chi \xi} \|e^{-\mu \xi \A}\|^2_{\gamma(\mathrm{K},\mathrm{X}_i)} \d \xi.
		\end{align} 
		Moreover, $\mathbb{E}\left[\|\z_{\chi} (t)\|^2_{\mathrm{X}_i}\right]\to 0$ as $\chi \to \infty.$
	\end{proposition}
	Since $\z_{\chi}(t)$ is a Gaussian random vector, by the Burkholder inequality (cf. \cite{Ondrejat}), for each $p\geq 2,$ there exists a constant $C_p > 0$ such that 
	\begin{align}
		\mathbb{E}\left[\|\z_{\chi}(t)\|^p_{\mathrm{X}_i}\right] \leq C_p \left(\mathbb{E}\left[\|\z_{\chi}(t)\|^2_{\mathrm{X}_i}\right]\right)^{p/2},
	\end{align}	
	and thus
	\begin{align}\label{alpha_con1}
		\mathbb{E}\left[\|\z_{\chi}(t)\|^{p}_{\mathrm{X}_i}\right]  \to  0\ \text{ as } \ \chi\to \infty.
	\end{align}
	Using  Proposition \ref{E-OUP1}, the process $\{\z_{\chi}(t), \ t\in \R \}$ is an $\mathrm{X}_i$-valued stationary and ergodic process. Hence, by the strong law of large numbers (cf. \cite{DZ}), we have
	\begin{align}
		\lim_{t \to \infty} \frac{1}{t} \int_{-t}^{0} \|\z_{\chi}(s)\|^{4}_{\mathrm{X}_1} \d s = \mathbb{E} \left[\|\z_{\chi}(0)\|^{4}_{\mathrm{X}_1}\right], \ \ \ \ \mathbb{P}_1\text{-a.s. on } \ C^{\xi}_{1/2}(\R, \mathrm{X}_1).\label{SLLN1}
	\end{align}
	Moreover, from \eqref{alpha_con1}, we can find a $\chi_0$ such that for all $\chi \geq \chi_0,$ 
	\begin{align}\label{Bddns_1}
		\mathbb{E}\left[\|\z_{\chi} (0)\|^{4}_{\mathrm{X}_1}\right] \leq \frac{\alpha}{R},
	\end{align}
	where $R=\frac{729 }{8 \mu^3}$ and $\alpha>0$ is the Darcy  constant.
	
	Denote by $\Omega_{\chi}(\xi, \mathrm{E}_1),$ the set of those $\omega\in \Omega(\xi, \mathrm{E}_1)$ for which equality \eqref{SLLN1} is satisfied. The set $\Omega_{\chi}(\xi, \mathrm{E}_1)$ has full measure.  Therefore, we fix $\xi \in (\delta, 1/2)$ and set $$\Omega_1 :=  \bigcap^{\infty}_{n=0} \Omega_{n}(\xi, \mathrm{E}_1).$$
	Also, we denote $\Omega_2:=\Omega(\xi, \mathrm{E}_2)$. Furthermore, in view of \eqref{IS}, the sets $\Omega_i,$ $i=1,2,$ are invariant with respect to the flow $\theta$.

	We take the quadruple $(\Omega_i, \hat{\mathcal{F}}_i, \hat{\mathbb{P}}_i, \hat{\theta})$ as a model of an MDS, where $\hat{\mathcal{F}}_i$, $\hat{\mathbb{P}}_i$, $\hat{\theta}$ are the natural restrictions of $\mathcal{F}_i$, $\mathbb{P}_i$ and $\theta$ to $\Omega_i$, respectively. The reason to take $(\Omega_i, \hat{\mathcal{F}}_i, \hat{\mathbb{P}}_i, \hat{\theta})$ as a model of MDS will be cleared later.
	\begin{proposition}\label{m-DS1}
		The quadruple $(\Omega_i, \hat{\mathcal{F}}_i, \hat{\mathbb{P}}_i, \hat{\theta})$ is an MDS. 
	\end{proposition}
	Let us now provide  an  important consequence of the above arguments.
	\begin{corollary}\label{Bddns1_1}
		For each $\omega\in \Omega_1$, there exists $t_0=t_0 (\omega) \geq 0 $, such that 
		\begin{align*}
			R \int_{-t}^{0} \|\z_{\chi}(s)\|^{4}_{\mathrm{X}_1} \d s \leq \alpha t, \ \ \ t\geq t_0.
		\end{align*}
		Also, since the embedding $\mathrm{X}_1\hookrightarrow \widetilde{\L}^4$ is a contraction, we have 
		\begin{align*}
			R \int_{-t}^{0} \|\z_{\chi}(s)\|^{4}_{\widetilde{\L}^4} \d s \leq  \alpha t, \ \ \ t\geq t_0.
		\end{align*}
	\end{corollary}

	\subsection{Random dynamical system}
	Remember that Assumptions \ref{assump1} (for $r\in[1,3)$) and  \ref{assump2} (for $r\geq3$) are satisfied and that $\delta$ has the property stated there. Let us fix $\mu,\alpha,\beta> 0$, and the parameters $\chi\geq 0$ and  $\xi \in (\delta, 1/2)$.
	
	Let us denote $\v^{\chi}(t)=\u(t) - \z_{\chi}(\omega)(t)$. For convenience, we write $\v^{\chi}(t)=\v(t)$ and $\z_{\chi}(\omega)(t)=\z(t)$. Then $\v(\cdot)$ satisfies the following system:
	\begin{equation}\label{cscbf}
		\left\{
		\begin{aligned}
			\frac{\d\v}{\d t} &= -\mu \A\v - \B(\v + \z)-\alpha\v - \beta \mathcal{C}(\v + \z) + (\chi-\alpha) \z + \f, \\
			\v(0)&= \boldsymbol{x} - \z_{\chi}(0).
		\end{aligned}
		\right.
	\end{equation}
	Since $\z_{\chi}(\omega) \in C_{1/2} (\mathbb{R}, \mathrm{X}_i), $ then $\z_{\chi}(\omega)(0)$ is a well defined element of $\H$. Let us now provide the definition of weak solution (in the deterministic sense, for each  fixed $\omega$) for \eqref{cscbf}.
	\begin{definition}\label{defn5.9}
		Assume that $\x \in \H$, $\f\in \V'$, $\z\in\mathrm{L}^2_{\emph{loc}}([0,\infty);\H)\cap\mathrm{L}^4_{\emph{loc}}([0,\infty);\widetilde{\L}^4)$ (for $r\in[1,3)$) and $\z\in\mathrm{L}^2_{\emph{loc}}([0,\infty);\H)\cap\mathrm{L}^{r+1}_{\emph{loc}}([0,\infty);\widetilde{\L}^{r+1})$ (for $r\geq3$). A function $\v(\cdot)$ is called a \emph{weak solution} of the system \eqref{cscbf} on time interval $[0, \infty)$, if $$\v\in  \mathrm{C}([0,\infty); \H) \cap \mathrm{L}^{2}_{\emph{loc}}([0,\infty); \V)\cap\mathrm{L}^{\r+1}_{\emph{loc}}([0,\infty); \widetilde{\L}^{r+1}),$$   $$\frac{\d\v}{\d t}\in\mathrm{L}^{2}_{\emph{loc}}([0,\infty);\V')+\mathrm{L}^{\frac{r+1}{r}}_{\emph{loc}}([0,\infty);\widetilde{\L}^{\frac{r+1}{r}}),$$ and it satisfies 
		\begin{itemize}
			\item [(i)] for any $\psi\in \V\cap\wi\L^{r+1},$ 
			\begin{align}\label{W-CSCBF}
				\left<\frac{\d\v(t)}{\d t}, \psi\right>&=  - \left\langle \mu \A\v(t)-\alpha\v(t)+\B(\v(t)+\z(t))+\beta \mathcal{C}(\v(s)+\z(t)) , \psi \right\rangle \nonumber\\&\quad+\left\langle(\chi-\alpha)\z(t)- \f , \psi \right\rangle,
			\end{align}
			for a.e. $t\in[0,\infty);$
			\item [(ii)] the initial data:
			$$\v(0)=\x-\z(0) \ \text{ in }\ \H;$$
			\item [(iii)] the energy equality:
			\begin{align}\label{eeq}
				&\|\v(t)\|_{\H}^2+2\mu\int_0^t\|\v(\zeta)\|_{\V}^2\d\zeta+2\alpha\int_0^t\|\v(\zeta)\|_{\H}^2\d\zeta\nonumber\\&= \|\boldsymbol{x}-\z(0)\|_{\H}^2-2\int_0^t\langle\B(\v(\zeta)+\z(\zeta)),\v(\zeta)\rangle\d \zeta -2\beta\int_0^t\langle\mathcal{C}(\v(\zeta)+\z(\zeta)),\v(\zeta)\rangle\d s\nonumber\\&\quad+2\int_0^t\langle\f,\v(\zeta)\rangle\d \zeta+2(\chi-\alpha)\int_0^t(\z(\zeta),\v(\zeta))\d \zeta, 	\text{	for all } t\in[0,T],
			\end{align}
			for $0<T<\infty$.
		\end{itemize}
	\end{definition}

	\begin{theorem}\label{solution}
		For the cases given in Table \ref{Table1}, let $\mathcal{O}$ satisfy  Assumption \ref{assumpO}, $\chi\geq0$, $\x \in \H$, $\f\in \V'$, $\z\in\mathrm{L}^2_{\emph{loc}}([0,\infty);\H)\cap\mathrm{L}^{4}_{\emph{loc}}([0,\infty);\widetilde{\L}^{4})$ (for $r\in[1,3)$) and $\z\in\mathrm{L}^2_{\emph{loc}}([0,\infty);\H)\cap\mathrm{L}^{r+1}_{\emph{loc}}([0,\infty);\widetilde{\L}^{r+1})$ (for $r\geq3$). Then there exists a unique solution $\v(\cdot)$ to the system \eqref{cscbf} in the sense of Definition \ref{defn5.9}.
	\end{theorem}
	\begin{proof}
		Let us fix $T>0$. In order to complete the proof on the interval $[0,\infty)$, it is enough to prove on the  interval $[0,T]$. 
		\vskip 2mm
		\noindent
		\textbf{Step I.} \emph{Existence of weak solutions.} Let us consider the following approximate equation for the system \eqref{cscbf} on the finite dimensional space $\H_n$ (see Subsection \ref{C_O}):
		\begin{equation}\label{cscbf_n}
			\left\{
			\begin{aligned}
				\frac{\d\v^n}{\d t} &= \P_n\bigg[-\mu \A\v^n -\alpha\v^n- \B(\v^n+\z)- \beta \mathcal{C}(\v^n + \z) + (\chi-\alpha)\z + \f\bigg], \\
				\v^n(0)&= \P_n[\x-\z(0)]:={\v_0}_n.
			\end{aligned}
			\right.
		\end{equation}
		We define $\A_n\v^n=\P_n\A\v^n$, $\B_n(\v^n+\z)=\mathrm{P}_n\B(\v^n+\z)$ and $\mathcal{C}_n(\v^n+\z)=\mathrm{P}_n\mathcal{C}(\v^n+\z)$ and consider the following system of ODEs:
		\begin{equation}\label{finite-dimS}
			\left\{
			\begin{aligned}
				\frac{\d\v^n(t)}{\d t}&=-\mu \A_n\v^n(t)-\alpha\v^n(t)-\B_n(\v^n(t)+\z(t))-\beta\mathcal{C}_n(\v^n(t)+\z(t))\\ &\quad+(\chi-\alpha)\z_n(t)+\f_n,\\
				\v^n(0)&={\v_0}_n.
			\end{aligned}
			\right.
		\end{equation}
		Since $\B_n(\cdot)$ and $\mathcal{C}_n(\cdot)$ are  locally Lipschitz, the system (\ref{finite-dimS}) has a unique local solution $\v^n\in\mathrm{C}([0,T^*];\H_n)$, for some $0<T^*<T$. The following a priori estimates show that the time $T^*$ can be extended to time $T$. Taking the inner product with $\v^n(\cdot)$ to the first equation of \eqref{cscbf_n}, we obtain
		\begin{align}\label{S1}
			\frac{1}{2}\frac{\d}{\d t}\|\v^n(t)\|^2_{\H}&=-\mu\|\v^n(t)\|^2_{\V}-\alpha\|\v^n(t)\|^2_{\H}-\beta\|\v^n(t)+\z(t)\|^{r+1}_{\wi\L^{r+1}}\nonumber\\&\quad+\beta\left(\mathcal{C}(\v^n(t)+\z(t)),\z(t)\right)+b(\v^n(t)+\z(t),\v^n(t),\z(t))\nonumber\\&\quad+ ((\chi-\alpha)\z(t),\v^n(t))+\langle\f,\v^n(t)\rangle.
		\end{align}
		Next, we estimate each term of the right hand side of \eqref{S1} as 
		\begin{align}
			|b(\v^n+\z, \v^n, \z)| & \leq  \begin{cases}
				\|\v^n+\z\|_{\widetilde{\L}^{r+1}} \|\v^n\|_{\V} \|\z\|_{\widetilde{\L}^{\frac{2(r+1)}{r-1}}},& \text{ for } d=2,3 \text{ with } r\geq3\nonumber\\
				\|\v^n+\z\|_{\widetilde{\L}^{4}} \|\v^n\|_{\V} \|\z\|_{\widetilde{\L}^{4}}, &\text{ for } d=2  \text{ with } r\in[1,3),
			\end{cases}\\
		&\leq \begin{cases}
				\|\v^n+\z\|_{\widetilde{\L}^{r+1}} \|\v^n\|_{\V} \|\z\|^{\frac{2}{r-1}}_{\widetilde{\L}^{r+1}}\|\z\|^{\frac{r-3}{r-1}}_{\H}, &\text{ for } d=2,3 \text{ with } r\geq3,\nonumber\\
				\|\v^n+\z\|_{\widetilde{\L}^{4}} \|\v^n\|_{\V} \|\z\|_{\widetilde{\L}^{4}}, &\text{ for } d=2  \text{ with } r\in[1,3),
			\end{cases}\\
			|b(\v^n+\z, \v^n, \z)|&\leq\begin{cases}
				\frac{\beta}{4}\|\v^n+\z\|^{r+1}_{\widetilde{\L}^{r+1}}+\frac{\mu}{4} \|\v^n\|^2_{\V}+ C\|\z\|^{r+1}_{\widetilde{\L}^{r+1}}+C\|\z\|^{2}_{\H}, &\text{for } d=2,3\\ & \text{with } r\geq3,\\
				\frac{\mu}{4} \|\v^n\|^2_{\V}+\|\v^n\|^2_{\H}\|\z\|^{4}_{\widetilde{\L}^{4}}+ C\|\z\|^{4}_{\widetilde{\L}^{4}}, &\text{for } d=2  \text{ with }\\ & r\in[1,3),
			\end{cases}	
			\label{Sb}\\
			\beta\left|\left(\mathcal{C}(\v^n+\z),\z\right)\right|&\leq\beta \|\v^n+\z\|^r_{\wi\L^{r+1}}\|\z\|_{\wi\L^{r+1}}\leq\frac{\beta}{4}\|\v^n+\z\|^{r+1}_{\wi\L^{r+1}} + C\|\z\|^{r+1}_{\wi\L^{r+1}},\\
			\big|((\chi-\alpha)\z(t),\v^n(t))&+\langle\f,\v^n(t)\rangle\big|\leq \frac{\mu}{4}\|\v^n\|^2_{\V}+\frac{\alpha}{2}\|\v^n\|^2_{\H}+C\|\f\|^2_{\V'}+C\|\z\|^2_{\H},\label{S2}
		\end{align}
		where we have used \eqref{lady} in \eqref{Sb} for $d=2$ with $r\in[1,3)$. Combining \eqref{S1}-\eqref{S2}, we deduce
		\begin{align}\label{S3}
			&\frac{\d}{\d t}\|\v^n(t)\|^2_{\H}+\mu\|\v^n(t)\|^2_{\H}+\alpha\|\v^n(t)\|^2+\beta\|\v^n(t)+\z(t)\|^{r+1}_{\wi\L^{r+1}}\nonumber\\&\leq C \times\begin{cases}
				\|\z(t)\|^2_{\H}+\|\z(t)\|^{4}_{\wi\L^{4}} +\|\z(t)\|^{r+1}_{\wi\L^{r+1}}+\|\f\|^2_{\V'},&\text{for } d=2,3\\ & \text{with } r\geq3,\\
				\|\v^n(t)\|^2_{\H}\|\z(t)\|^{4}_{\wi\L^{4}}+\|\z(t)\|^2_{\H}+\|\z(t)\|^{4}_{\wi\L^{4}} +\|\z(t)\|^{r+1}_{\wi\L^{r+1}}+\|\f\|^2_{\V'}, &\text{for } d=2  \text{ with }\\ & r\in[1,3),
			\end{cases}
		\end{align}
		which gives 
		\begin{align}
			\|\v^n(t)\|^2_{\H}\leq \begin{cases}
				\|\v^n(0)\|^2_{\H} +C\int_{0}^{t}\big[\|\z(s)\|^2_{\H}+\|\z(s)\|^{4}_{\wi\L^{4}} +\|\z(s)\|^{r+1}_{\wi\L^{r+1}}+\|\f\|^2_{\V'}\big]\d s,\\ \hspace{75mm} \text{for } d=2,3 \text{ with } r\geq3,\\
		 \|\v^n(0)\|^2_{\H}e^{\int_{0}^{t}\|\z(\zeta)\|^4_{\wi\L^4}\d\zeta} \\ +C\int_{0}^{t}e^{\int_{s}^{t}\|\z(\zeta)\|^4_{\wi\L^4}\d\zeta}\big[\|\z(s)\|^2_{\H}+\|\z(s)\|^{4}_{\wi\L^{4}} +\|\z(s)\|^{r+1}_{\wi\L^{r+1}}+\|\f\|^2_{\V'}\big]\d s,\\ \hspace{75mm}\text{for } d=2  \text{ with } r\in[1,3).
			\end{cases}\label{S4}
		\end{align}
		Furthermore, $\z\in \mathrm{L}^2(0,T; \H)\cap\mathrm{L}^{4} (0, T; \widetilde{\L}^{4})\cap\mathrm{L}^{r+1} (0, T; \widetilde{\L}^{r+1})$ for $d=2,3$ with $r\geq3$ as well as $d=2$ with $r\in[1,3)$ (see Remark \ref{RemarkI}). Hence, using the fact that $\|\v^n(0)\|_{\H}\leq\|\v(0)\|_{\H}$ and $\f\in\V'$, we have from \eqref{S4} that $\sup\limits_{t\in[0,T]}\|\v^n(t)\|^2_{\H}<\infty,$ from which we infer 
		\begin{align}\label{S5}
			\{\v^n\}_{n\in\N} \text{ is a bounded sequence in }\mathrm{L}^{\infty}(0,T;\H).
		\end{align}
		Now, integrating \eqref{S3} from $0$ to $T$, we obtain
		\begin{align}\label{S6}
			\{\v^n\}_{n\in\N} \text{ is a bounded sequence in }\mathrm{L}^{2}(0,T;\V)\cap\mathrm{L}^{r+1}(0,T;\widetilde{\L}^{r+1}).
		\end{align}
		For any arbitrary element $\boldsymbol{\psi}\in\mathrm{L}^2(0,T;\V)\cap\mathrm{L}^{r+1}(0,T;\widetilde{\L}^{r+1})$, using H\"older's inequality and Sobolev's embedding, we have from \eqref{cscbf_n}
		\begin{align*}
			&\left|\int_{0}^{T}\left\langle\frac{\d\v^n(t)}{\d t},\boldsymbol{\psi}(t)\right\rangle\d t\right| \nonumber\\&\leq \int_{0}^{T}\bigg[\mu\left|(\nabla\v^n(t),\nabla\boldsymbol{\psi}(t))\right|+\alpha\left|(\v^n(t),\boldsymbol{\psi}(t))\right|+\left|b(\v^n(t)+\z(t),\boldsymbol{\psi}(t),\v^n(t)+\z(t))\right|\nonumber\\&\quad+\beta\left|\left\langle\mathcal{C}(\v^n(t)+\z(t)),\boldsymbol{\psi}(t)\right\rangle\right| +(\chi-\alpha)\left|(\z(t),\boldsymbol{\psi}(t))\right|+\left|(\f,\boldsymbol{\psi}(t))\right|\bigg]\d t \nonumber\\&\leq C\times \begin{cases}
			\int_{0}^{T}\bigg[\|\v^n(t)\|_{\V}\|\boldsymbol{\psi}(t)\|_{\V}+\|\v^n(t)\|_{\H}\|\boldsymbol{\psi}(t)\|_{\H}+\|\v^n(t)+\z(t)\|_{\widetilde{\L}^{r+1}}^{\frac{r+1}{r-1}}\|\v^n(t)+\z(t)\|_{\H}^{\frac{r-3}{r-1}}\\ \times\|\boldsymbol{\psi}(t)\|_{\V}+\|\v^n(t)+\z(t)\|^r_{\wi\L^{r+1}}\|\boldsymbol{\psi}(t)\|_{\widetilde{\L}^{r+1}}+\|\z(t)\|_{\H}\|\boldsymbol{\psi}(t)\|_{\H}+\|\f\|_{\V'}\|\boldsymbol{\psi}(t)\|_{\V}\bigg]\d t, \\  \hspace{100mm}\text{for } d=2,3 \text{ with } r\geq3,\\
			\int_{0}^{T}\bigg[\|\v^n(t)\|_{\V}\|\boldsymbol{\psi}(t)\|_{\V}+\|\v^n(t)\|_{\H}\|\boldsymbol{\psi}(t)\|_{\H}+\|\v^n(t)+\z(t)\|_{\widetilde{\L}^{4}}^{2}\|\boldsymbol{\psi}(t)\|_{\V}\\ +\|\v^n(t)+\z(t)\|^r_{\wi\L^{r+1}}\|\boldsymbol{\psi}(t)\|_{\widetilde{\L}^{r+1}}+\|\z(t)\|_{\H}\|\boldsymbol{\psi}(t)\|_{\H} +\|\f\|_{\V'}\|\boldsymbol{\psi}(t)\|_{\V}\bigg]\d t, \\  \hspace{100mm} \text{for } d=2  \text{ with } r\in[1,3).
			\end{cases}
		\nonumber\\&\leq C\times \begin{cases}
			\bigg[\|\v^n\|_{\mathrm{L}^{2}(0,T;\V)}+\|\v^n+\z\|^{\frac{r+1}{r-1}}_{\mathrm{L}^{r+1}(0,T;\widetilde{\L}^{r+1})}\|\v^n+\z\|^{\frac{r-3}{r-1}}_{\mathrm{L}^{2}(0,T;\H)}+\|\v^n+\z\|^{r}_{\mathrm{L}^{r+1}(0,T;\widetilde{\L}^{r+1})}\\  +\|\z\|_{\mathrm{L}^{2}(0,T;\H)}+T^{1/2}\|\f\|_{\V'}\bigg] \|\boldsymbol{\psi}\|_{\mathrm{L}^2(0,T;\V)\cap\mathrm{L}^{r+1}(0,T;\widetilde{\L}^{r+1})},  \hspace{8mm}\text{for } d=2,3 \text{ with } r\geq3,\\
			\bigg[\|\v^n\|_{\mathrm{L}^{2}(0,T;\V)}+\|\v^n+\z\|_{\mathrm{L}^{4}(0,T;\widetilde{\L}^{4})}^{2} +\|\v^n+\z\|^r_{\mathrm{L}^{r+1}(0,T;\widetilde{\L}^{r+1})}+\|\z\|_{\mathrm{L}^{2}(0,T;\H)}\\ + T^{1/2}\|\f\|_{\V'}\bigg]\|\boldsymbol{\psi}\|_{\mathrm{L}^2(0,T;\V)\cap\mathrm{L}^{r+1}(0,T;\widetilde{\L}^{r+1})},  \hspace{34mm} \text{for } d=2  \text{ with } r\in[1,3),
		\end{cases}
		\end{align*}
		which implies that $\frac{\d \v^n}{\d t}\in \mathrm{L}^{2}(0,T;\V')+\mathrm{L}^{\frac{r+1}{r}}(0,T;\widetilde{\L}^{\frac{r+1}{r}})$. Using \eqref{S5}, \eqref{S6} and the \emph{Banach Alaoglu theorem}, we infer that there exists an element $\v\in\mathrm{L}^{\infty}(0,T;\H)\cap\mathrm{L}^{2}(0,T;\V)\cap\mathrm{L}^{r+1}(0,T;\widetilde{\L}^{r+1})$ and $\frac{\d \v}{\d t}\in \mathrm{L}^{\frac{r+1}{r}}(0,T;\V')$ such that
		\begin{align}
			\v^n\xrightharpoonup{w^*}&\ \v\text{ in }	\mathrm{L}^{\infty}(0,T;\H),\label{S7}\\
			\v^n\xrightharpoonup{w}&\ \v\text{ in } \mathrm{L}^{2}(0,T;\V)\cap\mathrm{L}^{r+1}(0,T;\widetilde{\L}^{r+1}),\label{S8}\\
			\frac{\d \v^n}{\d t}\xrightharpoonup{w}&\frac{\d \v}{\d t} \text{ in }\mathrm{L}^{2}(0,T;\V')+\mathrm{L}^{\frac{r+1}{r}}(0,T;\widetilde{\L}^{\frac{r+1}{r}}),\label{S8d}
		\end{align}
		along a subsequence. The fact that $\v^n\in\mathrm{L}^{2}(0,T;\V)$ and $\frac{\d \v^n}{\d t}\in \mathrm{L}^{\frac{r+1}{r}}(0,T;\V'+\wi\L^{\frac{r+1}{r}})$ imply $\v^n\in\mathrm{L}^{2}(0,T;\H^1_0(\mathcal{O}_R))$ and $\frac{\d \v^n}{\d t}\in \mathrm{L}^{\frac{r+1}{r}}(0,T;\H^{-1}(\mathcal{O}_R)+\L^{\frac{r+1}{r}}(\mathcal{O}_{R}))$, where $$\mathcal{O}_R=\mathcal{O}\cap\{x\in\R^3:|x|< R\}.$$ Since, $\v^n\in\mathrm{L}^{2}(0,T;\H^1(\mathcal{O}_R))$, $\frac{\d \v^n}{\d t}\in \mathrm{L}^{\frac{r+1}{r}}(0,T;\H^{-1}(\mathcal{O}_R)+\L^{\frac{r+1}{r}}(\mathcal{O}_{R}))$, the embeddings $$\H^1(\mathcal{O}_R)\subset\L^2(\mathcal{O}_R)\subset\H^{-1}(\mathcal{O}_R)+\L^{\frac{r+1}{r}}(\mathcal{O}_{R})$$ are continuous  and the embedding $\H^1(\mathcal{O}_R)\subset\L^2(\mathcal{O}_R)$ is compact, then  the \emph{Aubin-Lions compactness lemma} implies that 
		\begin{align}\label{S9}
			\v^n\to\v \ \text{ strongly in } \ \mathrm{L}^2(0,T;\L^2(\mathcal{O}_R)).
		\end{align}
		Next, we prove that $\v$ is a solution to the system \eqref{cscbf}. Let $\psi:[0,T]\to\R$ be a continuously differentiable function. Also, let $\phi\in\H_m$ for some $m\in\N$. Then from \eqref{finite-dimS}, we have 
		\begin{align}\label{S10}
			&\int_{0}^{T}\left(\frac{\d\v^n(t)}{\d t},\psi(t)\phi\right)\d t\nonumber\\&=-\mu\int_{0}^{T} (\A_n\v^n(t),\psi(t)\phi)\d t-\alpha\int_{0}^{T}(\v^n(t),\psi(t)\phi)\d t\nonumber\\&\quad-\int_{0}^{T}(\B_n(\v^n(t)+\z(t)),\psi(t)\phi)\d t-\beta\int_{0}^{T}(\mathcal{C}_n(\v^n(t)+\z(t)),\psi(t)\phi)\d t \nonumber\\&\quad+(\chi-\alpha)\int_{0}^{T}(\z_n(t),\psi(t)\phi)\d t+\int_{0}^{T}(\f_n,\psi(t)\phi)\d t,
		\end{align}
		where we have used an integration by parts. Our next goal is to pass the limit in \eqref{S10} as $n\to \infty$. Due to the choice of $\phi\in\H_m,$ for some $m\in\N$, we can say that there exists $R\in\N$ such that $\text{supp}\ \phi\subset\mathcal{O}_R.$ Since $\psi(\cdot)\phi\in\mathrm{L}^{2}(0,T;\V)\cap\mathrm{L}^{r+1}(0,T;\widetilde{\L}^{r+1})$, in view of \eqref{S8d}, we obtain 
		\begin{align}\label{S11}
			\int_{0}^{T}\left(\frac{\d\v^n(t)}{\d t},\psi(t)\phi\right)\d t-\int_{0}^{T}\left\langle\frac{\d\v(t)}{\d t},\psi(t)\phi\right\rangle\d t=\int_{0}^{T}\left\langle\frac{\d\v^n(t)}{\d t}-\frac{\d\v}{\d t},\psi(t)\phi\right\rangle\d t\to 0,
		\end{align}
		as $n\to \infty$. Since $\psi(\cdot)\phi\in\mathrm{L}^2(0,T;\L^2(\mathcal{O}_R))$, we obtain
		\begin{align}\label{S12}
			&\left|\int_{0}^{T}(\v^n(t),\psi(t)\phi)\d t-\int_{0}^{T}(\v(t),\psi(t)\phi)\d t\right|\nonumber\\&\leq\|\v^n-\v\|_{\mathrm{L}^2(0,T;\L^2(\mathcal{O}_R))}\|\psi(\cdot)\phi\|_{\mathrm{L}^2(0,T;\L^2(\mathcal{O}_R))}\to0 \ \text{ as } \ n\to\infty,
		\end{align}
		where we have used the strong convergence obtained in \eqref{S9}. Let us choose $n\geq m$ so that $\H_m\subset\H_n$ and $\P_n\phi=\phi$. Since $\psi(\cdot)\phi\in\mathrm{L}^2(0,T;\V)$, consider 
		\begin{align}\label{S13}
			&	\int_{0}^{T} (\A_n\v^n(t),\psi(t)\phi)\d t-	\int_{0}^{T} (\!(\v(t),\psi(t)\phi)\!)\d t\nonumber\\&=\int_{0}^{T} (\!(\v^n(t)-\v(t),\psi(t)\phi)\!)\d t\to0 \ \text{ as } \ n\to\infty,
		\end{align}
		where we have used the weak convergence given in \eqref{S8}. To prove the convergence of  third term in the right hand side of \eqref{S10}, we consider
		\begin{align}\label{S14}
			&\left|\int_{0}^{T}(\B_n(\v^n(t)+\z(t)),\psi(t)\phi)\d t-\int_{0}^{T}(\B(\v(t)+\z(t)),\psi(t)\phi)\d t\right|\nonumber\\&
			\leq \underbrace{\left|\int_{0}^{T}b(\v^n(t),\v^n(t),\psi(t)\phi)\d t-\int_{0}^{T}b(\v(t),\v(t),\psi(t)\phi)\d t\right|}_{:=B_1(n)}\nonumber\\&\qquad+\left|\int_{0}^{T}b(\z(t),\v^n(t)-\v(t),\psi(t)\phi)\d t+\int_{0}^{T}b(\v^n(t)-v(t),\z(t),\psi(t)\phi)\d t\right|\nonumber\\&
			\leq B_1(n)+2\int_{0}^{T}\|\z(t)\|_{\wi\L^{4}}\|\v^n(t)-\v(t)\|_{\L^{4}(\mathcal{O}_R)}\|\psi(t)\nabla\phi\|_{\wi\L^2}\d t, \nonumber\\&\leq
			\begin{cases}
				B_1(n)+C\int_{0}^{T}\|\z(t)\|_{\wi\L^{4}}\|\v^n(t)-\v(t)\|^{1/2}_{\L^2(\mathcal{O}_R)}\|\v^n(t)-\v(t)\|^{1/2}_{\V}\d t, \  \text{	for  } d=2 \text{	with  }  r\geq1 ,\\
				B_1(n)+C\int_{0}^{T}\|\z(t)\|_{\wi\L^{4}}\|\v^n(t)-\v(t)\|^{1/4}_{\L^2(\mathcal{O}_R)}\|\v^n(t)-\v(t)\|^{3/4}_{\V}\d t,\ \text{	for  } d=3 \text{	with  } r\geq 3,
			\end{cases}	\nonumber\\&
			\leq \begin{cases}
				B_1(n)+CT^{\frac{1}{4}} \|\z\|_{\mathrm{L}^{4}(0,T;\widetilde{\L}^{4})}\|\v^n-\v\|^{\frac{1}{2}}_{\mathrm{L}^2(0,T;\L^2(\mathcal{O}_R))}\left[\|\v^n\|^{\frac{1}{2}}_{\mathrm{L}^2(0,T;\V)}+\|\v\|^{\frac{1}{2}}_{\mathrm{L}^{2}(0,T;\V)}\right],\\ \hspace{35mm}\text{	for  } d=2 \text{	with  }  r\geq1,\\
				B_1(n)+CT^{\frac{1}{4}} \|\z\|_{\mathrm{L}^{4}(0,T;\widetilde{\L}^{4})}\|\v^n-\v\|^{\frac{1}{4}}_{\mathrm{L}^2(0,T;\L^2(\mathcal{O}_R))}\left[\|\v^n\|^{\frac{3}{4}}_{\mathrm{L}^2(0,T;\V)}+\|\v\|^{\frac{3}{4}}_{\mathrm{L}^{2}(0,T;\V)}\right],\\ \hspace{35mm} \text{	for  } d=3 \text{	with  }  r\geq3,
			\end{cases} \nonumber\\& \to0 \text{ as } n\to \infty,
		\end{align}
		where we have used the convergence from Lemma \ref{convergence_b*}, \eqref{lady} and \eqref{S9}. From Lemma \ref{convergence_c2_1}, we get 
		\begin{align}\label{S15}
			\int_{0}^{T}(\mathcal{C}_n(\v^n(t)+\z(t)),\psi(t)\phi)\d t&=\int_{0}^{T}\left\langle\mathcal{C}(\v^n(t)+\z(t)),\psi(t)\phi\right\rangle\d t\nonumber\\&\to \int_{0}^{T}\left\langle\mathcal{C}(\v(t)+\z(t)),\psi(t)\phi\right\rangle\d t\  \text{ as }\  n\to\infty.
		\end{align}
		Furthermore, it is immediate that
		\begin{align}\label{S16}
			&\int_{0}^{T}((\chi-\alpha)\z_n(t)+\f_n,\psi(t)\phi)\d t\nonumber\\&\to (\chi-\alpha)\int_{0}^{T}(\z(t),\psi(t)\phi)\d t+\int_{0}^{T}\langle\f,\psi(t)\phi\rangle\d t,
		\end{align}
		since $\|\mathrm{P}_n-\mathrm{I}\|_{\mathfrak{L}(\H)}\to 0$ as $n\to\infty$. 	Hence, on passing limit in \eqref{S10} as $n\to\infty$ with the help of \eqref{S11}-\eqref{S16}, we obtain
		\begin{align}\label{S17}
			&\int_{0}^{T}\left\langle\frac{\d\v(t)}{\d t},\psi(t)\phi\right\rangle\d t\nonumber\\&=-\mu\int_{0}^{T} \left\langle\A\v(t),\psi(t)\phi\right\rangle\d t-\alpha\int_{0}^{T}(\v(t),\psi(t)\phi)\d t\nonumber\\&\quad-\int_{0}^{T}\left\langle\B(\v(t)+\z(t)),\psi(t)\phi\right\rangle\d t-\beta\int_{0}^{T}\left\langle\mathcal{C}(\v(t)+\z(t)),\psi(t)\phi\right\rangle\d t \nonumber\\&\quad+(\chi-\alpha)\int_{0}^{T}(\z(t),\psi(t)\phi)\d t+\int_{0}^{T}\left\langle\f,\psi(t)\phi\right\rangle\d t.
		\end{align}
		Since \eqref{S17} holds for any $\phi\in\bigcup_{m=1}^{\infty}\H_m$ and $\bigcup_{m=1}^{\infty}\H_m$ is dense in $\V\cap\wi\L^{r+1}$, we have that \eqref{S17} holds true for all $\phi\in\V\cap\wi\L^{r+1}$ and $\psi\in\mathrm{C}^1([0,T])$. Hence $\v(\cdot)$ solves \eqref{W-CSCBF} and satisfies the first equation of \eqref{cscbf}. 
	
		Note that the embedding of $\H\subset\V'+\wi\L^{\frac{r+1}{r}}$ is continuous and $\v\in \mathrm{L}^{\infty}(0,T;\H)$ implies $\v\in \mathrm{L}^{\infty}(0,T;\V'+\wi\L^{\frac{r+1}{r}})$. Thus, we get $\v,\frac{\d\v}{\d t}\in \mathrm{L}^{\frac{r+1}{r}}(0,T;\V'+\wi\L^{\frac{r+1}{r}})$ and then invoking \cite[Section 5.9.2, Theorem 2]{LCE}, it is immediate that $\v\in\C([0,T];\V'+\wi\L^{\frac{r+1}{r}})$. Since $\H$ is reflexive, using  \cite[Proposition 1.7.1]{PCAM}, we obtain $\v\in\C_w([0,T];\H)$ and the map $t\mapsto\|\v(t)\|_{\H}$ is bounded, where $\C_w([0,T];\H)$ denotes the space of functions $\v:[0,T]\to\H,$ which are weakly continuous. 
		\vskip 2mm
		\noindent
		\textbf{Step II.} \emph{Energy equality:} Now, we show that $\v(\cdot)$ satisfies the energy equality \eqref{eeq} and $\v\in\C([0, T];\H)$. Note that such an energy equality is not immediate due to	the fact that $\frac{\d \v}{\d t} \in\mathrm{L}^{2}(0,T;\V')+\mathrm{L}^{\frac{r+1}{r}}(0,T;\widetilde{\L}^{\frac{r+1}{r}})$. In \cite{FHR,Mohan1,Mohan}, the authors established an approximation of $\v(\cdot)$ in bounded domains	such that the approximations are bounded and converge in both Sobolev and Lebesgue spaces simultaneously (cf. \cite{HR} for such an approximation in periodic domains). In particular, they approximate $\v(t)$, for each $t\in [0,T],$ by using the finite-dimensional space spanned by the first $n$ eigenfunctions of the Stokes operator $\A$. Since we are working on unbounded domains, we do not have the existence of eigenfunctions of the Stokes operator. Therefore, we use the eigenfunctions of operator $\mathcal{L}$ (cf. \eqref{L2} and \eqref{L3}) to obtain a sequence which approximates $\v(\cdot)$. Set 
		\begin{align}\label{3.32}
			\v_m(t):=\mathrm{P}_{1/m}\v(t)=\sum_{\mu_j<m^2}e^{-\mu_j/m}\langle\v(t),\boldsymbol{e}_j\rangle_{\mathbb{U}'\times\mathbb{U}} \boldsymbol{e}_j.
		\end{align}
		Since, for $s>2$, $\{\boldsymbol{e}_j\}_{j\in\N}\subset\D(\mathcal{L})\subset\V_s\subset\V\cap\wi\L^{r+1}$, one can obtain (cf. \cite{FHR,Mohan1,Mohan}) 
		\begin{align}
			\|\v_{m}-\v\|_{\mathrm{L}^{2}(0,T;\V)\cap\mathrm{L}^{r+1}(0,T;\widetilde{\L}^{r+1})}\to0 \ \text{ as }\  m\to\infty.
		\end{align}
		Now, we define $\mathcal{V}_{T}:=\{\boldsymbol{\phi}\in\mathrm{C}^{\infty}_{0}(\mathcal{O}\times[0,T)):\nabla\cdot\boldsymbol{\phi}(x,t)=0\}$. Observe that, for each $\boldsymbol{\phi}\in\mathcal{V}_{T}$, $\boldsymbol{\phi}(\cdot,T)=0$ and $\mathcal{V}_{T}$ is dense in $\mathrm{L}^{p}(0,T;\mathbb{H}^1(\mathcal{O})\cap\L^{r+1}(\mathcal{O}))$ (cf. \cite[Lemmas 2.5, 2.6]{GPG}). For $\v\in\mathrm{L}^p(0,T;\mathrm{X})$, $1\leq p<\infty$ and $T>h>0$, the mollifier $\v_h$ (in the sense of Friederichs) of $\v$ is defined by 
		\begin{align*}
			\v_h(t)=\int_{0}^{T}j_h(t-\zeta)\v(\zeta)\d\zeta,
		\end{align*}
		where $j_h(\cdot)$ is an infinite times differentiable function having support in $(-h,h)$, which is even and positive, such that $\int_{-\infty}^{+\infty}j_h(\zeta)\d\zeta=1$. In view of \cite[Lemma 2.5]{GPG}, we have that for $\v\in\mathrm{L}^p(0,T;\mathrm{X})$ with $1\leq p<\infty$, $\v_h\in\mathrm{C}^k([0,T];\mathrm{X})$ for all $k\geq0$ and
		\begin{align}\label{335}
			\lim_{h\to0}\|\v_h-\v\|_{\mathrm{L}^p(0,T;\mathrm{X})}=0.
		\end{align}
		Furthermore, if $\{\v_m\}_{m\in\N}\subset\mathrm{L}^p(0,T;\mathrm{X})$ converges to $\v$ in the norm of $\mathrm{L}^p(0,T;\mathrm{X})$, then
		\begin{align}\label{336}
			\lim_{m\to\infty}\|(\v_m)_h-\v_h\|_{\mathrm{L}^p(0,T;\mathrm{X})}=0.
		\end{align}
		We write the weak solution of \eqref{cscbf} as
		\begin{align}\label{337}
			\int_{0}^{t}\biggl\{\left\langle\frac{\d \v}{\d t}+\mu\A\v+\B(\v+\z)+\beta\mathcal{C}(\v+\z)-\f,\boldsymbol{\phi}\right\rangle+\left(\alpha\v+(\alpha-\chi)\z,\boldsymbol{\phi}\right)\biggr\}\d\zeta=0,
		\end{align}
		for all $t<T$ and $\boldsymbol{\phi}\in\mathcal{V}_{T}$. Choosing $\boldsymbol{\phi}=(\v_m)_h=:\v_{m,h}$ in \eqref{337}, where $(\cdot)_h$ is the mollification operator discussed above, for $0\leq t<T$, we get
		\begin{align}\label{338}
			\int_{0}^{t}\biggl\{\left\langle\frac{\d \v}{\d t},\v_{m,h}\right\rangle+\mu\left(\nabla\v,\nabla\v_{m,h}\right)&+\left\langle\B(\v+\z)+\beta\mathcal{C}(\v+\z)-\f,\v_{m,h}\right\rangle\nonumber\\&\quad+\left(\alpha\v+(\alpha-\chi)\z,\v_{m,h}\right)\biggr\}\d\zeta=0.
		\end{align}
		Using \eqref{336}, we obtain 
		\begin{align*}
			&\left|\int_{0}^{t}\left\langle\frac{\d \v}{\d t},\v_{m,h}-\v_{h}\right\rangle\d\zeta\right|\nonumber\\&\leq\left\|\frac{\d \v}{\d t}\right\|_{\mathrm{L}^{2}(0,T;\V')+\mathrm{L}^{\frac{r+1}{r}}(0,T;\widetilde{\L}^{\frac{r+1}{r}})}\|\v_{m,h}-\v_h\|_{\mathrm{L}^{2}(0,T;\V)\cap\mathrm{L}^{r+1}(0,T;\widetilde{\L}^{r+1})}\to0,
		\end{align*}
		and
		\begin{align*}
			\left|\int_{0}^{t}\left(\nabla \v,\nabla\v_{m,h}-\nabla\v_{h}\right)\d\zeta\right|\leq\| \v\|_{\mathrm{L}^{2}(0,T;\V)}\|\v_{m,h}-\v_h\|_{\mathrm{L}^{2}(0,T;\V)}\to0,
		\end{align*}
		as $m\to\infty$. Since $\v\in\mathrm{L}^{\infty}(0,T;\H)\cap\mathrm{L}^2(0,T;\V)\cap\mathrm{L}^{r+1}(0,T;\wi\L^{r+1})$ and $\z \in\mathrm{L}^{2}(0,T;\H)\cap\mathrm{L}^{r+1}(0,T;\wi\L^{r+1})\cap\mathrm{L}^{4}(0,T;\wi\L^{4})$ (see Remark \ref{RemarkI}), we have 
		\begin{align*}	
		&\left|\int_{0}^{t}\left\langle\B(\v+\z),\v_{m,h}-\v_{h}\right\rangle\d\zeta\right|\nonumber\\&\leq
		\begin{cases}
			\|\v+\z\|^{\frac{r+1}{r-1}}_{\mathrm{L}^{r+1}(0,T;\widetilde{\L}^{r+1})}\|\v+\z\|^{\frac{r-3}{r-1}}_{\mathrm{L}^{2}(0,T;\H)}\|\v_{m,h}-\v_h\|_{\mathrm{L}^{2}(0,T;\V)}, \text{	for  } d=2,3 \text{	with  } r\geq3,\\
			\bigg[\|\v\|^2_{\mathrm{L}^{\infty}(0,T;\H)}\|\v\|_{\mathrm{L}^{2}(0,T;\V)}+\|\z\|^2_{\mathrm{L}^{4}(0,T;\widetilde{\L}^4)}\bigg]\|\v_{m,h}-\v_h\|_{\mathrm{L}^{2}(0,T;\V)}, \text{	for  } d=2 \text{	with  } r\in[1,3),
		\end{cases}
	\nonumber\\&\to0 \ \text{ as } \ m\to\infty.
		\end{align*}
	Since $\v, \z \in\mathrm{L}^{r+1}(0,T;\wi\L^{r+1})$, we obtain
		\begin{align*}	\left|\int_{0}^{t}\left\langle\mathcal{C}(\v+\z),\v_{m,h}-\v_{h}\right\rangle\d\zeta\right|\leq \| \v+\z\|^r_{\mathrm{L}^{r+1}(0,T;\widetilde{\L}^{r+1})}\|\v_{m,h}-\v_h\|_{\mathrm{L}^{r+1}(0,T;\widetilde{\L}^{r+1})}\to0,
		\end{align*}
		as $m\to\infty$. Similarly, using \eqref{336}, we get
		\begin{align*}
			&\int_{0}^{t}\biggl\{-\left\langle\f,\v_{m,h}\right\rangle+\left(\alpha\v+(\alpha-\chi)\z,\v_{m,h}\right)\biggr\}\d\zeta\nonumber\\&\to		\int_{0}^{t}\biggl\{-\left\langle\f,\v_{h}\right\rangle+\left(\alpha\v+(\alpha-\chi)\z,\v_{h}\right)\biggr\}\d\zeta,\ \text{ as }\ m\to\infty.
		\end{align*}
		Hence, passing limit $m\to\infty$ in \eqref{338}, we obtain
		\begin{align}\label{339}
			\int_{0}^{t}\biggl\{\left\langle\frac{\d \v}{\d t},\v_{h}\right\rangle+\mu\left(\nabla\v,\nabla\v_{h}\right)&+\left\langle\B(\v+\z)+\beta\mathcal{C}(\v+\z)-\f,\v_{h}\right\rangle\nonumber\\&+\left(\alpha\v+(\alpha-\chi)\z,\v_{h}\right)\biggr\}\d\zeta=0.
		\end{align}
		Using \eqref{335} and similar arguments as above, we obtain the following convergence
		\begin{align}\label{340}
			&\lim_{h\to0}\int_{0}^{t}\biggl\{\mu\left(\nabla\v,\nabla\v_{h}\right)+\left\langle\B(\v+\z)+\beta\mathcal{C}(\v+\z)-\f,\v_{h}\right\rangle+\left(\alpha\v+(\alpha-\chi)\z,\v_{h}\right)\biggr\}\d\zeta\nonumber\\&=\int_{0}^{t}\biggl\{\mu\left(\nabla\v,\nabla\v\right)+\left\langle\B(\v+\z)+\beta\mathcal{C}(\v+\z)-\f,\v\right\rangle+\left(\alpha\v+(\alpha-\chi)\z,\v\right)\biggr\}\d\zeta.
		\end{align}
		Using integration by parts, we get
		\begin{align}\label{341}
			\int_{0}^{t}\left\langle\frac{\d \v}{\d t},\v_{h}\right\rangle\d\zeta&=-\int_{0}^{t}\left\langle\v,\frac{\d \v_h}{\d t}\right\rangle\d\zeta+\left(\v(0),\v_h(0)\right)-\left(\v(t),\v_h(t)\right)\nonumber\\&=-\int_{0}^{t}\int_{0}^{\zeta}\frac{\d j_h(\zeta-s)}{\d t}\left(\v(\zeta),\v(s)\right)\d s\d\zeta+\left(\v(0),\v_h(0)\right)-\left(\v(t),\v_h(t)\right)\nonumber\\&=\left(\v(0),\v_h(0)\right)-\left(\v(t),\v_h(t)\right)\nonumber\\&\to\left(\v(0),\v(0)\right)-\left(\v(t),\v(t)\right),
		\end{align}
		as $h\to0$, where we have used the property of mollifiers and the fact that the kernel $j_{h}(s)$ in the definition of mollifier is even in $(-h,h)$. From \eqref{339}-\eqref{341}, we infer that $\v(\cdot)$ satisfies the energy equality, that is, condition (iii) of Definition \ref{defn5.9}. Recalling that every weak solution of \eqref{cscbf} is $\H$-weakly continuous in time, all weak solutions satisfy the energy equality (condition (iii) of Definition \ref{defn5.9}) and so, all weak solutions of \eqref{cscbf} belong to $\C([0, T];\H)$ (cf. \cite{GPG,HR} also). Thus the condition (ii) in the Definition \ref{defn5.9} also makes sense.
		\vskip 2mm
		\noindent
		\textbf{Step III.}	\emph{Uniqueness:} Define $\mathfrak{F}=\v_1-\v_2$, where $\v_1$ and $\v_2$ are two weak solutions of the system \eqref{cscbf} in the sense of Definition \ref{defn5.9}. Then $\mathfrak{F}\in\mathrm{C}(0,T;\H)\cap\mathrm{L}^{2}(0,T;\V)\cap\mathrm{L}^{r+1}(0,T;\widetilde{\L}^{r+1})$ and satisfies
		\begin{equation}\label{Uni}
			\left\{
			\begin{aligned}
				\frac{\d\mathfrak{F}(t)}{\d t} &= -\mu \A\mathfrak{F} (t)-\alpha\mathfrak{F}(t)- \B(\v_1(t)+\z(t))+\B(\v_2(t)+\z(t))- \beta \mathcal{C}(\v_1 (t)+ \z(t)) \\&\quad+ \beta \mathcal{C}(\v_2 (t)+ \z(t)), \\
				\mathfrak{F}(0)&= \textbf{0},
			\end{aligned}
			\right.
		\end{equation}
		in the weak sense.	From the above equation, using  the energy equality, we obtain
		\begin{align}\label{U1}
			&	\frac{1}{2}\frac{\d}{\d t}\|\mathfrak{F}(t)\|^2_{\H}+\mu\|\mathfrak{F}(t)\|^2_{\V}+\alpha\|\mathfrak{F}(t)\|^2_{\H}+\beta\left\langle \mathcal{C}(\v_1 (t)+ \z(t)) - \mathcal{C}(\v_2 (t)+ \z(t)),\v_1(t)-\v_2(t)\right\rangle\nonumber\\&=-b(\v_1(t)+\z(t),\v_1(t)+\z(t),\mathfrak{F}(t)) +b(\v_2(t)+\z(t),\v_2(t)+\z(t),\mathfrak{F}(t))\nonumber\\&=b(\mathfrak{F}(t),\mathfrak{F}(t),\v_2(t)+\z(t)) \nonumber\\&\leq\begin{cases}
				\frac{\mu}{2}\|\mathfrak{F}(t)\|^{2}_{\V}+\frac{\beta}{2}\|\left|\v_2(t)+\z(t)\right|^{\frac{r-1}{2}}\mathfrak{F}(t)\|^2_{\H} +C\|\mathfrak{F}(t)\|^2_{\H},  &\text{ for } d=2,3 \text{	with  } r>3,\\
				\frac{1}{2\beta}\|\mathfrak{F}(t)\|^{2}_{\V}+\frac{\beta}{2}\|\left|\v_2(t)+\z(t)\right|\mathfrak{F}(t)\|^2_{\H}, &\text{ for } d=r=3 \text{	with  } 2\beta\mu\geq1,\\
				\frac{\mu}{2}\|\mathfrak{F}(t)\|^{2}_{\V}+C\|\v_2(t)+\z(t)\|^4_{\wi\L^4}\|\mathfrak{F}(t)\|^2_{\H},  &\text{ for } d=2 \text{	with  } r\in[1,3].
			\end{cases}
		\end{align}
		From \eqref{MO_c}, we have
		\begin{align}\label{U1'}
			-&\beta\left\langle \mathcal{C}(\v_1 (t)+ \z(t)) - \mathcal{C}(\v_2 (t)+ \z(t)),\v_1(t)-\v_2(t)\right\rangle\nonumber\\&\leq -\frac{\beta}{2}\|\left|\v_2(t)+\z(t)\right|^{\frac{r-1}{2}}\mathfrak{F}(t)\|^2_{\H} 
		\end{align}
		Therefore, \eqref{U1} and \eqref{U1'} imply that
		\begin{align}\label{U2}
			&\frac{\d}{\d t}\|\mathfrak{F}(t)\|^2_{\H}\leq \begin{cases}
				C\|\mathfrak{F}(t)\|^{2}_{\H}, &\text{ for } d=2,3 \text{	with  } r>3,\\
				0, &\text{ for } d=r=3 \text{ with } 2\beta\mu\geq1,\\
				C\|\v_2(t)+\z(t)\|^4_{\wi\L^4}\|\mathfrak{F}(t)\|^2_{\H},  &\text{ for } d=2 \text{	with  } r\in[1,3].
			\end{cases}
		\end{align}
		Applying Gronwall's inequality and using the fact that $\mathfrak{F}(0)=\textbf{0}$, we obtain that $\v_1(t)=\v_2(t)$, for all $t\in[0,T]$, which completes the proof of uniqueness.
	\end{proof}	
	
	\begin{remark}
	From \cite[Chapter II, Theorem 1.8]{CV}, $\v\in\mathrm{L}^{2}(0,T;\V)\cap\mathrm{L}^{r+1}(0,T;\widetilde{\L}^{r+1})$ and $\frac{\d \v}{\d t}\in\mathrm{L}^{2}(0,T;\V')+\mathrm{L}^{\frac{r+1}{r}}(0,T;\widetilde{\L}^{\frac{r+1}{r}})$ imply that $\v\in\mathrm{C}(0,T;\H)$ and satisfies the energy equality \eqref{eeq}.  The proof of Theorem 1.8 in \cite[Chapter II]{CV} requires a regularization technique which is not explicitly provided in  \cite{CV}, therefore we are not using Theorem 1.8 form \cite[Chapter II]{CV} directly.
	\end{remark}
	
	The following Theorem is general and we take $\f$ is dependent on $t$. 
	
	\begin{theorem}\label{RDS_Conti1}
		For all the cases given in Table \ref{Table1} and for some $T >0$ fixed, assume that $\boldsymbol{x}_n \to \boldsymbol{x}$ in $\H$, $\f_n \to \f \ \text{ in }\ \mathrm{L}^2 (0, T; \V')$,
		$$\z_n \to \z\ \text{ in }\ \mathrm{L}^2 (0, T; \H)\cap\mathrm{L}^{4} (0, T; \widetilde{\L}^{4}), \text{ for } r\in[1,3)$$
		and 
		\begin{align*}
			\z_n \to \z \ \text{ in }\  \mathrm{L}^2 (0, T; \H)\cap\mathrm{L}^{r+1} (0, T; \widetilde{\L}^{r+1}), \text{ for } r\geq3.
		\end{align*}
		Let us denote by $\v(t, \z)\boldsymbol{x},$ the solution of the system \eqref{cscbf} and by $\v(t, \z_n)\boldsymbol{x}_n,$  the solution of the system \eqref{cscbf} with $\z, \f, \boldsymbol{x}$ being replaced by $\z_n, \f_n, \boldsymbol{x}_n$. Then \begin{align}\label{5.20}\v(\cdot, \z_n)\boldsymbol{x}_n \to \v(\cdot, \z)\boldsymbol{x} \ \text{ in } \ \mathrm{C}([0,T];\H)\cap\mathrm{L}^2 (0, T;\V)\cap\mathrm{L}^{r+1}(0, T;\widetilde{\L}^{r+1}).\end{align}
		In particular, $\v(T, \z_n)\boldsymbol{x}_n \to \v(T, \z)\boldsymbol{x}$ in $\H$.
	\end{theorem}	
	\begin{proof}
		Let us introduce the following notations which help us to simplify the proof: 
		$$\v_n (\cdot) = \v(\cdot, \z_n)\boldsymbol{x}_n, \ \  \v(\cdot) = \v(\cdot, \z)\boldsymbol{x},\ \   \y_n (\cdot)= \v(\cdot, \z_n)\boldsymbol{x}_n - \v(\cdot, \z)\boldsymbol{x},$$     
		$$ \hat{ \z}_n = \z_n - \z, \ \  \hat{ \f}_n = \f_n - \f.$$
		Then $\y_n$ satisfies the following system:
		\begin{equation}\label{cscbf_n_1}
			\left\{
			\begin{aligned}
				\frac{\d\y_n}{\d t} &= -\mu \A\y_n-\alpha\y_n  - \B(\v_n  + \z_n ) + \B(\v + \z) - \beta \mathcal{C}(\v_n  + \z_n ) \\& \ \ \ + \beta \mathcal{C}(\v + \z) + (\chi-\alpha) \hat{ \z}_n + \hat{ \f}, \\
				\y_n(0)&= \boldsymbol{x}_n - \boldsymbol{x}.
			\end{aligned}
			\right.
		\end{equation}
		Multiplying by $\y_n(t)$ in the first equation of \eqref{cscbf_n_1} and integrating over $\mathcal{O}$, we obtain 
		\begin{align}\label{5.22}
			\frac{1}{2}& \frac{\d}{\d t}\|\y_n(t) \|^2_{\H}  \nonumber\\
			&=- \mu \|\y_n(t)\|^2_{\V}-\alpha\|\y_n(t)\|^2_{\H}+ b(\y_n(t), \y_n(t), \v_n(t)) + b(\y_n(t), \y_n(t), \z_n(t))  \nonumber\\&\quad+ b(\v(t),\y_n(t) , \hat{ \z}_n(t))+ b(\hat{ \z}_n(t), \y_n(t), \v_n(t)) + b(\hat{ \z}_n(t), \y_n(t), \z_n(t))  \nonumber\\&\quad+ b(\z(t), \y_n(t), \hat{ \z}_n(t))+ \beta\big\langle \mathcal{C}(\v_n(t)  + \z_n(t) ) ,  \hat{ \z}_n(t)\big\rangle -  \beta\big\langle \mathcal{C}(\v(t) + \z(t)),  \hat{ \z}_n(t)\big\rangle\nonumber\\&\quad- \beta\big\langle \mathcal{C}(\v_n (t) + \z_n(t) ) - \mathcal{C}(\v(t) + \z(t)), (\v_n(t)+\z_n(t))-(\v(t)+\z(t))\big\rangle  \nonumber\\&\quad+ (\chi-\alpha)\langle\hat{ \z}_n(t), \ \y_n(t)\rangle+ \langle\hat{ \f}_n(t), \y_n(t)\rangle, \ \text{ for a.e. }\ t\in [0,T].
		\end{align}
		We estimate trilinear terms $b(\cdot,\cdot,\cdot)$ for $d=2$ with $r\in[1,3]$, $d=2,3$ with $r>3$ and $d=r=3$ with $2\beta\mu\geq1$, separately.
		\vskip 2mm
		\noindent
		\textbf{Estimates for $b(\cdot,\cdot,\cdot):$} \emph{For $d=2$ with $r\in[1,3]$.} In view of the inequality \eqref{lady}, along with Young's and H\"older's inequalities, we calculate 
		\begin{align*}
			|b(\y_n, \y_n, \v_n)|&=|b(\y_n, \v_n, \y_n)|
			\leq \|\y_n\|^2_{\widetilde{\L}^{4}} \ \|\v_n\|_{\V} \leq \sqrt{2} \|\y_n\|_{\H} \|\y_n\|_{\V}\|\v_n\|_{\V}\\
			&\leq \frac{\mu}{8} \|\y_n\|^2_{\V} + C\|\v_n\|^2_{\V} \|\y_n\|^2_{\H},\\
			|b(\y_n, \y_n, \z_n)|&\leq \|\y_n\|_{\widetilde{\L}^4} \|\y_n\|_{\V} \|\z_n\|_{\widetilde{\L}^4}\leq 2^{1/4}\|\y_n\|^{1/2}_{\H} \|\y_n\|^{3/2}_{\V} \|\z_n\|_{\widetilde{\L}^4}\\
			&\leq \frac{\mu}{8} \|\y_n\|^2_{\V} + C \|\z_n\|^4_{\widetilde{\L}^4}\|\y_n\|^2_{\H},\\
			|b(\v,\y_n , \hat{ \z}_n)+b(\hat{ \z}_n, \y_n, \v_n)|&\leq (\|\v\|_{\widetilde{\L}^4}+\|\v_n\|_{\widetilde{\L}^4})\|\y_n\|_{\V}\|\hat{ \z}_n\|_{\widetilde{\L}^4}\\&\leq 2^{1/4}(\|\v\|^{1/2}_{\H}\|\v\|^{1/2}_{\V}+\|\v_n\|^{1/2}_{\H}\|\v_n\|^{1/2}_{\V})\|\y_n\|_{\V}\|\hat{ \z}_n\|_{\widetilde{\L}^4}\\
			&\leq\frac{\mu}{8}\|\y_n\|^2_{\V} + C (\|\v\|_{\H}\|\v\|_{\V}+\|\v_n\|_{\H}\|\v_n\|_{\V})\|\hat{ \z}_n\|^2_{\widetilde{\L}^4},\\
			|b(\hat{ \z}_n, \y_n, \z_n)+b(\z, \y_n, \hat{ \z}_n)|&\leq (\|\z_n\|_{\widetilde{\L}^4}+\|\z\|_{\widetilde{\L}^4})\|\hat{ \z}_n\|_{\widetilde{\L}^4} \|\y_n\|_{\V}\\& \leq\frac{\mu}{8}\|\y_n\|^2_{\V} + C (\|\z_n\|^2_{\widetilde{\L}^4}+\|\z\|^2_{\widetilde{\L}^4})\|\hat{ \z}_n\|^2_{\widetilde{\L}^4}.
		\end{align*}
		\vskip2mm
		\noindent
		\textit{For $d=2,3$ with $r>3$.} Using H\"older's, interpolation (Lemma \ref{Interpolation}) and Young's inequalities, we obtain
		\begin{align*}
			|b(\y_n+\hat{\z}_n, \y_n, \v_n+\z_n)|&\leq  \frac{\mu}{4}\|\y_n\|^2_{\V} + \frac{\beta}{4}\|\left|\v_n+\z_n\right|^{\frac{r-1}{2}}(\y_n+\hat{\z}_n)\|^2_{\H}+C\|\y_n\|^2_{\H}+C\|\hat{ \z}_n\|^2_{\H}\\
			|b(\v+\z,\y_n , \hat{ \z}_n)|&\leq (\|\v\|_{\widetilde{\L}^{\frac{2(r+1)}{r-1}}}+\|\z\|_{\widetilde{\L}^{\frac{2(r+1)}{r-1}}})\|\y_n\|_{\V}\|\hat{ \z}_n\|_{\widetilde{\L}^{r+1}}\\&\leq (\|\v\|^{\frac{2}{r-1}}_{\wi \L^{r+1}}\|\v\|^{\frac{r-3}{r-1}}_{\H}+\|\z\|^{\frac{2}{r-1}}_{\wi \L^{r+1}}\|\z\|^{\frac{r-3}{r-1}}_{\H})\|\y_n\|_{\V}\|\hat{ \z}_n\|_{\widetilde{\L}^{r+1}}\\
			&\leq\frac{\mu}{4}\|\y_n\|^2_{\V} + C( \|\v\|^{\frac{4}{r-1}}_{\wi\L^{r+1}}\|\v\|^{\frac{2(r-3)}{r-1}}_{\H}+ \|\z\|^{\frac{4}{r-1}}_{\wi\L^{r+1}}\|\z\|^{\frac{2(r-3)}{r-1}}_{\H})\|\hat{ \z}_n\|^2_{\widetilde{\L}^{r+1}}.
		\end{align*}
		\vskip2mm
		\noindent
		\textit{For  $d=r=3$ with $2\beta\mu\geq1$.} Using H\"older's and Young's inequalities, we obtain
		\begin{align*}
			|b(\y_n+\hat{\z}_n, \y_n, \v_n+\z_n)|&\leq  \frac{1}{2\beta}\|\y_n\|^2_{\V} + \frac{\beta}{2}\|\left|\v_n+\z_n\right|^{\frac{r-1}{2}}(\y_n+\hat{\z}_n)\|^2_{\H}\\
			|b(\v+\z,\y_n , \hat{ \z}_n)|&\leq (\|\v\|_{\widetilde{\L}^{4}}+\|\z\|_{\widetilde{\L}^{4}})\|\y_n\|_{\V}\|\hat{ \z}_n\|_{\widetilde{\L}^{4}}\\
			&\leq (\|\v\|_{\widetilde{\L}^{4}}+\|\z\|_{\widetilde{\L}^{4}})(\|\v_n\|_{\V}+\|\v\|_{\V})\|\hat{\z}_n\|_{\widetilde{\L}^{4}}.
		\end{align*}
		The following calculations hold true for $d=2$ with $r\in[1,\infty)$ and $d=3$ with $[3,\infty)$ (for $d=r=3$ with $2\beta\mu\geq1$). Using H\"older's and Young's inequalities, we infer
		\begin{align*}
			\big|\big\langle \mathcal{C}(\v_n   + \z_n  ) ,  \hat{ \z}_n \big\rangle\big|& \leq \|\v_n +\z_n \|^r_{\widetilde{\L}^{r+1}} \|\hat{ \z}_n \|_{\widetilde{\L}^{r+1}} \leq C \big(\|\v_n \|^r_{\widetilde{\L}^{r+1}}+\|\z_n \|^r_{\widetilde{\L}^{r+1}}\big)\|\hat{ \z}_n \|_{\widetilde{\L}^{r+1}},\\
			\big|\big\langle \mathcal{C}(\v   + \z  ) ,  \hat{ \z}_n \big\rangle\big|& \leq \|\v +\z \|^r_{\widetilde{\L}^{r+1}} \|\hat{ \z}_n \|_{\widetilde{\L}^{r+1}} \leq C \big(\|\v \|^r_{\widetilde{\L}^{r+1}}+\|\z \|^r_{\widetilde{\L}^{r+1}}\big)\|\hat{ \z}_n \|_{\widetilde{\L}^{r+1}},\\
			|(\chi-\alpha)(\hat{ \z}_n, \ \y_n)|& \leq \left|(\chi-\alpha)\right| \|\y_n\|_{\H} \|\hat{ \z}_n\|_{\H}\leq \frac{\alpha}{2}\|\y_n\|^2_{\H}+C\|\hat{ \z}_n\|^2_{\H},\\
			|	\langle\hat{ \f}_n, \ \y_n\rangle|& \leq  \|\y_n\|_{\V} \|\hat{ \f}_n\|_{\V'}\leq (\|\v_n\|_{\V}+\|\v\|_{\V}) \|\hat{ \f}_n\|_{\V'}.
		\end{align*}
		Making use of \eqref{MO_c} and \eqref{a215}, we obtain 
		\begin{align*}
			-&\beta\big\langle \mathcal{C}(\v_n  + \z_n ) - \mathcal{C}(\v + \z), (\v_n+\z_n)-(\v+\z)\big\rangle\\
			&\leq-\frac{\beta}{2}\|\left|\v_n+\z_n\right|^{\frac{r-1}{2}}(\y_n+\hat{\z}_n)\|^2_{\H}-\frac{\beta}{2}\|\left|\v+\z\right|^{\frac{r-1}{2}}(\y_n+\hat{\z}_n)\|^2_{\H}		
			\\&\leq-\frac{\beta}{4}\|\left|\v_n+\z_n\right|^{\frac{r-1}{2}}(\y_n+\hat{\z}_n)\|^2_{\H}-\frac{\beta}{4}\|\left|\v+\z\right|^{\frac{r-1}{2}}(\y_n+\hat{\z}_n)\|^2_{\H}\nonumber\\&\quad-\frac{\beta}{2^{2r}}\|\y_n\|_{\wi\L^{r+1}}^{r+1}+\frac{\beta}{2^{r}}\|\hat{\z}_n\|_{\wi\L^{r+1}}^{r+1}.
		\end{align*}
		We complete  the calculations for $d=2$ with $r\in[1,3]$ only. Other cases can be handled in a  similar way. Combining the above inequalities, from \eqref{5.22}, we deduce 
		\begin{align*}
			&\frac{\d}{\d t}\|\y_n(t) \|^2_{\H} + \mu \|\y_n(t)\|^2_{\V}+\frac{\beta}{2^{2r}}\|\y_n\|_{\wi\L^{r+1}}^{r+1} \nonumber\\& \leq C\|\v_n(t)\|^2_{\V} \|\y_n(t)\|^2_{\H} +C \|\z_n(t)\|^4_{\widetilde{\L}^4}\|\y_n(t)\|^2_{\H} + C \|\v(t)\|_{\H}\|\v(t)\|_{\V}\|\hat{ \z}_n(t)\|^2_{\widetilde{\L}^4} \\&\quad+ C \|\v_n(t)\|_{\H}\|\v_n(t)\|_{\V}\|\hat{ \z}_n(t)\|^2_{\widetilde{\L}^4}+ C \|\z(t)\|^2_{\widetilde{\L}^4}\|\hat{ \z}_n(t)\|^2_{\widetilde{\L}^4}+ C \|\z_n(t)\|^2_{\widetilde{\L}^4}\|\hat{ \z}_n(t)\|^2_{\widetilde{\L}^4}\\&\quad+ C \|\hat{ \z}_n(t)\|^2_{\H}+C\|\hat{\z}_n(t)\|_{\wi\L^{r+1}}^{r+1} + (\|\v_n(t)\|_{\V}+\|\v(t)\|_{\V}) \|\hat{\f}_n(t)\|_{\V'}\\&\quad+ C \big(\|\v_n (t)\|^r_{\widetilde{\L}^{r+1}}+\|\z_n(t) \|^r_{\widetilde{\L}^{r+1}} +  \|\v(t) \|^r_{\widetilde{\L}^{r+1}}+\|\z(t) \|^r_{\widetilde{\L}^{r+1}}\big)\|\hat{ \z}_n(t) \|_{\widetilde{\L}^{r+1}}, 
		\end{align*}
		for a.e. $t\in[0,T]$.	Now, integrating from $0$ to $t$ to the above inequality, we obtain
		\begin{align}\label{Energy_esti_n_2_1}
			&\|\y_n(t)\|^2_{\H} + \mu \int_{0}^{t} \|\y_n(s) \|^2_{\V}\d s+\frac{\beta}{2^{2r}}\int_{0}^{t}\|\y_n(s)\|_{\wi\L^{r+1}}^{r+1}\d s\nonumber\\&\leq \|\y_n(0)\|^2_{\H} + \int_{0}^{t} \upalpha_n (s) \|\y_n(s)\|^2_{\H}\ \d s + C \int_{0}^{t} \upbeta_n(s)\ \d s, 
		\end{align}
		for $ t\in[0, T]$,	where
		\begin{align*}
			\upalpha_n &= C\|\v_n\|^2_{\V} + C \|\z_n\|^4_{\widetilde{\L}^4}, \\
			\upbeta_n &= \|\v\|_{\H}\|\v\|_{\V}\|\hat{ \z}_n\|^2_{\widetilde{\L}^4} +  \|\v_n\|_{\H}\|\v_n\|_{\V}\|\hat{ \z}_n\|^2_{\widetilde{\L}^4}+  \|\z\|^2_{\widetilde{\L}^4}\|\hat{ \z}_n\|^2_{\widetilde{\L}^4}+  \|\z_n\|^2_{\widetilde{\L}^4}\|\hat{ \z}_n\|^2_{\widetilde{\L}^4} \\&\quad+ \big(\|\v_n \|^r_{\widetilde{\L}^{r+1}}+\|\z_n \|^r_{\widetilde{\L}^{r+1}}\big)\|\hat{ \z}_n \|_{\widetilde{\L}^{r+1}} + \big(\|\v \|^r_{\widetilde{\L}^{r+1}}+\|\z \|^r_{\widetilde{\L}^{r+1}}\big)\|\hat{ \z}_n \|_{\widetilde{\L}^{r+1}} \\&\quad+ \|\hat{ \z}_n\|^2_{\H}+\|\hat{\z}_n\|_{\wi\L^{r+1}}^{r+1} + (\|\v_n\|_{\V}+\|\v\|_{\V})\|\hat{\f}_n\|_{\V'}.
		\end{align*}
		Then by the Gronwall inequality, we find 
		\begin{align}\label{Energy_esti_n_3_1}
			\|\y_n(t)\|^2_{\H}
			&\leq \left(\|\y_n(0)\|^2_{\H}+ C \int_{0}^{T} \upbeta_n(s) \d s\right) e^{\int_{0}^{T} \upalpha_n (s) \d s},
		\end{align}
		for all $t\in [0,T]$.	On the other hand, we have
		\begin{align*}
			&\int_{0}^{T}\upbeta_n(s) \d s \\&=\int_{0}^{T} \bigg[\|\v(s)\|_{\H}\|\v(s)\|_{\V}\|\hat{ \z}_n(s)\|^2_{\widetilde{\L}^4} +  \|\v_n(s)\|_{\H}\|\v_n(s)\|_{\V}\|\hat{ \z}_n(s)\|^2_{\widetilde{\L}^4} +  \|\z(s)\|^2_{\widetilde{\L}^4}\|\hat{ \z}_n(s)\|^2_{\widetilde{\L}^4}\\&\quad+  \|\z_n(s)\|^2_{\widetilde{\L}^4}\|\hat{ \z}_n(s)\|^2_{\widetilde{\L}^4}+ \|\hat{ \z}_n(s)\|^2_{\H}+ (\|\v_n(s)\|_{\V}+\|\v(s)\|_{\V})\|\hat{\f}_n(s)\|_{\V'} +\|\hat{\z}_n(s)\|^{r+1}_{\wi\L^{r+1}}\\&\quad+  \big(\|\v_n(s) \|^r_{\widetilde{\L}^{r+1}}+\|\z_n(s) \|^r_{\widetilde{\L}^{r+1}}+\|\v(s) \|^r_{\widetilde{\L}^{r+1}}+\|\z(s) \|^r_{\widetilde{\L}^{r+1}}\big)\|\hat{ \z}_n(s) \|_{\widetilde{\L}^{r+1}} \bigg]\d s\\ &\leq  \bigg[\|\v\|_{\mathrm{L}^{\infty}(0, T; \H)} \|\v\|_{\mathrm{L}^2(0, T; \V)} + \|\v_n\|_{\mathrm{L}^{\infty}(0, T; \H)} \|\v_n\|_{\mathrm{L}^2(0, T; \V)}+ \|\z\|^2_{\mathrm{L}^4(0, T, \widetilde{\L}^4)} \\ &\quad+ \|\z_n\|^2_{\mathrm{L}^4(0, T, \widetilde{\L}^4)}\bigg]\|\hat{ \z}_n\|^2_{\mathrm{L}^4(0, T, \widetilde{\L}^4)} + \bigg[\|\v_n\|^r_{\mathrm{L}^{r+1}(0, T, \widetilde{\L}^{r+1})}+\|\z_n\|^r_{\mathrm{L}^{r+1}(0, T, \widetilde{\L}^{r+1})}\\&\quad+\|\v\|^r_{\mathrm{L}^{r+1}(0, T, \widetilde{\L}^{r+1})}+\|\z\|^r_{\mathrm{L}^{r+1}(0, T, \widetilde{\L}^{r+1})}\bigg]\|\hat{ \z}_n\|_{\mathrm{L}^{r+1}(0, T, \widetilde{\L}^{r+1})}+  \|\hat{ \z}_n\|^2_{\mathrm{L}^2(0, T, \H)}\\& \quad+ \|\hat{ \z}_n\|^{r+1}_{\mathrm{L}^{r+1}(0, T, \widetilde{\L}^{r+1})}+  \big(\|\v_n\|^2_{\mathrm{L}^2(0, T, \V)}+\|\v\|^2_{\mathrm{L}^2(0, T, \V)}\big)\|\f_n\|^2_{\mathrm{L}^2(0, T, \V')}.
		\end{align*}
		Making use of Remark \ref{RemarkI}, we obtain for $r\in[1,3]$, $\z_n\to\z$ in $\mathrm{L}^{r+1}(0, T, \widetilde{\L}^{r+1})$. Therefore, form previous arguments and our assumptions, we have that for $r\in[1,3]$, $\z_n\to\z$ in $\mathrm{L}^{2}(0, T, \H)\cap\mathrm{L}^{4}(0, T, \widetilde{\L}^{4})\cap\mathrm{L}^{r+1}(0, T, \widetilde{\L}^{r+1})$, and hence $\int_{0}^{T} \upbeta_n(s)\ \d s \to 0$ as $n\to \infty.$ Furthermore, we find  
		\begin{align*}
			\int_{0}^{T}\upalpha_n(s) \d s &\leq C\int_{0}^{T} \left[\|\v_n(s)\|^2_{\V} +  \|\z_n(s)\|^4_{\widetilde{\L}^4}\right] \d s= C\left[\|\v_n\|^2_{\mathrm{L}^2(0, T; \V)} +  \|\z_n\|^4_{\mathrm{L}^4(0, T; \widetilde{\L}^4)}\right] ,
		\end{align*}
		which is finite.	Since, $\|\y_n(0)\|_{\H} = \|\boldsymbol{x}_n- \boldsymbol{x}\|_{\H} \to 0 $ and $\int_{0}^{T} \upbeta_n(s)\ \d s \to 0$ as $n\to \infty$ and for all $n\in \mathbb{N},$ $\int_{0}^{T}\upalpha_n(s)\ \d s < \infty$, then \eqref{Energy_esti_n_3_1} asserts that $\|\y_n(t)\|_{\H}\to 0$ as $n\to\infty$ uniformly in $t\in[0, T].$ Since $\v_n(\cdot)$ and $\v(\cdot)$ are continuous, we further have  $$\v(\cdot, \z_n)\boldsymbol{x}_n \to \v(\cdot, \z)\boldsymbol{x}\  \text{ in } \ \C([0, T]; \H).$$
		By \eqref{Energy_esti_n_2_1}, we also get
		\begin{align*}
			&	\mu \int_{0}^{T} \|\y_n(s)\|^2_{\V}\d s+\frac{\beta}{2^{2r}}\int_{0}^{T}\|\y_n(s)\|_{\wi\L^{r+1}}^{r+1}\d s\nonumber\\
			& \leq \|\y_n(0)\|^2_{\H} + \sup_{s  \in [0, T]}\|\y_n(s)\|^2_{\H}\int_{0}^{T} \upalpha_n (s) \ \d s + \frac{8}{\mu} \int_{0}^{T} \upbeta_n(s)\ \d s\to 0,
		\end{align*}
		as $n\to \infty$ and therefore, $\v(\cdot, \z_n)\boldsymbol{x}_n \to \v(\cdot, \z)\boldsymbol{x}$ in $\mathrm{L}^2(0, T; \V)\cap\mathrm{L}^{r+1}(0, T; \widetilde{\L}^{r+1}),$ which completes the proof.
	\end{proof}
	\begin{definition}
		We define a map $\varphi_{\chi} : \mathbb{R}^+ \times \Omega_i \times \H \to \H$ by
		\begin{align}
			(t, \omega, \boldsymbol{x}) \mapsto \v^{\chi}(t)  + \z_{\chi}(\omega)(t) \in \H,
		\end{align}
		where $\v^{\chi}(t) = \v(t, \z_{\chi}(\omega)(t))(\boldsymbol{x} - \z_{\chi}(\omega)(0))$ is a solution to the system \eqref{cscbf} with the initial condition $\boldsymbol{x} - \z_{\chi}(\omega)(0).$
	\end{definition}
	\begin{proposition}\label{alpha_ind}
		If $\chi_1, \chi_2 \geq 0$, then $\varphi_{\chi_1} = \varphi_{\chi_2}.$
	\end{proposition}
	\begin{proof}
		Let us fix $\boldsymbol{x}\in \H.$ We need to prove that $$\v^{\chi_1}(t) + \z_{\chi_1}(t) = \v^{\chi_2}(t) + \z_{\chi_2}(t), \ \ \ t\geq 0, $$ where $\z_{\chi}$ is defined by \eqref{DOu1} and $\v^{\chi}$ is a solution to the system \eqref{cscbf}. From \eqref{cscbf}, we infer that $\v^{\chi_1}(0) - \v^{\chi_2}(0) = - (\z_{\chi_1}(0) - \z_{\chi_2}(0)) $ and 
		\begin{align*}
			\frac{\d(\v^{\chi_1}(t) - \v^{\chi_2}(t))}{\d t} = &-\mu \A(\v^{\chi_1}(t) - \v^{\chi_2}(t))-\alpha(\v^{\chi_1}(t) - \v^{\chi_2}(t)) \nonumber\\&+ \big[(\chi_1-\alpha) \z_{\chi_1}(t) -  (\chi_2-\alpha) \z_{\chi_2}(t)\big]\nonumber\\ &- [\B(\v^{\chi_1}(t) + \z_{\chi_1}(t))-\B(\v^{\chi_2}(t) + \z_{\chi_2}(t))] \nonumber\\&- \beta [\mathcal{C}(\v^{\chi_1}(t) + \z_{\chi_1}(t)) - \mathcal{C}(\v^{\chi_2}(t) + \z_{\chi_2}(t))], 
		\end{align*}
		for a.e. $t\in[0,T]$	in $\V'+\wi\L^{\frac{r+1}{r}}$.	Adding the equation \eqref{Dif_z1} to the above equation, we obtain 
		\begin{align}\label{Enrgy}
			\frac{\d(\u^{\chi_1}(t) - \u^{\chi_2}(t))}{\d t}& = -\mu \A(\u^{\chi_1}(t) - \u^{\chi_2}(t)) - [\B(\u^{\chi_1}(t))-\B(\u^{\chi_2}(t))] \nonumber\\&\quad-\alpha(\u^{\chi_1}(t) - \u^{\chi_2}(t)) - \beta [\mathcal{C}(\u^{\chi_1}(t)) - \mathcal{C}(\u^{\chi_2}(t) )], 
		\end{align}
		for a.e. $t\in[0,T]$	in $\V'+\wi\L^{\frac{r+1}{r}}$,	where $\u^{\chi_1}(t)=\v^{\chi_1}(t) + \z_{\chi_1}(t),   \u^{\chi_2}(t) = \v^{\chi_2}(t) + \z_{\chi_2}(t), \  t\geq 0$ and $\u^{\chi_1}(0) - \u^{\chi_2}(0)= \mathbf{0}.$ 
		
		Taking the inner product with $\u^{\chi_1}(t) - \u^{\chi_2}(t)$ in \eqref{Enrgy} and using Theorem \ref{LocMon}, we obtain
		\begin{align*}
			\frac{\d}{\d t} \|\u^{\chi_1}(t) - \u^{\chi_2}(t)\|^2_\H\leq  \begin{cases}
				\frac{27}{16\mu ^3}\|\u_2 (t)\|^4_{\widetilde{\L}^4}\|\u_2(t)-\u_2(t)\|_{\H}^2, &\text{ for } d=2 \text{ with }r\in[1,3]\\
				\eta\|\u_2(t)-\u_2(t)\|_{\H}^2, &\text{ for } d=2,3 \text{ with } r>3, \\
				0, &\text{ for } d=r=3 \text{ with } 2\beta\mu\geq1,
			\end{cases} 
		\end{align*}
		for a.e. $t\in[0,T]$. 	Since $\frac{27}{16\mu ^3}\int_{0}^{t}\|\u_2 (\tau)\|^4_{\widetilde{\L}^4} \d\tau< \infty$ (for $d=2$ with $r\in[1,3]$) and $\|\u^{\chi_1}(0) - \u^{\chi_2}(0)\|^2_\H = 0,$ by applying the Gronwall inequality, we deduce  that $\|\u^{\chi_1}(t) - \u^{\chi_2}(t)\|^2_\H = 0$, for all $t\geq 0$, which completes the proof.
	\end{proof}
	It is proved in Proposition \ref{alpha_ind} that the map $\varphi_{\chi}$ does not depend on $\chi$ and hence, from now onward, it will be denoted by $\varphi$. A proof of the following result is similar to that in \cite[Theorem 6.15]{BL} and hence we omit it here.

	\begin{theorem}
		$(\varphi, \theta)$ is an RDS.
	\end{theorem}

	\section{Random attractors for SCBF equations}\label{sec6}\setcounter{equation}{0}
In this section, we prove our main results of this work. Here, the RDS $\varphi$ is considered over the MDS $(\Omega_i, \hat{\mathcal{F}}_i, \hat{\mathbb{P}}_i, \hat{\theta})$.	The results that we have obtained in the previous sections  provide a unique solution to the system \eqref{S-CBF}, which is continuous with respect to the data (particularly $\u_0$ and $\f$). Furthermore, if we define, for $\u_0 \in \H,\ \omega \in \Omega_i,$ and $t\geq s,$
	\begin{align}\label{combine_sol}
		\u(t, s;\omega, \u_0) := \varphi(t-s; \theta_s \omega)\u_0 = \v\big(t, s; \omega, \u_0 - \z(s)\big) + \z(t),
	\end{align}
	then the process $\{\u(t): \ t\geq s\},$ is a solution to the system \eqref{S-CBF}, for each $s\in \mathbb{R}$ and each $\u_0 \in \H$.

	\begin{lemma}\label{RA1}
		Suppose that $\v$ solves the system \eqref{cscbf} on the time interval $[a, \infty)$ with $\z \in \mathrm{L}^{4}_{\emph{loc}}( \R^+;\widetilde{\L}^{4})\cap \mathrm{L}^2_{\emph{loc}}(\R^+, \H)$ (for $r\in[1,3)$) and with $\z \in \mathrm{L}^{r+1}_{\emph{loc}}( \R^+; \widetilde{\L}^{r+1})\cap \mathrm{L}^2_{\emph{loc}}(\R^+, \H)$ (for $r\geq3$), and $\chi\geq 0.$ Then, for any $t\geq \tau \geq a,$
		\begin{align}\label{Energy_esti1}
			\|\v(t)\|^2_{\H} \leq \begin{cases}
				\|\v(\tau)\|^2_{\H}\  e^{-2\alpha(t-\tau) +R \int_{\tau}^{t}\|\z(s)\|^{4}_{\widetilde{\L}^{4}}\d s}  + C\int_{\tau}^{t} \bigg[\|\z(s)\|^{2}_{\H}+\|\z(s)\|^{4}_{\widetilde{\L}^{4}}+\|\z(s)\|^{r+1}_{\widetilde{\L}^{r+1}}  \\ \qquad\qquad+ \|\f\|^2_{\V'}\bigg] e^{-2\alpha(t-s) +R \int_{s}^{t}\|\z(\zeta)\|^{4}_{\widetilde{\L}^{4}}\d\zeta} \d s,\qquad \text{ for } d=2 \text{ with } r\in[1,3),\\	\|\v(\tau)\|^2_{\H}\  e^{-2\alpha(t-\tau)}  + C\int_{\tau}^{t} \bigg[\|\z(s)\|^{2}_{\H}+\|\z(s)\|^{4}_{\widetilde{\L}^{4}}+\|\z(s)\|^{r+1}_{\widetilde{\L}^{r+1}}  \\ \qquad\qquad+ \|\f\|^2_{\V'}\bigg] e^{-2\alpha(t-s)} \d s,\qquad\qquad\qquad\qquad \text{ for } d=2,3 \text{ with } r\geq3,
			\end{cases}
		\end{align}
		and 
		\begin{align}\label{Energy_esti2}
			\|\v(t)\|^2_{\H} =& \|\v(\tau)\|^2_{\H} e^{- 2\alpha(t-\tau)} + 2 \int_{\tau}^{t}\ e^{- 2\alpha(t-s)}\bigg[b(\v(s), \v(s), \z(s))-b(\z(s), \z(s), \v(s))\nonumber\\
			& - \beta\big\langle\mathcal{C}(\v(s) +\z(s)),\z(s)\big\rangle  +(\chi-\alpha)(\z(s),\v(s)) +\big\langle \f, \v(s)\big \rangle - \mu\|\v(s)\|^2_{\V}\nonumber\\&-\beta\|\v(s)+\z(s)\|^{r+1}_{\wi\L^{r+1}}\bigg] \d s,
		\end{align}
		where $\ t\in [a, \infty).$
	\end{lemma}
	\begin{proof}
		From \eqref{cscbf}, we obtain 
		\begin{align}\label{Energy_esti3}
			\frac{1}{2}\frac{\d}{\d t} \|\v(t)\|^2_{\H} 
			&=  - \mu \|\v(t)\|^2_{\V}-\alpha\|\v(t)\|^2_{\H} -b(\v(t), \z(t), \v(t)) - \beta\big\langle\mathcal{C}(\v(t)+\z(t)), \v(t) \big\rangle \nonumber\\
			&\quad- b(\z(t),\z(t) ,\v(t)) +(\chi-\alpha)(\z(t),\v(t))+\big\langle \f, \v(t)\big \rangle \nonumber\\
			&	=  - \mu \|\v(t)\|^2_{\V}-\alpha\|\v(t)\|^2_{\H}  - \beta\|\v(t)+\z(t)\|^{r+1}_{\wi\L^{r+1}}+b(\v(t), \v(t), \z(t))\nonumber\\
			&\quad+\beta\big\langle\mathcal{C}(\v(t)+\z(t)), \z(t)\big\rangle  - b(\z(t),\z(t) ,\v(t)) +(\chi-\alpha)(\z(t),\v(t))\nonumber\\&\quad +\big\langle \f, \v(t)\big \rangle, 
		\end{align}
	for a.e. $t\in[0,T]$.	For $d=2$ with $r\in[1,3)$, by using H\"older's and Young's inequalities, and \eqref{lady},  we obtain
		\begin{align*}
			|b(\v, \v, \z)|  \leq &  \|\v\|_{\widetilde{\L}^{4}} \|\v\|_{\V} \|\z\|_{\widetilde{\L}^{4}}	\leq 2^{\frac{1}{4}}  \|\v\|^{\frac{1}{2}}_{\H} \|\v\|^{\frac{3}{2}}_{\V} \|\z\|_{\widetilde{\L}^{4}}\leq \frac{\mu}{4} \|\v\|^2_{\V} + \frac{R}{2} \|\v\|^2_{\H}\ \|\z\|^{4}_{\widetilde{\L}^{4}},
		\end{align*}
		where $R=\frac{729 }{8 \mu^3}$. For $d=2,3$ with $r\geq3$, by using H\"older's, interpolation and Young's inequalities, we obtain (taking without loss of generality that $r>3,$ but the final estimate holds for $r=3$ also)
		\begin{align*}
			|b(\v, \v, \z)| & \leq  \|\v\|_{\widetilde{\L}^{r+1}} \|\v\|_{\V} \|\z\|_{\widetilde{\L}^{\frac{2(r+1)}{r-1}}}\leq \|\v\|_{\widetilde{\L}^{r+1}} \|\v\|_{\V} \|\z\|^{\frac{2}{r-1}}_{\widetilde{\L}^{r+1}}\|\z\|^{\frac{r-3}{r-1}}_{\H}\nonumber\\
			&\leq\|\v+\z\|_{\widetilde{\L}^{r+1}} \|\v\|_{\V} \|\z\|^{\frac{2}{r-1}}_{\widetilde{\L}^{r+1}}\|\z\|^{\frac{r-3}{r-1}}_{\H}+ \|\v\|_{\V} \|\z\|^{\frac{r+1}{r-1}}_{\widetilde{\L}^{r+1}}\|\z\|^{\frac{r-3}{r-1}}_{\H}\nonumber\\
			&\leq\frac{\beta}{4}\|\v+\z\|^{r+1}_{\widetilde{\L}^{r+1}}+\frac{\mu}{4} \|\v\|^2_{\V}+ C\|\z\|^{r+1}_{\widetilde{\L}^{r+1}}+C\|\z\|^{2}_{\H}.
		\end{align*}
		For $r\in[1,\infty)$,	using Young's and H\"older's inequalities, we also obtain 
		\begin{align*}
			\beta\big\langle\mathcal{C}(\v+\z),\v+\z\big\rangle &=   \beta\|\v+\z\|^{r+1}_{\widetilde{\L}^{r+1}},\\
			|\beta\big\langle\mathcal{C}(\v+\z),\z\big\rangle|&\leq  \beta \|\v+\z\|^{r}_{\widetilde{\L}^{r+1}} \|\z\|_{\widetilde{\L}^{r+1}}\leq  \frac{\beta}{4} \|\v+\z\|^{r+1}_{\widetilde{\L}^{r+1}} + C\|\z\|^{r+1}_{\widetilde{\L}^{r+1}},\\
			|\big\langle(\chi-\alpha)\z- \B(\z)+\f, \v\big \rangle| &\leq \big(\left|\chi-\alpha\right|\|\z\|_{\V'} +\|\B(\z)\|_{\V'}+\|\f\|_{\V'}\big) \|\v\|_{\V}\nonumber\\&\leq  \frac{\mu}{4} \|\v\|^2_{\V} +C\|\z\|^2_{\H}+C\|\z\|^4_{\widetilde{\L}^{4}}+C\|\f\|^2_{\V'}.
		\end{align*}
		Hence, from \eqref{Energy_esti3}, we deduce that 
		\begin{align*}
			\frac{\d}{\d t} \|\v(t)\|^2_{\H}\leq  \begin{cases}
				\left[-2\alpha +R \|\z(t)\|^{4}_{\widetilde{\L}^{4}}\right]\|\v(t)\|^2_{\H}+C\left[\|\z(t)\|^{2}_{\H}+\|\z(t)\|^{4}_{\widetilde{\L}^{4}}+\|\z(t)\|^{r+1}_{\widetilde{\L}^{r+1}}  + \|\f\|^2_{\V'}\right],\\
				\qquad\qquad\qquad\qquad\qquad\qquad\qquad\qquad\qquad\qquad\qquad\text{ for } d=2 \text{ with } r\in[1,3),\\
				-2\alpha\|\v(t)\|^2_{\H}+C\left[\|\z(t)\|^{2}_{\H}+\|\z(t)\|^{4}_{\widetilde{\L}^{4}}+\|\z(t)\|^{r+1}_{\widetilde{\L}^{r+1}}  + \|\f\|^2_{\V'}\right], \\ \qquad\qquad\qquad\qquad\qquad\qquad\qquad\qquad\qquad\qquad\qquad\text{ for } d=2,3 \text{ with } r\geq3,
			\end{cases}
		\end{align*}
		and an application of  Gronwall's inequality yields \eqref{Energy_esti1}. Now, applying the variation of constants formula to \eqref{Energy_esti3}, we obtain \eqref{Energy_esti2} immediately.
	\end{proof}
	
	\begin{lemma}\label{weak_topo1}
		Let	$\v(t, \v_0)$ be the unique solution to the initial value problem \eqref{cscbf} with the initial condition  $\v_0 \in \H $, and with a deterministic function $\z \in \mathrm{L}^4_{\emph{loc}}(\R^+; \widetilde{\L}^4)\cap \mathrm{L}^2_{\emph{loc}} (\R^+; \H)$ (for $r\in[1,3)$) and $\z \in \mathrm{L}^{r+1}_{\emph{loc}}(\R^+; \widetilde{\L}^{r+1})\cap \mathrm{L}^2_{\emph{loc}} (\R^+; \H)$ (for $r\geq3$). For $T>0,$ if $\boldsymbol{y}_n$ converges to $\boldsymbol{y}$ in $\H$ weakly, then $\v(\cdot, \boldsymbol{y}_n)$ converges to $\v(\cdot, \boldsymbol{y})$ in $\mathrm{L}^2 (0, T; \V )\cap\mathrm{L}^{r+1}(0, T; \widetilde{\L}^{r+1})$ weakly.
	\end{lemma}
	\begin{proof}
		Assume that $\{\boldsymbol{y}_n\}_{n\in \N}$ is an $\H$-valued sequence such that $\boldsymbol{y}_n$ converges to $\boldsymbol{y}\in \H$ weakly. Let $\v_n(\cdot)= \v(\cdot, \boldsymbol{y}_n)$ and $\v(\cdot)=\v(\cdot, \boldsymbol{y})$. Since  $\{\boldsymbol{y}_n\}_n$ is a bounded sequence in $\H$,
		\begin{align}\label{Bounded}
			\text{ the sequence } \ \{\v_n\}_{n\in \N} \ \text{ is bounded in }\  \mathrm{L}^{\infty}(0, T; \H)\cap\mathrm{L}^{2}(0, T; \V)\cap\mathrm{L}^{r+1}(0, T; \widetilde{\L}^{r+1}).
		\end{align}
		Hence, there exists a subsequence $\{\v_{n'}\}_{n'\in \N}$ of $\{\v_n\}_{n\in \N}$ and  $$\widetilde{\v}\in \mathrm{L}^{\infty}(0, T; \H)\cap\mathrm{L}^{2}(0, T; \V)\cap\mathrm{L}^{r+1}(0, T; \widetilde{\L}^{r+1}),$$ such that, as $n' \to \infty,$ (by the Banach-Alaoglu theorem)
		\begin{align}\label{lim1}
			\begin{cases}
				\v_{n'} \xrightharpoonup{w^*} \widetilde{\v} \ \text{  in } \ \mathrm{L}^{\infty}(0, T; \H)\\ \v_{n'} \xrightharpoonup{w} \widetilde{\v} \ \text{  in }\ \mathrm{L}^{2}(0, T; \V)\cap\mathrm{L}^{r+1}(0, T; \widetilde{\L}^{r+1}).
			\end{cases}
		\end{align}
		Moreover, $\v_{n'} \to \widetilde{\v}$ strongly in $\mathrm{L}^2(0, T; {\L}^2_{\text{loc}}(\mathcal{O})).$ Using Corollaries \ref{convergence_b1}, \ref{convergence_b2} and \ref{convergence_c3_1}, we can conclude that $\widetilde{\v}$ is a solution of \eqref{cscbf} with $\widetilde{\v}(0)=\boldsymbol{y}$. Since \eqref{cscbf} has unique solution, we infer that $\widetilde{\v}=\v.$ By a contradiction argument, we infer that the whole sequence $\{\v_n\}_{n\in \N}$ converges to $\v$ in $\mathrm{L}^{2}(0, T; \V)\cap\mathrm{L}^{r+1}(0, T; \widetilde{\L}^{r+1})$ weakly.
	\end{proof}

	\begin{lemma}\label{weak_topo2}
		Let	$\v(t, \v_0)$ be the unique solution to the initial value problem \eqref{cscbf} with initial condition \ $\v_0 \in \H $, and with a deterministic function $\z \in \mathrm{L}^4_{\emph{loc}}(\R^+; \widetilde{\L}^4 (\mathcal{O}))\cap \mathrm{L}^2_{\emph{loc}} (\R^+; \H)$ (for $r\in[1,3)$) and $\z \in \mathrm{L}^{r+1}_{\emph{loc}}(\R^+; \widetilde{\L}^{r+1} (\mathcal{O}))\cap \mathrm{L}^2_{\emph{loc}} (\R^+; \H)$ (for $r\geq3$). For $T>0,$ if $\boldsymbol{y}_n$ converges to $\boldsymbol{y}$ in $\H$ weakly, then for any $\phi\in \H, (\v(\cdot, \boldsymbol{y}_n), \phi)$ converges uniformly to $(\v(\cdot, \boldsymbol{y}), \phi)$ on $[0, T]$, as $n \to \infty.$
	\end{lemma}
	\begin{proof}
		Assume that $\{\boldsymbol{y}_n\}_{n\in \N}$ is an $\H$-valued sequence such that $\boldsymbol{y}_n$ converges to $\boldsymbol{y}\in \H$ weakly. Let $\v_n(t)= \v(t, \boldsymbol{y}_n)$ and $\v(t)=\v(t, \boldsymbol{y})$. From the  proof of Lemma \ref{weak_topo1}, we infer that  \eqref{Bounded} and \eqref{lim1} hold true. Take any function $\phi \in \mathcal{V}.$ Then, by \eqref{lim1}, for a.e. $t\in[0, T], (\v_n(t), \phi)$ converges to $(\v(t), \phi).$ Furthermore, since $\{\v_n\}_{n\in\N}$ is a bounded sequence in $\mathrm{L}^{\infty}(0,T; \H), \{(\v_n(\cdot), \phi)\}_n$ is uniformly bounded on $[0, T].$

		Also, from Theorem \ref{solution}, we have $\big\|\frac{\d\v_{n}}{\d t}\big\|_{\mathrm{L}^{\frac{r+1}{r}}(0, T; \V'+\widetilde{\L}^{\frac{r+1}{r}})}\leq C,$ for some $C>0$ and all $n\in \N.$ Hence by the Cauchy-Schwartz inequality, for all $0\leq t \leq t+a \leq T$ and $n\in \N,$ we obtain
		\begin{align*}
			|(\v_n(t+a)-\v_n(t), \phi)|\leq \int_{t}^{t+a}\bigg|\bigg\langle\frac{\d\v_{n}(s)}{\d t} , \phi \bigg\rangle\bigg|\d s\leq C \|\phi\|_{\V\cap\widetilde{\L}^{r+1}} a^{\frac{1}{r+1}}.
		\end{align*}
		This shows that the sequence $\{(\v_n(\cdot), \phi)\}_{n\in \N}$ is uniformly equicontinuous on $[0, T].$ Hence, by the Arzela-Ascoli theorem, there exists a subsequence $\{(\v_{n'}(\cdot), \phi)\}_{n'\in \N}$ of $\{(\v_n(\cdot), \phi)\}_{n\in \N}$, such that $(\v_{n'}(\cdot), \phi)$ converges to $(\v(\cdot), \phi)$ uniformly on $[0, T].$ Again, using the standard contradiction argument, we assert that 
		\begin{align*}
			(\v_{n}(\cdot), \phi) \to (\v(\cdot), \phi) \text{ uniformly on  } [0, T].
		\end{align*}
		Using the density of $\mathcal{V}$ in $\H$ and $\sup\limits_{t  \in [0, T]} \|\v_n(t)\|_{\H}< \infty$,  for any $\phi\in \H,$
		\begin{align*}
			(\v_{n}(t), \phi) \to (\v(\cdot), \phi) \text{ uniformly on  } [0, T],
		\end{align*}
		which completes the proof.
	\end{proof}
	\begin{lemma}\label{Bddns4}
		1. 	For $r\in[1,3)$ and each $\omega\in \Omega_1,$ we have 
		\begin{align*}
			\limsup_{t\to - \infty} \|\z(\omega)(t)\|^2_{\H}\  e^{2\alpha t +R\int_{t}^{0}\|\z(\zeta)\|^{4}_{\widetilde{\L}^{4}}\d\zeta} = 0.
		\end{align*}	
		2. 	For $r\geq3$ and each $\omega\in \Omega_2,$ we obtain 
		\begin{align*}
			\limsup_{t\to - \infty} \|\z(\omega)(t)\|^2_{\H}\  e^{2\alpha t} = 0.
		\end{align*}
	\end{lemma}
	\begin{proof}
			Let us fix $\omega\in \Omega_i$, for $i=1,2$. Then by Corollary \ref{Bddns1_1}, there exists a $t_0\leq 0$ such that for $t\leq t_0,$
		\begin{align}\label{Bddns3}
			R \int_{t}^{0} \|\z(s)\|^{4}_{\widetilde{\L}^4} \d s \leq - \alpha t, \ \ \ t\leq t_0.
		\end{align}
		Because of \eqref{X_bound_of_z}, there exists a $\rho_1= \rho_1(\omega)\geq 0$ such that,
		\begin{align}\label{rho}
			\frac{\|\z(t)\|_{\H}}{|t|} \leq \rho_1 , \  \frac{\|\z(t)\|_{\widetilde{\L}^4}}{|t|} \leq \rho_1, \ \text{ and } \  \frac{\|\z(t)\|_{\widetilde{\L}^{r+1}}}{|t|} \leq \rho_1\ \text{ for }\  t\leq t_0.
		\end{align}
		Therefore, we have, for every $\omega\in \Omega_1,$
		\begin{align*}
			\limsup_{t\to - \infty} \|\z(\omega)(t)\|^2_{\H}\  e^{2\alpha t +R \int_{t}^{0}\|\z(\zeta)\|^{4}_{\widetilde{\L}^{4}}\d\zeta}\leq&  \rho_1^2 \limsup_{t\to - \infty}  |t|^2 e^{\alpha t }=0,
		\end{align*}
		and for every $\omega\in \Omega_2,$
		\begin{align*}
			\limsup_{t\to - \infty} \|\z(\omega)(t)\|^2_{\H}\  e^{2\alpha t}\leq&  \rho_1^2 \limsup_{t\to - \infty}  |t|^2 e^{2\alpha t} =0,
		\end{align*}
		which completes the proof.
	\end{proof}
	\begin{lemma}\label{Bddns5}
		1. 	For $r\in[1,3)$ and each $\omega\in \Omega_1,$ we have 
		\begin{align*}
			\int_{- \infty}^{0} \bigg\{ 1 + \|\z(t)\|^2_{\H} + \|\z(t)\|^4_{\widetilde{\L}^4}+ \|\z(t)\|^{r+1}_{\widetilde{\L}^{r+1}}   \bigg\}e^{2\alpha t +R \int_{t}^{0}\|\z(\zeta)\|^{4}_{\widetilde{\L}^{4}}\d\zeta} \d t < \infty.
		\end{align*}	
		2. For $r\geq3$ and each $\omega\in \Omega_2,$ we get 
		\begin{align*}
			\int_{- \infty}^{0} \bigg\{ 1 + \|\z(t)\|^2_{\H} + \|\z(t)\|^4_{\widetilde{\L}^4}+ \|\z(t)\|^{r+1}_{\widetilde{\L}^{r+1}}   \bigg\}e^{2\alpha t} \d t < \infty.
		\end{align*}
	\end{lemma}
	\begin{proof}
		We first consider the case $r\in[1,3)$. Note that for $t_0\leq 0$,
		\begin{align*}
			\int_{t_0}^{0} \bigg\{ 1 + \|\z(t)\|^2_{\H} + \|\z(t)\|^4_{\widetilde{\L}^4}+ \|\z(t)\|^{r+1}_{\widetilde{\L}^{r+1}} \bigg\}e^{2\alpha t +R \int_{t}^{0}\|\z(\zeta)\|^{4}_{\widetilde{\L}^{4}}\d\zeta} \d t < \infty.
		\end{align*}
		Therefore, we only need to show that the integral 
		\begin{align*}
			\int_{- \infty}^{t_0} \bigg\{ 1+ \|\z(t)\|^2_{\H} + \|\z(t)\|^4_{\widetilde{\L}^4}  + \|\z(t)\|^{r+1}_{\widetilde{\L}^{r+1}} \bigg\}e^{2\alpha t +R \int_{t}^{0}\|\z(\zeta)\|^{4}_{\widetilde{\L}^{4}}\d\zeta} \d t < \infty.
		\end{align*}
		Using the estimate \eqref{Bddns3}, we find 
		\begin{align*}
			\int_{- \infty}^{t_0} e^{2\alpha t +R \int_{t}^{0}\|\z(\zeta)\|^{4}_{\widetilde{\L}^{4}}\d\zeta} \d t\leq \int_{- \infty}^{t_0} e^{\alpha t} \d t< \infty.
		\end{align*}
		Making use of \eqref{Bddns3} and \eqref{rho}, we obtain 
		\begin{align*} 
			&	\int_{- \infty}^{t_0} \bigg\{ \|\z(t)\|^2_{\H} + \|\z(t)\|^4_{\widetilde{\L}^4}+\|\z(t)\|^{r+1}_{\widetilde{\L}^{r+1}}  \bigg\}e^{2\alpha t +R \int_{t}^{0}\|\z(\zeta)\|^{4}_{\widetilde{\L}^{4}}\d\zeta} \d t\\& \leq \int_{- \infty}^{t_0} \big\{ \rho^2_1|t|^2+ \rho^4_1|t|^4+\rho^{r+1}_1|t|^{r+1}  \big\}e^{\alpha t } \d t< \infty,
		\end{align*}
		which  completes the proof for $r\in[1,3)$. For the case $r\geq3$, using \eqref{rho}, the proof is immediate by applying similar arguments as in the previous case.
	\end{proof}
	
	\begin{definition}\label{RA2}
		A function $\kappa: \Omega_1\to (0, \infty)$ belongs to class $\mathfrak{K}_1$ if and only if 
		\begin{align}
			\limsup_{t\to \infty} [\kappa(\theta_{-t}\omega)]^2 e^{-2\alpha t +R \int_{-t}^{0}\|\z(\omega)(s)\|^{4}_{\widetilde{\L}^{4}}\d s} = 0, 
		\end{align}
		where $R=\frac{729 }{8 \mu^3}$ and $\alpha>0$ is Darcy's constant.
		
		A function $\widetilde{\kappa}: \Omega_2\to (0, \infty)$ belongs to class ${\mathfrak{K}}_2$ if and only if 
		\begin{align}
			\limsup_{t\to \infty} [\widetilde{\kappa}(\theta_{-t}\omega)]^2 e^{-2\alpha t} = 0, 
		\end{align}
		where $\alpha>0$ is Darcy's constant.
	\end{definition}
	Let us denote the class of all closed and bounded random sets $D_1$ on $\H$ by $\mathfrak{DK}_1,$ such that the radius function $\Omega_1\ni \omega \mapsto \kappa(D_1(\omega)):= \sup\{\|x\|_{\H}:x\in D_1(\omega)\}$ belongs to class $\mathfrak{K}_1.$ It is straight forward by Corollary \ref{Bddns1_1} that the constant functions belongs to $\mathfrak{K}_1$. 
	It is clear by the Definition \ref{RA2} that the class $\mathfrak{K}_1$ is closed with respect to sum, multiplication by a constant and if $\kappa \in \mathfrak{K}_1, 0\leq \bar{\kappa} \leq \kappa,$ then $\bar{\kappa}\in \mathfrak{K}_1.$ A similar definition of class $\mathfrak{DK}_2$ can be derived in the case of the class $\mathfrak{K}_2$. 
	
	\begin{proposition}\label{radius}
		For $r\in[1,3)$, we define functions $\kappa_{i}:\Omega_1\to (0, \infty), i= 1, 2, 3, 4, 5, 6,$ by the following formulae, for $\omega\in\Omega_1,$
		\begin{align*}
			[\kappa_1(\omega)]^2 &:= \|\z(\omega)(0)\|_{\H},\ \ \
			[\kappa_2(\omega)]^2 := \sup_{s\leq 0} \|\z(\omega)(s)\|^2_{\H}\  e^{2\alpha s +R \int_{s}^{0}\|\z(\omega)(\zeta)\|^{4}_{\widetilde{\L}^{4}}\d\zeta}, \\
			[\kappa_3(\omega)]^2 &:= \int_{- \infty}^{0} \|\z(\omega)(t)\|^{r+1}_{\widetilde{\L}^{r+1}}\ e^{2\alpha t +R \int_{t}^{0}\|\z(\omega)(\zeta)\|^{4}_{\widetilde{\L}^{4}}\d\zeta} \d t, \\
			[\kappa_4(\omega)]^2 &:= \int_{- \infty}^{0} \|\z(\omega)(t)\|^2_{\H}\ e^{2\alpha t +R \int_{t}^{0}\|\z(\omega)(\zeta)\|^{4}_{\widetilde{\L}^{4}}\d\zeta} \d t,\\
			[\kappa_5(\omega)]^2 &:= \int_{- \infty}^{0} \|\z(\omega)(t)\|^4_{\widetilde{\L}^4}\ e^{2\alpha t +R \int_{t}^{0}\|\z(\omega)(\zeta)\|^{4}_{\widetilde{\L}^{4}}\d\zeta} \d t,\ \ \\
			[\kappa_6(\omega)]^2 &:= \int_{- \infty}^{0} e^{2\alpha t +R \int_{t}^{0}\|\z(\omega)(\zeta)\|^{4}_{\widetilde{\L}^{4}}\d\zeta} \d t.
		\end{align*}
		Then all these functions belongs to class $\mathfrak{K}_1.$
		\begin{proof}Let us recall from \eqref{stationary} that $\z(\theta_{-t}\omega)(s) = \z(\omega)(s-t)$.
			We consider
			\begin{align*}
				\limsup_{t\to  \infty}[\kappa_1(\theta_{-t}\omega)]^2 e^{-2\alpha t +R \int_{-t}^{0}\|\z(\omega)(s)\|^{4}_{\widetilde{\L}^{4}}\d s} =& \limsup_{t\to \infty}\|\z(\theta_{-t}\omega)(0)\|^2_{\H} e^{-2\alpha t +R \int_{-t}^{0}\|\z(\omega)(s)\|^{4}_{\widetilde{\L}^{4}}\d s}\\
				=&\limsup_{t\to \infty}\|\z(\omega)(-t)\|^2_{\H} e^{-2\alpha t +R \int_{-t}^{0}\|\z(\omega)(s)\|^{4}_{\widetilde{\L}^{4}}\d s}.
			\end{align*}
			Using Lemma \ref{Bddns4}, we have, $\kappa_1 \in \mathfrak{K}_1.$ It can be easily seen that 
			\begin{align*}
				[\kappa_2(\theta_{-t}\omega)]^2 
				= & \sup_{s\leq 0}  \|\z(\omega)(s-t)\|^2_{\H}\  e^{2\alpha s +R \int_{s}^{0}\|\z(\omega)(\zeta -t)\|^{4}_{\widetilde{\L}^{4}}\d\zeta}\\
				= & \sup_{s\leq 0}  \|\z(\omega)(s-t)\|^2_{\H}\  e^{2\alpha (s-t) +R \int_{s-t}^{-t}\|\z(\omega)(\zeta)\|^{4}_{\widetilde{\L}^{4}}\d\zeta}\ e^{2\alpha t}\\
				= & \sup_{\sigma\leq -t}  \|\z(\omega)(\sigma)\|^2_{\H}\  e^{2\alpha \sigma +R \int_{\sigma}^{-t}\|\z(\omega)(\zeta)\|^{4}_{\widetilde{\L}^{4}}\d\zeta}\ e^{2\alpha t}
			\end{align*}
			and 
			\begin{align*}
				&	\limsup_{t\to \infty} [\kappa_2(\theta_{-t}\omega)]^2 e^{-2\alpha t +R \int_{-t}^{0}\|\z(\omega)(s)\|^{4}_{\widetilde{\L}^{4}}\d s} \nonumber\\&=\limsup_{t\to \infty} \sup_{\sigma\leq -t}  \|\z(\omega)(\sigma)\|^2_{\H}\  e^{2\alpha \sigma +R \int_{\sigma}^{0}\|\z(\omega)(\zeta)\|^{4}_{\widetilde{\L}^{4}}\d\zeta}\\
				&=\limsup_{\sigma\to -\infty} \|\z(\omega)(\sigma)\|^2_{\H}\  e^{2\alpha \sigma +R \int_{\sigma}^{0}\|\z(\omega)(\zeta)\|^{4}_{\widetilde{\L}^{4}}\d\zeta}
				= 0,
			\end{align*}
			where we have used Lemma \ref{Bddns4}. This implies that $\kappa_2\in \mathfrak{K}_1.$ From the previous part of the proof, we obtain 
			\begin{align*}
				&\bigg\{[\kappa_3(\theta_{-t}\omega)]^2+ [\kappa_4(\theta_{-t}\omega)]^2+ [\kappa_5(\theta_{-t}\omega)]^2 +[\kappa_6(\theta_{-t}\omega)]^2\bigg\} e^{-2\alpha t +R \int_{-t}^{0}\|\z(\omega)(s)\|^{4}_{\widetilde{\L}^{4}}\d s}\\
				&=\int_{- \infty}^{-t} \bigg\{  \|\z(\omega)(t)\|^{r+1}_{\widetilde{\L}^{r+1}} + \|\z(\omega)(t)\|^2_{\H} + \|\z(\omega)(t)\|^4_{\widetilde{\L}^4} + 1 \bigg\}e^{2\alpha \sigma +R \int_{\sigma}^{0}\|\z(\omega)(\zeta)\|^{4}_{\widetilde{\L}^{4}}\d\zeta} \d\sigma.
			\end{align*}
			Invoking Lemma \ref{Bddns5}, we find 
			\begin{align*}
				\int_{- \infty}^{0} \bigg\{ \|\z(\omega)(t)\|^{r+1}_{\widetilde{\L}^{r+1}} + \|\z(\omega)(t)\|^2_{\H} + \|\z(\omega)(t)\|^4_{\widetilde{\L}^4} +1  \bigg\}e^{2\alpha t +R \int_{t}^{0}\|\z(\omega)(\zeta)\|^{4}_{\widetilde{\L}^{4}}\d\zeta} \d t < \infty.
			\end{align*}
			By an application of the Lebesgue monotone theorem, we conclude that as $t\to \infty$
			\begin{align*}
				\int_{- \infty}^{-t} \bigg\{  \|\z(\omega)(t)\|^{r+1}_{\widetilde{\L}^{r+1}} + \|\z(\omega)(t)\|^2_{\H} + \|\z(\omega)(t)\|^4_{\widetilde{\L}^4} + 1 \bigg\}e^{2\alpha \sigma +R \int_{\sigma}^{0}\|\z(\omega)(\zeta)\|^{4}_{\widetilde{\L}^{4}}\d\zeta} \d\sigma \to 0.
			\end{align*}
			This implies that $\kappa_3, \kappa_4, \kappa_5, \kappa_6\in \mathfrak{K}_1$, which completes the proof.
		\end{proof}
	\end{proposition}
	\begin{proposition}\label{radius_2}
		For $r\geq3$, define functions $\widetilde{\kappa}_{i}:\Omega_2\to (0, \infty), i= 1, 2, 3, 4, 5, 6,$ by the following formulae, for $\omega\in\Omega_2,$
		\begin{align*}
			[\widetilde{\kappa}_1(\omega)]^2 &:= \|\z(\omega)(0)\|_{\H},\ \ \
			[\widetilde{\kappa}_2(\omega)]^2 := \sup_{s\leq 0} \|\z(\omega)(s)\|^2_{\H}e^{2\alpha s}, \\
			[\widetilde{\kappa}_3(\omega)]^2 &:= \int_{- \infty}^{0} \|\z(\omega)(t)\|^{r+1}_{\widetilde{\L}^{r+1}}e^{2\alpha t} \d t, \ \ \ 
			[\widetilde{\kappa}_4(\omega)]^2 := \int_{- \infty}^{0} \|\z(\omega)(t)\|^2_{\H}e^{2\alpha t} \d t,\\
			[\widetilde{\kappa}_5(\omega)]^2 &:= \int_{- \infty}^{0} \|\z(\omega)(t)\|^4_{\wi\L^4}e^{2\alpha t} \d t,\ \ \ 
			[\widetilde{\kappa}_6(\omega)]^2 := \int_{- \infty}^{0} e^{2\alpha t} \d t.
		\end{align*}
		Then all these functions belongs to the class ${\mathfrak{K}}_2.$
	\end{proposition}
	\begin{proof}
		Proof is similar to the proof of Proposition \ref{radius}.
	\end{proof}
	
	\begin{theorem}\label{Main_theorem_1}
		Suppose that the Assumptions \ref{assumpO} (for domain $\mathcal{O}$), \ref{assump1} (for $r\in[1,3)$) and \ref{assump2} (for $r\geq3$)  are satisfied. Consider the MDS, $\Im = (\Omega_i, \hat{\mathcal{F}}_i, \hat{\mathbb{P}}_i, \hat{\theta})$ from Proposition \ref{m-DS1}, and the RDS $\varphi$ on $\H$ over $\Im$ generated by the stochastic convective Brinkman-Forchheimer equations \eqref{S-CBF} with additive noise satisfying the Assumptions \ref{assump1} (for $r\in[1,3)$) and \ref{assump2} (for $r\geq3$). Then, for $i\in\{1,2\}$, there exists a unique random $\mathfrak{DK}_{i}$-attractor for continuous RDS $\varphi$ in $\H$.
	\end{theorem}
	\begin{proof}
		Because of \cite[Theorem 2.8]{BCLLLR}, it is only needed to prove that there exists a $\mathfrak{DK}_i$-absorbing set $\textbf{B}_i\in \mathfrak{DK}_i$ and the RDS $\varphi$ is $\mathfrak{DK}_i$-asymptotically compact. 	
		\vskip 0.2 cm 
		\noindent
		\textbf{Existence of $\mathfrak{DK}_i$-absorbing set $\textbf{B}_i\in \mathfrak{DK}_i$:}	Let $\mathrm{D}_i$ be a random set from the class $\mathfrak{DK}_i,$ for $i=1,2$. Let $\kappa_{\mathrm{D}_1}(\omega)$ and $\widetilde{\kappa}_{\mathrm{D}_2}(\omega)$ be the radii of $\mathrm{D}_1(\omega)$ and $\mathrm{D}_2(\omega)$ respectively, that is, $\kappa_{\mathrm{D}_1}(\omega):= \sup\{\|x\|_{\H} : x \in \mathrm{D}_1(\omega)\}, \ \omega\in \Omega_1$ and $\widetilde{\kappa}_{\mathrm{D}_2}(\omega):= \sup\{\|x\|_{\H} : x \in \mathrm{D}_2(\omega)\}, \ \omega\in \Omega_2$.
		
		Let $\omega\in \Omega_i$ be fixed. For given $s\leq 0$ and $\boldsymbol{x}\in \H$, let $\v$ be the solution of \eqref{cscbf} on the time interval $[s, \infty)$ with the initial condition $\v(s)= \boldsymbol{x}-\z(s).$ For $r\in[1,3)$, using \eqref{Energy_esti1} for $t=0 \text{ and } \tau=s\leq0$, we obtain 
		\begin{align}\label{Energy_esti5}
			\|\v(0)\|^2_{\H} \leq &\ \ 2 \|\boldsymbol{x}\|^2_{\H}\  e^{2\alpha s +R \int_{s}^{0}\|\z(\zeta)\|^{4}_{\widetilde{\L}^{4}}\d\zeta}  + 2 \|\z(s)\|^2_{\H}\  e^{2\alpha s +R \int_{s}^{0}\|\z(\zeta)\|^{4}_{\widetilde{\L}^{4}}\d\zeta} \nonumber\\
			&+ C\int_{s}^{0} \bigg\{ \|\z(t)\|^{2}_{\H}+\|\z(t)\|^{4}_{\widetilde{\L}^{4}}+ \|\z(t)\|^{r+1}_{\widetilde{\L}^{r+1}}  + \|\f\|^2_{\V'}\bigg\}e^{2\alpha t +R \int_{t}^{0}\|\z(\zeta)\|^{4}_{\widetilde{\L}^{4}}\d\zeta} \d t,
		\end{align}
		and for $r\geq3$, using \eqref{Energy_esti1}, we get 
		\begin{align}\label{Energy_esti5_2}
			\|\v(0)\|^2_{\H} &\leq 2 \|\boldsymbol{x}\|^2_{\H}\ e^{2\alpha s} + 2 \|\z(s)\|^2_{\H}\ e^{2\alpha s} \nonumber\\&\quad+C \int_{s}^{0} \bigg\{ \|\z(t)\|^{2}_{\H}+\|\z(t)\|^{4}_{\widetilde{\L}^{4}}+ \|\z(t)\|^{r+1}_{\widetilde{\L}^{r+1}}  + \|\f\|^2_{\V'}\bigg\}e^{2\alpha t}\d t.
		\end{align}
		For $\omega\in \Omega_1,$ let us set
		\begin{align}
			[\kappa_{11}(\omega)]^2 &= 2 +  2\sup_{s\leq 0}\bigg\{ \|\z(s)\|^2_{\H}\  e^{2\alpha s +R \int_{s}^{0}\|\z(\zeta)\|^{4}_{\widetilde{\L}^{4}}\d\zeta}\bigg\} +C \int_{- \infty}^{0} \bigg\{\|\z(s)\|^2_{\H} + \|\z(s)\|^4_{\widetilde{\L}^4} \nonumber\\
			&\quad \qquad +\|\z(t)\|^{r+1}_{\widetilde{\L}^{r+1}} +  \|\f\|^2_{\V'}\bigg\}e^{2\alpha t +R \int_{t}^{0}\|\z(\zeta)\|^{4}_{\widetilde{\L}^{4}}\d\zeta} \d t,\\	
			\kappa_{12}(\omega)=& \ \  \|\z(\omega)(0)\|_{\H}.
		\end{align}
		Invoking Lemma \ref{Bddns5} and Proposition \ref{radius}, we get that both $\kappa_{11},\kappa_{12}\in  \mathfrak{K}_1$ and also that $\kappa_{11}+\kappa_{12}=:\kappa_{13} \in \mathfrak{K}_1$ as well. Therefore the random set $\textbf{B}_1$ defined by $$\textbf{B}_1(\omega) := \{\u\in\H: \|\u\|_{\H}\leq \kappa_{13}(\omega)\}$$ is such that $\textbf{B}_1\in\mathfrak{DK}_1.$ For $\omega\in \Omega_2,$ let us set
		\begin{align}
			[\widetilde{\kappa}_{11}(\omega)]^2 &= 2 +  2\sup_{s\leq 0}\bigg\{ \|\z(s)\|^2_{\H}\  e^{2\alpha s}\bigg\} +C \int_{- \infty}^{0} \bigg\{\|\z(s)\|^2_{\H} + \|\z(s)\|^4_{\widetilde{\L}^4} \nonumber\\&\qquad+\|\z(t)\|^{r+1}_{\widetilde{\L}^{r+1}} +  \|\f\|^2_{\V'}\bigg\}e^{2\alpha t} \d t.
		\end{align}
		Invoking Lemma \ref{Bddns5} and Proposition \ref{radius_2} we get that both $\widetilde{\kappa}_{11},\kappa_{12}\in \mathfrak{K}_2$ and also that $\widetilde{\kappa}_{11}+\kappa_{12}=:\widetilde{\kappa}_{13}\in \mathfrak{K}_2$ as well. Therefore the random set $\textbf{B}_2$ defined by $$\textbf{B}_2(\omega) := \{\u\in\H: \|\u\|_{\H}\leq \widetilde{\kappa}_{13}(\omega)\}$$ is such that $\textbf{B}_2\in\mathfrak{DK}_2.$
		
		Let us now prove that $\textbf{B}_i$ absorbs $\mathrm{D}_i$. Let $\omega\in\Omega_i$ be fixed. Since $\kappa_{\mathrm{D}_1}(\omega)\in \mathfrak{K}_1$ and $\widetilde{\kappa}_{\mathrm{D}_2}(\omega)\in \mathfrak{K}_2$, there exists $t_{\mathrm{D}_i}(\omega)\geq 0$ such that 
		\begin{align*}
			[\kappa_{\mathrm{D}_1}(\theta_{-t}\omega)]^2 e^{-2\alpha t +R\int_{-t}^{0}\|\z(\omega)(s)\|^{4}_{\widetilde{\L}^{4}}\d s} &\leq 1, \ \ \ \text{for}\  t\geq t_{\mathrm{D}_1}(\omega), 
		\end{align*}
		and
		\begin{align*}
			[\widetilde{\kappa}_{\mathrm{D}_2}(\theta_{-t}\omega)]^2 e^{-2\alpha t}&\leq 1, \ \ \ \text{for}\  t\geq t_{\mathrm{D}_2}(\omega). 
		\end{align*}
		Thus, for $\omega\in\Omega_i$, if $\boldsymbol{x}\in \mathrm{D}_i(\theta_{-t}\omega)$ and $s\leq- t_{\mathrm{D}_i}(\omega),$ then by \eqref{Energy_esti5}, we obtain  $$\|\v(0,\omega; s, \boldsymbol{x}-\z(s))\|_{\H}\leq \kappa_{11}(\omega), \ \ \ \text{ for } \omega\in \Omega_1 $$ and by \eqref{Energy_esti5_2}, we have $$\|\v(0,\omega; s, \boldsymbol{x}-\z(s))\|_{\H}\leq \widetilde{\kappa}_{11}(\omega), \ \ \ \text{ for } \omega\in \Omega_2.$$ Thus, we conclude that, for $\omega\in\Omega_1$  $$\|\u(0,\omega; s, \boldsymbol{x})\|_{\H} \leq \|\v(0,\omega; s, \boldsymbol{x}-\z(s))\|_{\H} + \|\z(\omega)(0)\|_{\H}\leq \kappa_{13}(\omega),$$ and for $\omega\in\Omega_2$ $$\|\u(0,\omega; s, \boldsymbol{x})\|_{\H} \leq \|\v(0,\omega; s, \boldsymbol{x}-\z(s))\|_{\H} + \|\z(\omega)(0)\|_{\H}\leq \widetilde{\kappa}_{13}(\omega).$$The above inequalities implies that for $\omega\in \Omega_i$, $\u(0,\omega; s, \boldsymbol{x}) \in \textbf{B}_i(\omega)$, for all $s\leq -t_{\mathrm{D}_i}(\omega).$ This proves  $\textbf{B}_i$ absorbs $\mathrm{D}_i$.	
		\vskip 0.2 cm 
		\noindent
		\textbf{The RDS $\varphi$ is $\mathfrak{DK}_i$-asymptotically compact.} 		
		Let us assume that $\mathrm{D}_i \in \mathfrak{DK}_i$ and $\textbf{B}_i\in \mathfrak{DK}_i$ be such that $\textbf{B}_i$ absorbs $\mathrm{D}_i$. Let us fix $\omega\in \Omega_i$ and take a sequence of positive numbers $\{t_m\}^{\infty}_{m=1}$ such that $t_1\leq t_2 \leq t_3 \leq \cdots$ and $t_m \to \infty$. We take an $\H$-valued sequence $\{\boldsymbol{x}_m\}^{\infty}_{m=1}$ such that $\boldsymbol{x}_m \in \mathrm{D}_i(\theta_{-t}\omega),$ for all $m\in \mathbb{N}.$
		\vskip 0.2 cm 
		\noindent 
		\textbf{Step I.} \textit{Reduction.} Since $\textbf{B}_i$ absorbs $\mathrm{D}_i$, we obtain  \begin{align}\label{619}\varphi(t_m, \theta_{-t_m}\omega, \mathrm{D}_i(\theta_{-t_m}\omega))\subset \textbf{B}_i(\omega),\end{align} 
		for sufficient large $m\in \mathbb{N}.$ Since $\textbf{B}_i(\omega) \subset \H$ is a bounded set, which implies that $\textbf{B}_i(\omega)$ is weakly pre-compact in $\H$, without loss of generality, we may assume that \eqref{619} holds for all $m\in \mathbb{N}$ and, for some $\y_0\in\H,$ 
		\begin{align}\label{weak_lim1}
			\varphi(t_m, \theta_{-t_m}\omega, \boldsymbol{x}_m)\xrightharpoonup{w} \y_0 \ \text{ in } \ \H.
		\end{align}
		Since $\z(0)\in \H,$ we also have
		\begin{align}\label{weak_lim2}
			\varphi(t_m, \theta_{-t_m}\omega, \boldsymbol{x}_m)-\z(0) \xrightharpoonup{w}  \y_0-\z(0)\ \text{ in } \ \H.
		\end{align}
		Then by the weak lower semicontinuity of the $\H$-norm, we get 
		\begin{align}\label{weak_lim3}
			\|\y_0-\z(0)\|_{\H} \leq \liminf_{m\to \infty} \|\varphi(t_m, \theta_{-t_m}\omega, \boldsymbol{x}_m)-\z(0)\|_{\H}.
		\end{align}
		Now it is only need to show that for some subsequence $\{m'\}\subset \mathbb{N}$
		\begin{align}\label{weak_lim4}
			\|\y_0-\z(0)\|_{\H} \geq \limsup_{m'\to \infty} \|\varphi(t_{m'}, \theta_{-t_{m'}}\omega, \boldsymbol{x}_{m'})-\z(0)\|_{\H}.
		\end{align}
		In fact, since $\H$ is a Hilbert space, \eqref{weak_lim3} combined with \eqref{weak_lim4} imply that $$\varphi(t_m, \theta_{-t_m}\omega, \boldsymbol{x}_m) - \z(0) \to \y_0-\z(0)$$ in $\H$,
		which implies that $\varphi(t_m, \theta_{-t_m}\omega, \boldsymbol{x}_m) \to \y_0$ in $\H$.
		\vskip 0.2 cm
		\noindent
		\textbf{Step II.} \textit{Construction of a negative trajectory, that is, a sequence $\{\y_m\}^0_{m=-\infty}$ such that $\y_m\in \boldsymbol{B}_i(\theta_m\omega), m\in \mathbb{Z}^{-},$ and 
			$$\y_j = \varphi(j-m, \theta_m\omega, \y_m), \ \ \ m<j\leq 0.$$}
		Since $\textbf{B}_i$ absorbs $\mathrm{D}_i$, there exists a constant $N_i(\omega)\in \mathbb{N},$ such that$$ \{\varphi(-1 + t_m, \theta_{1- t_m}\theta_{-1}, \boldsymbol{x}_m) : m\geq N_i(\omega)\} \subset \textbf{B}_i(\theta_{-1}\omega).$$ Thus, there exists a subsequence $\{m'\}\subset\mathbb{N}$ and $\y_{-1} \in \textbf{B}_i(\theta_{-1}\omega)$ such that
		\begin{align}\label{weak_lim5}
			\varphi(-1+t_{m'}, \theta_{-t_{m'}}\omega, \boldsymbol{x}_{m'}) \text{ converges to } \y_1\ \text{ in }\  \H \text{ weakly}.
		\end{align}
		Now, using the cocycle property of $\varphi$, with $t=1, \ s=-1 +  t_{m'} ,$ and $\omega$ being replaced by $\theta_{-t_{m'}}\omega,$ we have:
		\begin{align*}
			\varphi(t_{m'}, \theta_{-t_{m'}}\omega) = \varphi(1, \theta_{-1}\omega)\varphi(-1 + t_{m'}, \theta_{-t_{m'}}\omega).
		\end{align*}
		Using Lemma \ref{weak_topo2}, from \eqref{weak_lim1} and \eqref{weak_lim5}, we derive that $\varphi(1, \theta_{-1}\omega, \y_{-1}) = \y_0.$ Making use of mathematical induction, for each $j= 1, 2, \ldots,$ the construction of a subsequence $\{m^{(j)}\}\subset \{m^{(j-1)}\}$ and $\y_{-j}\in \textbf{B}_i(\theta_{-j}\omega)$ is possible such that $\varphi(1, \theta_{-j}\omega, \y_{-j})= \y_{-j+1}$ and 
		\begin{align}\label{weak_lim6}
			\varphi(-j + t_{m^{(j)}}, \theta_{-t_{m^{(j)}}}\omega, \boldsymbol{x}_{m^{(j)}})\ \text{ converges to }\ \y_{-j} \  \text{ in }\ \H\ \text{ weakly}  \text{ as }  m^{(j)} \to \infty.
		\end{align}
		As discussed above, using the cocycle property of $\varphi$, with $t=j,\ s=-j +t_{m^{(j)}}$ and $\omega$ being replaced by $\theta_{-t_{m^{(j)}}}\omega,$ gives
		\begin{align}\label{weak_lim7}
			\varphi(t_{m^{(j)}}, \theta_{-t_{m^{(j)}}}\omega) = \varphi(j, \theta_{-j}\omega)\varphi(t_{m^{(j)}}-j, \theta_{-t_{m^{(j)}}}\omega), \ \ \ j\in \mathbb{N}.
		\end{align}
		Hence, by using Lemma \ref{weak_topo2} with \eqref{weak_lim6}, we obtain 
		\begin{align}\label{weak_lim8}
			\y_{-k} = & \ \textrm{w-} \lim_{m^{(j)} \to \infty} \varphi (-k+t_{m^{(j)}}, \theta_{-t_{m^{(j)}}}\omega, \boldsymbol{x}_{m^{(j)}})\nonumber\\
			=&\ \textrm{w-} \lim_{m^{(j)} \to \infty}\varphi\big(-k+j, \theta_{-j}\omega, \varphi (t_{m^{(j)}}-j, \theta_{-t_{m^{(j)}}}\omega, \boldsymbol{x}_{m^{(j)}})\big)\nonumber\\
			=&\  \varphi\bigg(-k+j, \theta_{-j}\omega,\big(\textrm{w-} \lim_{m^{(j)} \to \infty}\varphi (t_{m^{(j)}}-j, \theta_{-t_{m^{(j)}}}\omega, \boldsymbol{x}_{m^{(j)}})\big)\bigg)\nonumber\\
			= & \ \varphi (-k+j, \theta_{-j}\omega, \y_{-j}),
		\end{align}
		where $\textrm{w-}\lim$ represents the weak limit in $\H.$ Similarly, one can obtain $$\varphi(k, \theta_{-j}\omega, \y_{-j})= \y_{-j+k}, \ \text{ if }\ 0\leq k\leq j.$$	
		More precisely, in \eqref{weak_lim8}, $\y_{-k} = \u(-k, -j; \omega,\y_{-j}),$ where $\u$ is given by \eqref{combine_sol}.
		\vskip 0.2 cm
		\noindent
		\textbf{Step III.} \textit{Proof of \eqref{weak_lim4}.} For further proof, we fix $j\in\mathbb{N}$ (until explicitly stated), and consider the system \eqref{S-CBF} on $[-j, 0].$ From \eqref{combine_sol} and \eqref{weak_lim7}, for $t=0$ and $s=-k,$ we get
		\begin{align}\label{weak_lim9}
			&\|\varphi(t_{m^{(j)}}, \theta_{-t_{m^{(j)}}}\omega, \boldsymbol{x}_{m^{(j)}}) - \z(0)\|^2_{\H}\nonumber\\ &=  \|\varphi\big(j, \theta_{-j}\omega,\varphi(t_{m^{(j)}}-j, \theta_{-t_{m^{(j)}}}\omega, \boldsymbol{x}_{m^{(j)}})\big) - \z(0)\|^2_{\H}\nonumber\\
			&=\|\v\big(0, -j; \omega, \varphi(t_{m^{(j)}}-j, \theta_{-t_{m^{(j)}}}\omega, \boldsymbol{x}_{m^{(j)}}) - \z(-j)\big)\|^2_{\H}.
		\end{align}
		Let $\v$ be the solution to the system \eqref{cscbf} on $[-j, \infty)$ with the initial data at time $-j:$ $$\v(-j) = \varphi(t_{m^{(j)}}-j, \theta_{-t_{m^{(j)}}}\omega, \boldsymbol{x}_{m^{(j)}}) - \z(-j).$$ Also, we can write $$\v(s) = \v\big(s, -j; \omega, \varphi(t_{m^{(j)}}-j, \theta_{-t_{n^{(j)}}}\omega, \boldsymbol{x}_{m^{(j)}}) - \z(-j)\big), \ \ \ s\geq -j.$$
		Using \eqref{Energy_esti2} for $t=0$ and $\tau = -j$, we obtain 
		\begin{align}\label{Energy_esti6}
			&\|\varphi(t_{m^{(j)}}, \theta_{-t_{m^{(j)}}}\omega, \boldsymbol{x}_{m^{(j)}}) - \z(0)\|^2_{\H}\nonumber\\ 
			&= e^{- 2\alpha j} \|\varphi(t_{m^{(j)}}-j, \theta_{-t_{m^{(j)}}}\omega, \boldsymbol{x}_{m^{(j)}}) - \z(-j)\|^2_{\H}  \nonumber\\
			&\quad+ 2 \int_{-j}^{0} e^{2\alpha s}\bigg(b(\v(s), \v(s), \z(s))-b(\z(s), \z(s), \v(s)) - \beta\big\langle\mathcal{C}(\v(s) +\z(s)),\z(s)\big\rangle  \nonumber\\
			& \quad+(\chi-\alpha)(\z(s),\v(s)) +\big\langle \f, \v(s)\big \rangle- \mu\|\v(s)\|^2_{\V}-\beta\|\v(s)+\z(s)\|^{r+1}_{\wi\L^{r+1}}\bigg) \d s.
		\end{align}	
		In order to  complete the proof, we only need to prove the existence of a function (non-negative) $\h\in \mathrm{L}^1 (-\infty, 0)$ such that  
		\begin{align}\label{weak_lim10}
			\limsup_{m^{(j)}\to \infty} \|\varphi(t_{m^{(j)}}, \theta_{-t_{m^{(j)}}}\omega, \boldsymbol{x}_{m^{(j)}}) - \z(0)\|^2_{\H}\leq \int_{-\infty}^{-j} \h(s)\ \d s + \|\y_0 -\z(0)\|^2_{\H}.
		\end{align}
		For this, if we define the diagonal process $\{n_k\}^{\infty}_{k=1}$ by $n_k=k^{(k)}, k\in \mathbb{N},$ the sequence $\{n_k\}^{\infty}_{k=j}$ is a subsequence of the sequence $(m^{(j)})$ and hence by \eqref{weak_lim10}, 
		\begin{align}\label{weak_lim11}
			\limsup_{k} \|\varphi(t_{n_k}, \theta_{-t_{n_k}}\omega, \boldsymbol{x}_{n_k}) - \z(0)\|^2_{\H}\leq \int_{-\infty}^{-j} \h(s)\ \d s + \|\y_0 -\z(0)\|^2_{\H}.
		\end{align}
		Taking the limit $j\to \infty$ in \eqref{weak_lim11}, we arrive at
		\begin{align*}
			\limsup_{j} \|\varphi(t_{n_k}, \theta_{-t_{n_k}}\omega, \boldsymbol{x}_{n_k}) - \z(0)\|^2_{\H}\leq \|\y_0 -\z(0)\|^2_{\H},
		\end{align*}
		which proves \eqref{weak_lim4}.
		\vskip 0.2 cm
		\noindent
		\textbf{Step IV.}
		\textit{Proof of \eqref{weak_lim10}.} Firstly, we estimate the first term on the right hand side of \eqref{Energy_esti6}. 
		\vskip 0.2 cm
		\noindent
		\textbf{Case I:} Let us take  $r\in[1,3)$ and $\omega\in \Omega_1$.
		If $-t_{m^{(j)}}<-j,$ then by \eqref{combine_sol} and \eqref{Energy_esti1}, we obtain 
		\begin{align}\label{Energy_esti7}
			&\|\varphi(t_{m^{(j)}}-j, \theta_{-t_{m^{(j)}}}\omega, \boldsymbol{x}_{m^{(j)}}) - \z(-j)\|^2_{\H}\ e^{-2\alpha j}\nonumber\\ &=\|\v\big(-j, -t_{m^{(j)}}; \omega, \boldsymbol{x}_{m^{(j)}}-\z(-t_{m^{(j)}})\big)\|^2_{\H} \ e^{-2\alpha j}\nonumber\\
			&\leq e^{-2\alpha j}  \bigg[\|\boldsymbol{x}_{m^{(j)}}-\z(-t_{m^{(j)}})\|^2_{\H}\  e^{-2\alpha(t_{m^{(j)}}-j) +R \int_{-t_{m^{(j)}}}^{-j}\|\z(s)\|^{4}_{\widetilde{\L}^{4}}\d s}  \nonumber\\
			&\quad+ C\int_{-t_{m^{(j)}}}^{-j} \bigg\{ \|\z(s)\|^{2}_{\H}+\|\z(s)\|^{4}_{\widetilde{\L}^{4}}+ \|\z(s)\|^{r+1}_{\widetilde{\L}^{r+1}} + \|\f\|^2_{\V'}\bigg\}e^{-2\alpha(-j-s) +R \int_{s}^{-j}\|\z(\zeta)\|^{4}_{\widetilde{\L}^{4}}\d\zeta} \d s\bigg]\nonumber\\
			&\leq  2 K_{m^{(j)}}^1 + 2 K_{m^{(j)}}^2 + CK_{m^{(j)}}^3 + CK_{m^{(j)}}^4+CK_{m^{(j)}}^5 +  \|\f\|^2_{\V'}K_{m^{(j)}}^6,
		\end{align}
		where 
		\begin{align*}
			K_{m^{(j)}}^1 &=\|\boldsymbol{x}_{m^{(j)}}\|^2_{\H}\  e^{-2\alpha t_{m^{(j)}} +R \int\limits_{-t_{m^{(j)}}}^{-j}\|\z(t)\|^{4}_{\widetilde{\L}^{4}}\d s},
			K_{m^{(j)}}^2 =\|\z(-t_{m^{(j)}})\|^2_{\H}\  e^{-2\alpha t_{m^{(j)}} +R \int\limits_{-t_{m^{(j)}}}^{-j}\|\z(t)\|^{4}_{\widetilde{\L}^{4}}\d s},\\
			K_{m^{(j)}}^3 &= \int_{-\infty}^{-j} \|\z(s)\|^{2}_{\H} \ e^{2\alpha s +R \int_{s}^{-j}\|\z(\zeta)\|^{4}_{\widetilde{\L}^{4}}\d\zeta} \d s,\ \ 
			K_{m^{(j)}}^4 = \int_{-\infty}^{-j} \|\z(s)\|^{4}_{\widetilde{\L}^{4}} \ e^{2\alpha s +R \int_{s}^{-j}\|\z(\zeta)\|^{4}_{\widetilde{\L}^{4}}\d\zeta} \d s,\\
			K_{m^{(j)}}^5 &= \int_{-\infty}^{-j} \|\z(s)\|^{r+1}_{\widetilde{\L}^{r+1}} \ e^{2\alpha s +R \int_{s}^{-j}\|\z(\zeta)\|^{4}_{\widetilde{\L}^{4}}\d\zeta} \d s,\ \ 
			K_{n^{(k)}}^6 = \int_{-\infty}^{-j} e^{2\alpha s +R \int_{s}^{-j}\|\z(\zeta)\|^{4}_{\widetilde{\L}^{4}}\d\zeta} \d s.
		\end{align*}
		\vskip 0.2 cm
		\noindent
		\textbf{Case II:} For $r\geq3$ and $\omega\in \Omega_2$. If $-t_{m^{(j)}}<-j,$ then by \eqref{combine_sol} and \eqref{Energy_esti1}, we find
		\begin{align}\label{Energy_esti7_2}
			&\|\varphi(t_{m^{(j)}}-j, \theta_{-t_{m^{(j)}}}\omega, \boldsymbol{x}_{m^{(j)}}) - \z(-j)\|^2_{\H}\ e^{-2\alpha j}\nonumber\\ &=\|\v\big(-j, -t_{m^{(j)}}; \omega, \boldsymbol{x}_{m^{(j)}}-\z(-t_{m^{(j)}})\big)\|^2_{\H} \ e^{-2\alpha j}\nonumber\\
			&\leq e^{-2\alpha j}  \bigg[\|\boldsymbol{x}_{m^{(j)}}-\z(-t_{m^{(j)}})\|^2_{\H}\  e^{-2\alpha(t_{m^{(j)}}-j)}  \nonumber\\
			&\quad+ C\int_{-t_{m^{(j)}}}^{-j} \bigg\{ \|\z(s)\|^{2}_{\H}+\|\z(s)\|^{4}_{\widetilde{\L}^{4}}+ \|\z(s)\|^{r+1}_{\widetilde{\L}^{r+1}} + \|\f\|^2_{\V'}\bigg\}e^{-2\alpha(-j-s)} \d s\bigg]\nonumber\\
			&\leq  2 \widetilde{K}_{m^{(j)}}^1 + 2 \widetilde{K}_{m^{(j)}}^2 + C\widetilde{K}_{m^{(j)}}^3 + C\widetilde{K}_{m^{(j)}}^4+C\widetilde{K}_{m^{(j)}}^5 +  \|\f\|^2_{\V'}\widetilde{K}_{m^{(j)}}^6,,
		\end{align}
		where 
		\begin{align*}
			\widetilde{K}_{m^{(j)}}^1 =\|\boldsymbol{x}_{m^{(j)}}\|^2_{\H}\  e^{-2\alpha t_{m^{(j)}} },\ \ \ 
			\widetilde{K}_{m^{(j)}}^2 =\|\z(-t_{m^{(j)}})\|^2_{\H}\  e^{-2\alpha t_{m^{(j)}} },\\
			\widetilde{K}_{m^{(j)}}^3 = \int_{-\infty}^{-j} \|\z(s)\|^{2}_{\H} \ e^{2\alpha s} \d s,\ \  \
			\widetilde{K}_{m^{(j)}}^4 = \int_{-\infty}^{-j} \|\z(s)\|^{4}_{\widetilde{\L}^{4}} \ e^{2\alpha s} \d s,\\
			\widetilde{K}_{m^{(j)}}^5 = \int_{-\infty}^{-j} \|\z(s)\|^{r+1}_{\widetilde{\L}^{r+1}} \ e^{2\alpha s } \d s,\ \ \ \ \ \ \ \ \ \ \ \ \
			\widetilde{K}_{n^{(k)}}^6 = \int_{-\infty}^{-j} e^{2\alpha s} \d s.
		\end{align*}
		Let us prove the existence of a function (non-negative) $\h\in \mathrm{L}^1(-\infty, 0)$ such that 
		\begin{align}\label{weak_lim12}
			\limsup_{m^{(j)}\to \infty} \|\varphi&(t_{m^{(j)}}-j, \theta_{-t_{m^{(j)}}}\omega, \boldsymbol{x}_{m^{(j)}}) - \z(-j)\|^2_{\H}\ e^{-2\alpha k} \leq \int_{- \infty}^{-j} \h(s)\ \d s, \ \ \ j\in \mathbb{N}.
		\end{align}
		\vskip 0.2 cm
		\noindent
		\textbf{Step V.} \emph{We claim that, for $r\in[1,3)$ and $\omega\in\Omega_1$
			\begin{align}\label{Bddns6}
				\limsup_{m^{(j)}\to \infty} K_{m^{(j)}}^1 = 0,
			\end{align}
			and for $r\geq3$ and $\omega\in\Omega_2$
			\begin{align}\label{Bddns6_2}
				\limsup_{m^{(j)}\to \infty} \widetilde{K}_{m^{(j)}}^1 = 0.
		\end{align}}
		Making use of Corollary \ref{Bddns1_1}, we have for sufficiently large $m^{(j)},$
		\begin{align*}
			R \int_{-t_{m^{(j)}}}^{-j} \|\z(s)\|^{4}_{\widetilde{\L}^4} \d s \leq  \alpha(t_{m^{(j)}}-j).
		\end{align*}
		Since $\mathrm{D}_i(\omega)\subset \H$, which is bounded, we can find $\rho_3>0,$ such that $\|\boldsymbol{x}_{m^{(j)}}\|_{\H}\leq \rho_3,$ for every $m^{(j)}.$ Hence, we obtain 
		$$\limsup_{m^{(j)}\to \infty}\|\boldsymbol{x}_{m^{(j)}}\|^2_{\H}\  e^{-2\alpha t_{m^{(j)}} +R \int_{-t_{m^{(j)}}}^{-j}\|\z(s)\|^{4}_{\widetilde{\L}^{4}}\d s}\leq \limsup_{m^{(j)}\to \infty} \rho_3^2 e^{-\alpha(t_{m^{(j)}}+j)} = 0,$$ and 
		$$\limsup_{m^{(j)}\to \infty}\|\boldsymbol{x}_{m^{(j)}}\|^2_{\H}\  e^{-2\alpha t_{m^{(j)}}}\leq \limsup_{m^{(j)}\to \infty} \rho_3^2 e^{-2\alpha t_{m^{(j)}}} = 0.$$
		Therefore, in view of \eqref{Energy_esti7} and \eqref{Bddns6} for $r\in[1,3)$, and \eqref{Energy_esti7_2} and \eqref{Bddns6_2} for $r\geq3$ with Lemmas \ref{Bddns4} and \ref{Bddns5}, the proof of \eqref{weak_lim12} is completed, and we are only left to prove the inequality \eqref{weak_lim10}.
		\vskip 0.2 cm
		\noindent
		\textbf{Step VI.} Let us denote 
		\begin{align*}
			\v^{m^{(j)}}(s) &= \v\big(s, -j; \omega, \varphi(t_{m^{(j)}}-j, \theta_{-t_{m^{(j)}}}\omega)\boldsymbol{x}_{m^{(j)}} - \z(-j)\big), \ s\in (-j, 0),\\
			\v_j(s) &= \v\big(s, -j; \omega, \y_{-j} - \z(-j)\big), \ s\in (-j, 0).
		\end{align*}
		By Lemma \ref{weak_topo1} and convergence property \eqref{weak_lim6}, we conclude that 
		\begin{align}\label{weak_lim13}
			\begin{cases}
				\v^{m^{(j)}}(\cdot) \text{ converges to } \v_j(\cdot)  \text{ in } \mathrm{L}^2(-j,0;\V) \text{ weakly},\\
				\v^{m^{(j)}}(\cdot) \text{ converges to } \v_j(\cdot)  \text{ in } \mathrm{L}^{r+1}(-j,0;\widetilde{\L}^{r+1}) \text{ weakly}.
			\end{cases}
		\end{align}
		Since $e^{2\alpha \cdot }\B(\z(\cdot)) \in \mathrm{L}^2(-j,0;\V')$ (see \eqref{2.8} and \eqref{2.13}), $ e^{2\alpha \cdot} \f \in \mathrm{L}^2(-j,0;\V')$ and $e^{2\alpha \cdot }\z(\cdot)\in \mathrm{L}^2(-j,0;\H)$, we obtain 
		\begin{align}\label{weak_lim14}
			\lim_{m^{(j)} \to \infty} \int_{-j}^{0} e^{2\alpha s} b(\z(s),\z(s),\v^{m^{(j)}}(s))  \d s = \int_{-j}^{0} e^{2\alpha s} b(\z(s),\z(s), \v_j(s)) \d s,
		\end{align}
		\begin{align}\label{weak_lim15}
			\lim_{m^{(j)} \to \infty} \int_{-j}^{0} e^{2\alpha s} \big\langle \f, \v^{m^{(j)}}(s)\big\rangle \ \d s = \int_{-j}^{0} e^{2\alpha s} \big\langle \f, \v_j(s)\big\rangle \ \d s,
		\end{align}
		and
		\begin{align}\label{weak_lim15'}
			\lim_{m^{(j)} \to \infty} \int_{-j}^{0} e^{2\alpha s} ( \z(s), \v^{m^{(j)}}(s)) \ \d s = \int_{-j}^{0} e^{2\alpha s} (\z(s), \v_j(s)) \ \d s.
		\end{align}
		Since we have the convergence property \eqref{weak_lim13}, we can find a subsequence of $\{\v^{m^{(j)}}\}$  (denoted as the same) such that
		\begin{align}\label{weak_lim16}
			\v^{m^{(j)}}(\cdot) \ \text{ converges to } \ \v_j(\cdot)  \  \text{ in } \ \mathrm{L}^2(-j,0;{\L}^2_{\text{loc}}(\mathcal{O}))\  \text{ strongly}.
		\end{align} 
		Next, since $e^{2\alpha t}\z(t), t\in \R,$ is an $\H\cap\widetilde{\L}^4$-valued process (respectively, $\H\cap\widetilde{\L}^{r+1}$-valued process) for $r\in[1,3)$ (respectively, for $r\geq3$), in view of Corollary \ref{convergence_b1},  for $r\in[1,3)$  (respectively, Corollary \eqref{convergence_b2}, for $r\geq3$), along with \eqref{weak_lim13} and \eqref{weak_lim16}, we infer 
		\begin{align}\label{weak_lim17}
			\lim_{m^{(j)} \to \infty}\int_{-j}^{0} e^{2\alpha s} b\big(\v^{m^{(j)}}(s), \z(s), \v^{m^{(j)}}(s) \big)\ \d s = \int_{-j}^{0} e^{2\alpha s} b\big(\v_j(s), \z(s), \v_j(s) \big)\ \d s.
		\end{align}
		Once again using the fact that $e^{2\alpha t}\z(t), t\in \R,$ is an $\H\cap\widetilde{\L}^4$-valued and $\H\cap\widetilde{\L}^{r+1}$-valued process for $r\in[1,3)$ and $r\geq3$, respectively and invoking Corollary \ref{convergence_c3_1}, \eqref{weak_lim13} and \eqref{weak_lim16}, we arrive at
		\begin{align}\label{weak_lim18}
			\lim_{m^{(j)} \to \infty}\int_{-j}^{0} e^{2\alpha s}\big\langle\mathcal{C}(\v^{m^{(j)}}(s) +\z(s)),\z(s)\big\rangle \ \d s = \int_{-j}^{0} e^{2\alpha s} \big\langle\mathcal{C}(\v_j(s) +\z(s)),\z(s)\big\rangle\ \d s.
		\end{align}
		Now, since for any $s\in [-j, 0], e^{-2\alpha j}\leq e^{2\alpha s}\leq 1,\  (\int_{-j}^{0} e^{2\alpha s} \|\cdot\|^2_{\V}\ \d s)^{1/2} $ defines a norm in $\mathrm{L}^2(-k,0; \V)$, which is equivalent to the standard norm. Hence, from \eqref{weak_lim13}, we get
		\begin{align*}
			\int_{-j}^{0} e^{2\alpha s} \|\v_j(s)\|^2_{\V} \d s \leq \liminf_{m^{(j)}\to \infty} \int_{-j}^{0} e^{2\alpha s} \|\v^{m^{(j)}}(s)\|^2_{\V} \d s.
		\end{align*}
		We can also write the above inequality as 
		\begin{align}\label{weak_lim19}
			\limsup_{m^{(j)}\to \infty} \bigg\{- \int_{-j}^{0} e^{2\alpha s} \|\v^{m^{(j)}}(s)\|^2_{\V} \d s\bigg\}  \leq - \int_{-j}^{0} e^{2\alpha s} \|\v_j(s)\|^2_{\V} \d s.
		\end{align}
		Similarly, since for any $s\in [-j, 0], e^{-2\alpha j}\leq e^{2\alpha s}\leq 1,$   and $\bigg(\int_{-j}^{0} e^{2\alpha s} \|\cdot\|^{r+1}_{\widetilde{\L}^{r+1}} \d s\bigg)^{\frac{1}{r+1}} $ defines a norm in $\mathrm{L}^{r+1}(-k,0;\widetilde{\L}^{r+1} )$, which is equivalent to the standard norm. Thus, from \eqref{weak_lim13}, we obtain 
		\begin{align}\label{weak_lim21}
			\limsup_{m^{(j)}\to \infty} \bigg\{- \int_{-j}^{0} e^{2\alpha s} \|\v^{m^{(j)}}(s) + \z(s)\|^{r+1}_{\widetilde{\L}^{r+1}}\ \d s \bigg\} \leq - \int_{-j}^{0} e^{2\alpha s} \|\v_j(s) + \z(s)\|^{r+1}_{\widetilde{\L}^{r+1}}\ \d s.
		\end{align}
		From \eqref{Energy_esti6}, \eqref{weak_lim12}, \eqref{weak_lim14}-\eqref{weak_lim15'}, \eqref{weak_lim17} and  \eqref{weak_lim18}, and inequalities \eqref{weak_lim19} and \eqref{weak_lim21}, we conclude
		\begin{align}\label{Energy_esti8}
			&	\limsup_{m^{(j)}\to \infty} \|\varphi(t_{m^{(j)}}, \theta_{-t_{m^{(j)}}}\omega, \boldsymbol{x}_{m^{(j)}}) - \z(0)\|^2_{\H}\nonumber\\ 
			&\leq \int_{- \infty}^{-j} \h(s)\d s  + 2 \int_{-j}^{0} e^{2\alpha s}\bigg(b(\v_j(s), \v_j(s), \z(s))- b(\z(s),\z(s), \v_j(s)) \nonumber\\
			&\quad+\beta\big\langle\mathcal{C}(\v_j(s) +\z(s)),\z(s)\big\rangle  +(\chi-\alpha)(\z(s),\v_j(s)) +\big\langle \f, \v_j(s)\big \rangle \nonumber\\&\quad- \|\v_j(s)\|^2_{\V}- \beta\|\v_j(s) + \z(s)\|^{r+1}_{\widetilde{\L}^{r+1}}\bigg) \d s.
		\end{align} 
		Now, by \eqref{weak_lim8} and \eqref{Energy_esti2}, we obtain 
		\begin{align}\label{Energy_esti9}
			&	\|\y_0-\z(0)\|^2_{\H}\nonumber \\&=  \|\varphi(j, \theta_{-j}\omega,\y_{-j}) - \z(0)\|^2_{\H}= \|\v\big(0, -j; \omega, \y_{-j} - \z(-j)\big)\|^2_{\H}\nonumber\\
			&	= \|\y_{-j}-\z(-j)\|^2_{\H}\ e^{-2\alpha j} + 2 \int_{-j}^{0} e^{2\alpha s}\bigg(b(\v_j(s), \v_j(s), \z(s))- b(\z(s),\z(s), \v_j(s)) \nonumber\\
			&\quad+\beta\big\langle\mathcal{C}(\v_j(s) +\z(s)),\z(s)\big\rangle  +(\chi-\alpha)(\z(s),\v_j(s)) +\big\langle \f, \v_j(s)\big \rangle - \|\v_j(s)\|^2_{\V}\nonumber\\&\quad- \beta\|\v_j(s) + \z(s)\|^{r+1}_{\widetilde{\L}^{r+1}}\bigg) \d s.
		\end{align}
		After combining \eqref{Energy_esti8} with \eqref{Energy_esti9}, we find 
		\begin{align*}
			\limsup_{m^{(j)}\to \infty} \|\varphi(t_{m^{(j)}}, \theta_{-t_{m^{(j)}}}\omega, \boldsymbol{x}_{m^{(j)}}) - \z(0)\|^2_{\H}
			\leq \int_{- \infty}^{-j} \h(s)\ \d s + \|\y_0-\z(0)\|^2_{\H}, 
		\end{align*}
		which shows \eqref{weak_lim10} and hence we conclude the proof of Theorem \ref{Main_theorem_1}. 
	\end{proof}

	\section{Invariant Measures}\label{sec7}\setcounter{equation}{0}
This section is devoted to show the existence of invariant measures for SCBF equations in $\H$. It is demonstrated in \cite{CF} that the existence of compact invariant random set is a sufficient condition for the existence of invariant measures, that is, if a random dynamical system $\varphi$ has compact invariant random set, then there exist invariant measures for $\varphi$ (\cite[Corollary 4.4]{CF}). Since, the random attractor itself is a compact invariant random set, the existence of invariant measures for the 2D SCBF equations \eqref{S-CBF} is a direct consequence of \cite[Corollary 4.4]{CF} and Theorem \ref{Main_theorem_1}. The existence of random attractors for 2D stochastic NSE in unbounded Poincar\'e domains has been established in \cite{BMO,BL}, etc.  Recently, the existence and uniqueness  of invariant measures for 2D stochastic NSE perturbed by a linear multiplicative Gaussian noise defined on the whole space has  been obtained in \cite{KM8}.  The existence of a unique invariant measure for 2D SCBF equations \eqref{SCBF} (for $r\in[1,3]$) defined on Poincar\'e domains (bounded or unbounded) in $\H$ is established in \cite{KM2}. Therefore, in this section, we prove the existence of unique invariant measures for SCBF equations \eqref{SCBF} for $d=2,3$ with $r\geq3$ ($r=3$ with $2\beta\mu\geq1$).

\subsection{Existence of invariant measures}
Let us define the transition operator $\{\mathrm{P}_t\}_{t\geq 0}$ by 
\begin{align}\label{71}
	\mathrm{P}_t f(\x)=\int_{\Omega}f(\varphi(\omega,t,\x))\d\mathbb{P}(\omega)=\E\left[f(\varphi(t,\x))\right], 
\end{align}
  for all $f\in\mathcal{B}_b(\H)$, where $\mathcal{B}_b(\H)$ is the space of all bounded and Borel measurable functions on $\H$ and $\varphi$ is the random dynamical system corresponding to the SCBF equations \eqref{S-CBF}, which is defined by \eqref{combine_sol}. The continuity of $\varphi$ (cf. Lemma \ref{RDS_Conti1}), \cite[Proposition 3.8]{BL} provides the following result: 
\begin{lemma}\label{Feller}
	The family $\{\mathrm{P}_t\}_{t\geq 0}$ is Feller, that is, $\mathrm{P}_tf\in\C_{b}(\H)$ if $f\in\C_b(\H)$, where $\C_b(\H)$ is the space of all bounded and continuous functions on $\H$. Furthermore, for any $f\in\C_b(\H)$, $\mathrm{P}_tf(\x)\to f(\x)$ as $t\downarrow 0$. 
\end{lemma}
Analogously as in the proof of \cite[Theorem 5.6]{CF}, one can prove that $\varphi$ is a Markov random dynamical system, that is, $\mathrm{P}_{t_1+t_2}=\mathrm{P}_{t_1}\mathrm{P}_{t_2}$, for all $t_1,t_2\geq 0$. Since, we know by \cite[Corollary 4.4]{CF} that if a Markov RDS on a Polish space has an invariant compact random set, then there exists a Feller invariant probability measure $\nu$ for $\varphi$. 
\begin{definition}
	A Borel probability measure $\nu$ on $\H$  is called an \emph{invariant measure} for a Markov semigroup $\{\mathrm{P}_t\}_{t\geq 0}$ of Feller operators on $\C_b(\H)$ if and only if $$\mathrm{P}_{t}^*\nu=\nu, \ t\geq 0,$$ where $(\mathrm{P}_{t}^*\nu)(\Gamma)=\int_{\V}\mathrm{P}_{t}(\y,\Gamma)\nu(\d\y),$ for $\Gamma\in\mathcal{B}(\H)$ and  $\mathrm{P}_t(\y,\cdot)$ is the transition probability, $\mathrm{P}_{t}(\y,\Gamma)=\mathrm{P}_{t}(\chi_{\Gamma})(\y),\ \y\in\H$.
\end{definition}

By the definition of random attractors, it is clear  that there exists an invariant compact random set in $\H$. A Feller invariant probability measure for a Markov RDS $\varphi$ on $\H$ is, by definition, an invariant probability measure for the semigroup $\{\mathrm{P}_t\}_{t\geq 0}$ defined by \eqref{71}. Hence, we have the following result on the existence of invariant measures for the SCBF equations \eqref{S-CBF} defined on Poincar\'e domains in $\H$.
\begin{theorem}\label{thm6.3}
	For all the cases given in Table \ref{Table1}, there exists an invariant measure for the SCBF equations \eqref{S-CBF} in $\H$.
\end{theorem}
\subsection{Uniqueness of invariant measures}
In this work, $\W(\cdot)$ is a Wiener process with RKHS $\mathrm{K}$ satisfying Assumptions \ref{assump1} (for $r\in[1,3)$) and \ref{assump2} (for $r\geq3$). In particular, $\mathrm{K} \subset\H$ and the natural embedding  $i : \mathrm{K}\hookrightarrow \H$ is a Hilbert-Schmidt operator. For a fixed orthonormal basis $\{w_k\}_{k\in\N}$ of $\mathrm{K}$ and a sequence $\{\beta_k\}_{k\in\N}$ of independent Brownian motions defined on some filtered probability space $(\Omega, \mathscr{F}, (\mathscr{F}_t)_{t\in \R}, \mathbb{P})$ such that $\W(\cdot)$ can be written in the following form
\begin{align}\label{Sum-W}
	\W(t)=\sum_{k=1}^{\infty}\beta_k(t) w_k,  \ \ \ t\in\mathbb{R}.
\end{align}
Moreover, there exists a covariance operator $\J \in \mathfrak{L}(\H)$ associated with $\W(\cdot)$ defined by 
\begin{align*}
	\left\langle \J h_1,h_2\right\rangle=\mathbb{E}\left[\left\langle h_1,\W(1)\right\rangle_{\H}\left\langle \W(1),h_2\right\rangle_{\H}\right], \ \ \ h_1,h_2\in \H. 
\end{align*}
It is well known from \cite{DZ1} that $\J$ is a non-negative self-adjoint and trace class operator in $\H$. Furthermore, $\J = ii^*$ and $\mathrm{K} = R(\J^{\frac{1}{2}} ),$ where $R(\J^{\frac{1}{2}} )$ is the range of the operator $\J^{\frac{1}{2}}$ (see \cite{BN1}). Note that 
\begin{align*}
	\sum_{k=1}^{\infty}\|iw_k\|^2_{\H}= \text{Tr}\left[\J\right]<\infty.
\end{align*}
For $d=2$ with $r\in[1,3]$, the uniqueness of invariant measures is proved in \cite[Theorem 5.5]{KM2}. Therefore we are not repeating here. We consider here $d=2,3$ with $r\geq3$ (for $r=3$ with $2\beta\mu\geq1$) only. 
\subsubsection{Exponential estimates}
Here, we obtain some exponential estimates which is used to obtain  the uniqueness of invariant measures.
\begin{theorem}\label{UIM1}
	For $d=2,3$ with $r\geq3$ (for $r=3$ with $2\beta\mu\geq1$), let $\u_1(\cdot)$ and $\u_2(\cdot)$ be two solutions of the system \eqref{S-CBF} with the initial data $\u_1^0,\u_2^0\in\H$, respectively. Then, we have
	\begin{align}\label{62}
		\mathbb{E}\left[\|\u_1(t)-\u_2(t)\|^2_{\H}\right]\leq\begin{cases}
			\|\u_1^0-\u_2^0\|^2_{\H}\ \exp[-(\mu\lambda_1+2\alpha-2\eta)t],&\text{	for  } r>3,\\
			\|\u_1^0-\u_2^0\|^2_{\H},\ \exp[-(\mu\lambda_1+2\alpha)t],&\text{	for  } r=3 \text{	with  }2\beta\mu\geq1,
		\end{cases}
	\end{align}
	provided $\mu\lambda_1+2\alpha>2\eta$ for $r>3$, where $\eta=\frac{r-3}{2\mu(r-1)}\left(\frac{2}{\beta\mu (r-1)}\right)^{\frac{2}{r-3}}$.
\end{theorem}
\begin{proof}
	Let $\mathfrak{X}(\cdot)=\u_1(\cdot)-\u_2(\cdot)$, then $\mathfrak{X}(\cdot)$ satisfies  the following equality:
	\begin{align}\label{UN1}
		\|\mathfrak{X}(t)\|^2_{\H}&=\|\mathfrak{X}(0)\|^2_{\H}-2\mu\int_{0}^{t}\|\mathfrak{X}(\zeta)\|^2_{\V}\d\zeta-2\alpha\int_{0}^{t}\|\mathfrak{X}(\zeta)\|^2_{\H}\d\zeta\nonumber\\&\quad-2\int_{0}^{t}\left\langle\B(\u_1(\zeta))-\B(\u_2(\zeta)),\mathfrak{X}(\zeta)\right\rangle\d\zeta-2\beta\int_{0}^{t}\left\langle\mathcal{C}(\u_1(\zeta))-\mathcal{C}(\u_2(\zeta)),\mathfrak{X}(\zeta)\right\rangle\d\zeta\nonumber\\&=\|\mathfrak{X}(0)\|^2_{\H}-2\mu\int_{0}^{t}\|\mathfrak{X}(\zeta)\|^2_{\V}\d\zeta-2\alpha\int_{0}^{t}\|\mathfrak{X}(\zeta)\|^2_{\H}\d\zeta\nonumber\\&\quad-2\int_{0}^{t}b(\mathfrak{X}(\zeta),\mathfrak{X}(\zeta),\u_1(\zeta))\d\zeta-2\beta\int_{0}^{t}\left\langle\mathcal{C}(\u_1(\zeta))-\mathcal{C}(\u_2(\zeta)),\mathfrak{X}(\zeta)\right\rangle\d\zeta,
	\end{align}
	for a.e. $t\in[0,T]$,	where we have used \eqref{2.1}, \eqref{b0}-\eqref{lady}, \eqref{MO_c}, H\"older's and Young's inequalities. From \eqref{MO_c}, we obtain
\begin{align}\label{UN2}
	-2\beta\left\langle\mathcal{C}(\u_1)-\mathcal{C}(\u_2),\mathfrak{X}\right\rangle\leq-\beta \||\u_1|^{\frac{r-1}{2}}\mathfrak{X}\|^2_{\H}-\beta \||\u_2|^{\frac{r-1}{2}}\mathfrak{X}\|^2_{\H}.
\end{align}
Using H\"older's and Young's inequalities, we get (cf. \cite[Theorem 2.2]{Mohan})
\begin{align}\label{UN3}
|b(\mathfrak{X},\mathfrak{X},\u_1)|\leq\begin{cases}
		\frac{\mu}{2}\|\mathfrak{X}\|^2_{\V}+\frac{\beta}{2}\||\u_1|^{\frac{r-1}{2}}\mathfrak{X}\|^2_{\H}+\eta\|\mathfrak{X}\|^2_{\H},  &\text{	for  } r>3,\\
		\frac{\mu}{2}\|\mathfrak{X}\|^2_{\V}+\frac{1}{2\mu}\||\u_1|^{\frac{r-1}{2}}\mathfrak{X}\|^2_{\H}, &\text{	for  } r=3.
	\end{cases}
\end{align}
 where, $\eta=\frac{r-3}{2\mu(r-1)}\left(\frac{2}{\beta\mu (r-1)}\right)^{\frac{2}{r-3}}$. Combining \eqref{UN1}-\eqref{UN3}, using \eqref{2.1} and taking expectation, we find 
 \begin{align}
 		\mathbb{E}[\|\mathfrak{X}(t)\|^2_{\H}]\leq\begin{cases}
 			\|\mathfrak{X}(0)\|^2_{\H}-\int_{0}^{t}\left[(\mu\lambda_1+2\alpha)-2\eta\right]\mathbb{E}[\|\mathfrak{X}(\zeta)\|^2_{\H}]\d\zeta,  &\text{	for  } r>3,\\
 			\|\mathfrak{X}(0)\|^2_{\H}-\int_{0}^{t}(\mu\lambda_1+2\alpha)\mathbb{E}[\|\mathfrak{X}(\zeta)\|^2_{\H}]\d\zeta, &\text{	for  } r=3 \text{	with  }2\beta\mu\geq1.
 		\end{cases}
 \end{align}
Applying Gronwall's inequality, we conclude 
	\begin{align}\label{63}
		\mathbb{E}[\|\mathfrak{X}(t)\|^2_{\H}]\leq\begin{cases}
			\|\mathfrak{X}(0)\|^2_{\H}\ \text{exp}[-(\mu\lambda_1+2\alpha-2\eta)t],&\text{	for  } r>3,\\
			\|\mathfrak{X}(0)\|^2_{\H}\ \text{exp}[-(\mu\lambda_1+2\alpha)t],&\text{	for  } r=3 \text{	with  }2\beta\mu\geq1,
		\end{cases}
	\end{align}
which completes the proof.
\end{proof}
\begin{theorem}\label{UIM2}
	For $d=2,3$ with $r\geq3$ (for $r=3$ with $2\beta\mu\geq1$), let the condition given in Theorem \ref{UIM1} be satisfied and $\u_0\in\H$ be given. Then, there is a unique invariant measure for the system \eqref{S-CBF}. Moreover, the invariant measure is ergodic and strongly mixing.
\end{theorem}
\begin{proof}
	See the proof of Theorem 5.5 in \cite{Mohan}.
\end{proof}

	\begin{remark}\label{RemarkD}
	For the SCBF equations \eqref{SCBF}, the results of this work, for all the cases given in Table \ref{Table1}, can be proved in general unbounded domains or on the whole space also. The presence of Darcy's coefficient $\alpha>0$ in \eqref{SCBF} helps us to get such results.    In that case, one has to take the norm defined on $\V$ space as $\|\u\|^2_{\V} := \|\u\|^2_{\H} + \|\nabla\u\|^2_{\H}$. Since the Stokes operator $\A$ is not invertible in general unbounded domains or on the whole space, one has to make changes in Assumptions \ref{assump1} and \ref{assump2} also. Instead of $\A^{-\delta}$, one needs to take $(1+\A)^{-\delta}$. Under the above change in Assumptions \ref{assump1} and \ref{assump2} (which help us to prove Proposition \ref{SOUP1} in general unbounded domains) and with some minor changes in the calculations, the results of this work hold true in general unbounded domains and on the whole space also. 
\end{remark}

	\medskip\noindent
	{\bf Acknowledgments:}    The first author would like to thank the Council of Scientific $\&$ Industrial Research (CSIR), India for financial assistance (File No. 09/143(0938)/2019-EMR-I).  M. T. Mohan would  like to thank the Department of Science and Technology (DST), Govt of India for Innovation in Science Pursuit for Inspired Research (INSPIRE) Faculty Award (IFA17-MA110).

		\medskip\noindent
	{\bf Data availability:} 
	Data sharing not applicable to this article as no datasets were generated or analysed during the current study.

	\medskip\noindent	{\bf Deceleration:} 	The author has no competing interests to declare that are relevant to the content of this article.


\begin{thebibliography}{99}
		
		\bibitem{Abergel}
		F. Abergel, Existence and finite dimensionality of the global attractors for evolution equations on unbounded domains, \emph{J. Differential Equations}, \textbf{83}(1) (1990), 85--108.
		
		
		
		\bibitem{AO}	S.N. Antontsev and H.B. de Oliveira, The Navier–Stokes problem modified by an absorption term, \emph{Appl. Anal.}, {\bf 89}(12)  (2010), 1805--1825.
		
		
		
		
		
	\bibitem{Arnold}	L. Arnold, \emph{Random Dynamical Systems}, Springer-Verlag, Berlin, Heidelberg, New York, 1998.
		
		
		
		\bibitem{Ball} J. M. Ball, Global attractors for damped semilinear wave equations, \emph{Discrete Contin. Dyn. Syst. Ser. B}, \textbf{10} (2004), 31--52.
		
		
			\bibitem{BLW} P. Bates, K. Lu and B. Wang, Random attractors for stochastic reaction-diffusion equations on unbounded domains, \emph{J. Differential Equations}, \textbf{246} (2009), 845--869.
		
		
		
		
		
		
		
		\bibitem{BGT}	Z. Brze\'zniak, B. Goldys and Q. T. Le Gia, Random attractors for the stochastic Navier-Stokes equations on the 2D unit sphere, \emph{J. Math. Fluid Mech.}, \textbf{20} (2018), 227--253.
		
		
		
		\bibitem{BM} Z. Brze\'zniak and  E. Motyl, Existence of a martingale solution of the stochastic Navier-Stokes equations in unbounded 2D and 3D domains, \emph{J. Differential Equations}, \textbf{254}(4) (2013), 1627--1685.
		
		\bibitem{BMO}  Z. Brze\'zniak, E. Motyl and M. Ondrejat, Invariant measure for the stochastic Navier-Stokes equations in unbounded 2D domains, \emph{Ann. Probab.}, {\bf 45}(5) (2017),  3145--3201.
		
		 
		 
		\bibitem{BLL} P. Bates, H. Lisei and K. Lu, Attractors for stochastic lattice dynamical systems, \emph{Stoch. Dyn.}, \textbf{6}(1) (2006), 1--21.
		
		
		
		
	
		
		
		\bibitem{Brze} Z. Brze\'zniak, On Sobolev and Besov spaces regularity of Brownian paths, \emph{Stoch. Stoch. Rep.}, \textbf{56}(1–2) (1996), 1--15.
		
		
		
		\bibitem{Brze1} Z. Brze\'zniak, Stochastic convolution in Banach spaces, \emph{Stoch. Stoch. Rep.}, \textbf{61} (1997), 245--295.
		
		
		\bibitem{Brze2} Z. Brze\'zniak, Stochastic partial differential equations in M-type 2 Banach spaces, \emph{Potential Anal.}, \textbf{4} (1995), 1--45.
		
		
		\bibitem{BCF}  Z. Brze\'zniak, M. Capi\'nski and F. Flandoli, Pathwise global attractors for stationary random dynamical systems, \emph{Probab. Theory Related Fields}, \textbf{95} (1993), 87--102.
		
		
		
		\bibitem{BCLLLR} Z. Brz\'ezniak, T. Caraballo, J. A. Langa, Y. Li, G. Lukaszewicz and J. Real, Random attractors for stochastic 2D Navier-Stokes equations in some unbounded domains, \emph{J. Differential Equations}, \textbf{255} (2013), 3897--3919.
		
		
		
		
		
		
		\bibitem{BH} 	Z. Brz\'ezniak and H. Long, A note on $\gamma$-radonifying and summing operators, \emph{Stochastic Analysis, Banach center for publications,} \textbf{105} (2015), 43--57.
	
		\bibitem{BL}  Z. Brz\'ezniak and Y. Li, Asymptotic compactness and absorbing sets for 2D stochastic Navier-Stokes equations in some unbounded domains, \emph{Trans. Amer. Math. Soc.}, \textbf{358}(12) (2006) 5587--5629.
		
		
		\bibitem{BP} Z. Brz\'ezniak and S. Peszat, Stochastic two dimensional Euler equations, \emph{Ann. Probab.}, {\bf 29}(4) (2001), 1796--1832.
		
		
		\bibitem{BN} Z. Brzeźniak and J. van Neerven, Space-time regularity for linear stochastic evolution equations driven by spatially homogeneous noise, \emph{J. Math. Kyoto Univ.}, \textbf{43}(2) (2003), 261--303. 
		
		\bibitem{BN1}    Z. Brze\'zniak, J. van Neerven, Stochastic convolution in separable Banach spaces and the stochastic linear Cauchy problem, \emph{Studia Math.}, \textbf{143} (2000), 43–74.
		
		
		
		
		\bibitem{CLR1}
		T. Caraballo, G. Lukaszewicz and J. Real, Pullback attractors for asymptotically compact non-autonomous dynamical systems, \emph{Nonlinear Analysis}, \textbf{64}(3) (2006), 484--498.
		
		
		
		\bibitem{CLR2}
		T. Caraballo, G. Lukaszewicz and J. Real, Pullback attractors for non-autonomous 2D-Navier-Stokes equations in some unbounded domains, \emph{C. R. Math. Acad. Sci. Paris}, \textbf{342}(4) (2006), 263--268.
		
		
		
		
		\bibitem{CV} V. V. Chepyzhov and M. I. Vishik, \emph{Attractors for Equations of Mathematical Physics}, American Mathematical Society, Providence, Rhode Island, 2002.
		
		
		
		\bibitem{PCAM} P. Cherier and  A. Milani,  \emph{Linear and Quasi-linear Evolution Equations	in Hilbert Spaces}, American Mathematical Society Providence,Rhode Island, 2012. 
		
		
		
		\bibitem{PGC} 	P. G. Ciarlet, \emph{Linear and Nonlinear Functional Analysis with Applications}, SIAM Philadelphia, 2013.
		
		
		
		\bibitem{CF} H. Crauel and F. Flandoli, Attractors for random dynamical systems, \emph{Probab. Theory Related Fields}, \textbf{100} (1994), 365--393.
		
		
		
		
		
		
		\bibitem{Crauel}	H. Crauel, Random Probability Measures on Polish Spaces, \emph{Stochastics Monographs}, vol. 11, Taylor \& Francis, London, 2002.
		
		
		
		\bibitem{CDF}	H. Crauel, A. Debussche and F. Flandoli, Random attractors, \emph{J. Dynam. Differential Equations}, \textbf{9}(2) (1995), 307--341.
		
		
		
		\bibitem{DZ}	G. Da Prato and J. Zabczyk, \emph{Ergodicity for Infinite Dimensional Systems}, London Mathematical Society Lecture Note Series, vol. 229, Cambridge University Press, Cambridge, 1996.
		
		
		
			\bibitem{DZ1}	 G. Da Prato and J. Zabczyk, \emph{Stochastic Equations in Infinite Dimensions}, 2nd edition, Cambridge Univ. Press, Cambridge, 2014.
		
	
		
		\bibitem{LCE}	L. C. Evans, \emph{Partial Differential Equations}, Grad. Stud. Math., vol. 19, Second Edition, Amer. Math. Soc., Providence, RI, 2010.
		
		
		\bibitem{FHR} 	C. L. Fefferman, K. W. Hajduk and J. C. Robinson, Simultaneous approximation in Lebesgue and Sobolev norms via eigenspaces, \emph{Proc. London Math. Soc.}, \textbf{3} (2022), 1-19. 
		
		
		
		
		
		
		\bibitem{FY}	X. Feng and B. You, Random attractors for the two-dimensional stochastic g-Navier-Stokes equations, \emph{Stochastics}, \textbf{92}(4) (2020), 613--626.
		
		
		
		\bibitem{FMRT}	\newblock C. Foias, O. Manley, R. Rosa and R. Temam, \newblock \emph{Navier-Stokes Equations and Turbulence}, \newblock Cambridge University Press, 2008.
		
		
		
		
		
		
		\bibitem{GPG} 	G. P.  Galdi, An introduction to the Navier–Stokes initial-boundary value problem. In \emph{Fundamental directions in mathematical fluid mechanics}, Adv. Math. Fluid Mech. Birkh\"auser, Basel, 2000, pp. 1--70.
		
		
		
		
		\bibitem{GLS}	B. Gess, W. Liu and A. Schenke, Random attractors for locally monotone stochastic partial differential equations, \emph{J. Differential Equations}, \textbf{269} (2020), 3414--3455.
		
		
		
		
		\bibitem{Ghidaglia}		J. M. Ghidaglia, A note on the strong convergence towards attractors of damped forced KdV equations, \emph{J. Differential Equations}, \textbf{110}(2) (1994), 356--359.
		
		
		
		\bibitem{HR}	K. W. Hajduk and J. C. Robinson, Energy equality for the 3D critical convective Brinkman-Forchheimer equations, \emph{J. Differential Equations}, {\bf 263} (2017), 7141--7161.
		
		
		
		\bibitem{HR1}	K. W. Hajduk, J. C. Robinson and W. Sadowski,	Robustness of regularity for the 3D convective Brinkman--Forchheimer equations, \emph{J. Math. Anal. Appl.}, \textbf{500}(1) (2021), 125058.
		
		
		\bibitem{HZ} Z. Han and S. Zhou, Random exponential attractor for the 3D non-autonomous stochastic damped Navier-Stokes equation, \emph{J. Dynam. Differential Equations}, (2021), \url{https://doi.org/10.1007/s10884-021-09951-x}. 
		
		
		
		\bibitem{Heywood} J. G. Heywood, The Navier-Stokes Equations: on the existence, regularity and decay of solutions, \emph{Ind. Univ. Math. J.}, \textbf{29} (1980), 639--681.
		
		
		
		\bibitem{HW} K. Holly and M. Wiciak, Compactness method applied to an abstract nonlinear parabolic equation, in: Selected Problems of Mathematics, Cracow University of Technology, 1995, pp. 95--160.
		
		
		
		
		
		
		\bibitem{KZ} 	V. K. Kalantarov and S. Zelik, Smooth attractors for the Brinkman-Forchheimer equations with fast growing nonlinearities, \emph{Commun. Pure Appl. Anal.}, {\bf 11}	(2012), 2037--2054.
		
		
		\bibitem{KM1} K. Kinra and M. T. Mohan, Existence and upper semicontinuity of random attractors for the 2D stochastic convective Brinkman-Forchheimer equations in bounded domains, Accepted in \emph{Stochastics}, (2022), \url{https://arxiv.org/pdf/2011.06206.pdf}.
		 
		 
		 \bibitem{KM2} K. Kinra and M. T. Mohan, $\H^1$-Random attractors for 2D stochastic convective Brinkman-Forchheimer equations in unbounded domains, Accepted in \emph{Adv. Differential Equations}, (2022), \url{https://arxiv.org/pdf/2111.07841.pdf}.  
		 
		 
		
		\bibitem{KM3} K. Kinra and M. T. Mohan, Large time behavior of the deterministic and stochastic 3D convective Brinkman-Forchheimer equations in periodic domains, \emph{J. Dynam. Differential Equations}, (2021), pp. 1--42.
		
		
		\bibitem{KM6} K. Kinra and M. T. Mohan, Existence and upper semicontinuity of random pullback attractors for 2D and 3D non-autonomous stochastic convective Brinkman-Forchheimer equations on whole space, Accepted in \emph{Differential Integral Equations}, (2022), \url{https://arxiv.org/pdf/2105.13770.pdf}. 
		
		
		
		\bibitem{KM7} K. Kinra and M. T. Mohan, Long term behavior of 2D and 3D non-autonomous random convective Brinkman-Forchheimer equations driven by colored noise, \emph{Submitted}, \url{https://arxiv.org/pdf/2107.08890.pdf}.
		
	\bibitem{KM8}	K. Kinra and M. T. Mohan, Bi-spatial random attractor, ergodicity and a random Liouville type theorem for stochastic Navier-Stokes equations on the whole space, \emph{Submitted}, \url{https://arxiv.org/pdf/2209.08915.pdf}.
		
		
		\bibitem{PEK}	P. E. Kloeden and M.  Rasmussen, \emph{Nonautonomous dynamical systems}, Mathematical Surveys and Monographs, 176, American Mathematical Society, Providence, RI, 2011. 
		
		\bibitem{OAL}	O. A. Ladyzhenskaya, \emph{The Mathematical Theory of Viscous Incompressible Flow}, Gordon and Breach, New York, 1969.
		
		
		\bibitem{LG}
		H. Liu and H. Gao, Ergodicity and dynamics for the stochastic 3D Navier-Stokes equations with damping, \emph{Commun. Math. Sci.}, \textbf{16}(1) (2018), 97--122.
		
		
		\bibitem{LL} F. Li and Y. Li, Asymptotic behavior of stochastic g-Navier-Stokes equations on a sequence of expanding domains, \emph{J. Math. Phys.}, \textbf{60} (2019), 061505.
		
		
		\bibitem{LYZ} Y. Li, S. Yang and Q. Zhang, Odd random attractors for stochastic non-autonomous Kuramoto-Sivashinsky equations without dissipation, \emph{Electron. Res. Arch.}, 28 (2020), 1529--1544.
		\bibitem{LXS}	F. Li, D. Xu and L. She, Large-domain stability of random attractors for stochastic g-Navier–Stokes equations with additive noise, \emph{J. Inequal. Appl.}, \textbf{193} (2020), Paper No. 193, 24 pp. 
		\bibitem{Mohan1}	 M. T. Mohan, On the convective Brinkman-Forchheimer equations, \emph{Submitted}.
		\bibitem{Mohan}	 M. T. Mohan, Stochastic convective Brinkman-Forchheimer equations, \emph{Submitted}, \url{https://arxiv.org/abs/2007.09376}.
		\bibitem{Mohan2}	M. T. Mohan, Asymptotic analysis of the 2D convective Brinkman-Forchheimer equations in unbounded domains: Global attractors and upper semicontinuity, \emph{Submitted}, \url{https://arxiv.org/abs/2010.12814}. 
		\bibitem{MTM2} M. T. Mohan, The $\H^1$-compact global attractor for the two dimensional convective Brinkman-Forchheimer equations in unbounded domains, \emph{J. Dyn. Control Syst.}, \textbf{28} (2021), 791–816.
		\bibitem{Ondrejat}		M. Ondrej\'at,  Uniqueness for stochastic evolution equations in Banach spaces, \emph{Dissertationes Mathematics (Rozprawy Mat.)}, \textbf{426} (2004), 63pp.
		
		
		
		\bibitem{Robinson2} J. C. Robinson, \emph{Infinite-Dimensional Dynamical Systems, An Introduction to Dissipative Parabolic PDEs and the Theory of Global Attractors}, Cambridge Texts in Applied Mathematics, 2001.
		
		
		
		\bibitem{Rosa}		R. Rosa, The global attractor for the 2D Navier-Stokes flow on some unbounded domains, \emph{Nonlinear Analysis}, \textbf{32} (1998), 71--85.
		
		
		
		\bibitem{Slavik}	J. Slav\'ik, Attractors for stochastic reaction-diffusion equation with additive homogeneous noise, \emph{Czechoslovak Math. J.}, \textbf{71}(146) (2021), 21--43. 
		
		
		
		\bibitem{SLHZ}	J. Shu, H. Li, X. Huang and J. Zhang, Asymptotic behaviour of stochastic heat equations in materials with memory on thin domains, \emph{Dyn. Syst.}, \textbf{35} (2020), 704–728.
		
		\bibitem{R.Temam}	R. Temam, \emph{Infinite-Dimensional Dynamical Systems in Mechanics and Physics,}  \textbf{68}, Applied Mathematical Sciences, Springer, 1988.
		
		
		\bibitem{Temam}	R. Temam, \emph{Navier-Stokes Equations, Theory and Numerical Analysis}, North-Holland, Amsterdam, 1977.
		
		
		\bibitem{Wang} B. Wang, Random attractors for the stochastic Benjamin–Bona–Mahony equation on unbounded domains, \emph{J. Differential Equations}, \textbf{246}(6) (2008), 2506--2537.
		
		
		\bibitem{You}		B. You, The existence of a random attractor for the three dimensional damped Navier-Stokes equations with additive noise, \emph{Stoch. Anal. Appl.}, \textbf{35}(4) (2017), 691--700.
		
		
	\end{thebibliography}
\end{document}